\documentclass{amsart}

 \usepackage[fontsize=10pt]{fontsize}  
 \usepackage[a4paper, left=1.5cm, right=1.4cm, top=0.5cm,bottom=0.7cm]{geometry}  %margin for arXiv label

\usepackage{amsmath,amsthm} %thms
\usepackage{latexsym,amssymb}  %symbols
\usepackage{amsfonts,mathrsfs} %fonts
\usepackage{amsopn}

\usepackage{stmaryrd} %longmapsto?? no use?

\usepackage{color}
\usepackage{graphicx}
\usepackage{epstopdf} %%2024.2: to include graphics??
\usepackage{indentfirst} %%2024.2: to change indentation? did I ever use?
\usepackage{url} 
\usepackage[linktocpage=true, hidelinks, colorlinks, linkcolor=blue, citecolor=magenta, bookmarksopen=true, urlcolor=brown]{hyperref}
\usepackage{enumitem}
 \usepackage[normalem]{ulem} %to strike out sentence, use command ``sout".
 \usepackage{verbatim, comment} %for comment  
 \usepackage{svn, mathtools, mathscinet} %2024.2记不清了
 \usepackage{relsize} %字体无穷缩小。。in gray stuff. 
\usepackage{lmodern} % Better font scaling 
\usepackage{setspace} % For line spacing control 
 \newcommand{\shorturl}[1]{\href{#1}{link}}  %缩短url

\usepackage[all]{xy} 
\usepackage{tikz-cd, amscd}
 \usepackage{array}
 \usepackage{tabu} %for math mode of tables 
 
 \usepackage{tikz}
 \usetikzlibrary{decorations.markings}
\tikzset{
    closed/.style={
        decoration={
            markings,
            mark=at position 0.5 with {\node[transform shape, xscale=.8, yscale=.4] {/};}
        },
        postaction={decorate}
    }
}  % closed embedding  X \arrow[r, hookrightarrow, closed] & Y  

\tikzset{open/.style = {decoration = {markings, mark = at position 0.5 with { \node[transform shape, scale = .7] {$\circ$}; } }, postaction = {decorate} }
}   % open embedding   X \arrow[r, hookrightarrow, open] & Y

 \label{Theorems environ}
 \newtheorem{thm}{Theorem}[section]
\newtheorem{theorem}[thm]{Theorem}

\newtheorem{prop}[thm]{Proposition}
\newtheorem{proposition}[thm]{Proposition}
\newtheorem{lem}[thm]{Lemma}
\newtheorem{lemma}[thm]{Lemma}

\newtheorem{cor}[thm]{Corollary}

\newtheorem{lem-def}[thm]{Lemma-Definition}

\theoremstyle{definition}
\newtheorem{defn}[thm]{Definition}

\newtheorem{rem}[thm]{Remark}
\newtheorem{remark}[thm]{Remark}

 \newtheorem{construction}[thm]{Construction}

\newtheorem{example}[thm]{Example}

\newtheorem{notation}[thm]{Notation}

\newtheorem{convention}[thm]{Convention}

\numberwithin{equation}{section}
\label{••••••••••3}
\label{From Min-Wang}

\newcommand{\Zp}{{\bZ_p}}

\newcommand{\Qp}{{\bQ_p}}

\newcommand{\End}{{\mathrm{End}}}           %End-functor
\DeclareMathOperator{\Fil}{{\mathrm{Fil}}}           %Filtration
\newcommand{\Gal}{{\mathrm{Gal}}}           %Galois groups
\newcommand{\Hom}{{\mathrm{Hom}}}           %Hom-functor
\DeclareMathOperator{\Ker}{{\mathrm{Ker}}}           %kernal
\DeclareMathOperator{\Vect}{{\mathrm{Vect}}}          % Vector bundle

\newcommand{\GL}{{\mathrm{GL}}}             % genetal linear group
\newcommand{\dR}{{\mathrm{dR}}}             % de Rham
\newcommand{\st}{{\mathrm{st}}}             % semi stable
\newcommand{\tor}{{\mathrm{tor}}}           % torsion
\label{••••••••••5}

\label{site: prismatic}

\usepackage{relsize}
\usepackage[bbgreekl]{mathbbol}
\DeclareSymbolFontAlphabet{\mathbb}{AMSb} %to ensure that the meaning of \mathbb does not change
\DeclareSymbolFontAlphabet{\mathbbl}{bbold}
\newcommand{\Prism}{{\mathlarger{\mathbbl{\Delta}}}} % Prism
\newcommand{\prism}{{\mathlarger{\mathbbl{\Delta}}}} % Prism
 \newcommand{\pris}{{\mathlarger{\mathbbl{\Delta}}}} % Prism
\newcommand{\opris}{{\mathcal{O}_\prism}}

\newcommand{\baropris}{{\overline{\O}_\prism}}

\label{to move around}

\label{Topology, THH, TC, TP}

 \label{colors}

 \newcommand{\ghblue}[1]{{\color{blue}#1}}

\label{arrows}
\newcommand \into {\hookrightarrow }
\renewcommand \to {\rightarrow}
\newcommand \onto {\twoheadrightarrow}

\label{•••••••••4}
 \label{matrix, groups}
 \DeclareMathOperator{\diag}{\mathrm{diag}}

\DeclareMathOperator{\vect}{\mathrm{Vect}}

\label{module, homological alg}
 
\newcommand{\rep}{{\mathrm{Rep}}}

\def\inf{{\mathrm{inf}}}

\renewcommand{\max}{{\mathrm{max}}}

\newcommand{\gal}{{\mathrm{Gal}}}

\DeclareMathOperator{\fil}{{\mathrm{Fil}}}
\newcommand{\gr}{{\mathrm{gr}}}
 
  \newcommand{\rmconj}{{\mathrm{conj}}}
  
\DeclareMathOperator{\spf}{{\mathrm{Spf}}}

\DeclareMathOperator{\Lie}{\mathrm{Lie}}

\newcommand{\syn}{{\mathrm{syn}}}

\label{condesend, solid math}

\label{!•••••••••}
\label{ALL: p-adic Hodge}

\label{LAV}
\newcommand{\dla}{\text{$\mbox{-}\mathrm{la}$}}

\newcommand{\dpa}{\text{$\mbox{-}\mathrm{pa}$}}
\renewcommand{\log}{\mathrm{log}}

\newcommand{\kpinfty}{{K_{p^\infty}}}

\newcommand{\kinfty}{{K_{\infty}}}

\label{LAV: Lie gp Lie alg}

\newcommand{\gammak}{{\Gamma_K}}

\newcommand{\gk}{{G_K}}

\label{period ring: Fontaine}
\newcommand\crys{{\mathrm{crys}}}

\newcommand\HT{{\mathrm{HT}}}

\newcommand\Sen{{\mathrm{Sen}}}

\newcommand\rig{{\mathrm{rig}}}

\newcommand{\ainf}{{\mathbf{A}_{\mathrm{inf}}}}
 
\newcommand{\acris}{{\mathbf{A}_{\mathrm{cris}}}}

\newcommand{\bcris}{{\mathbf{B}_{\mathrm{cris}}}}
\newcommand{\bcrisplus}{{\mathbf{B}^+_{\mathrm{cris}}}}

\newcommand{\bdrplus}{{\mathbf{B}^+_{\mathrm{dR}}}}
\newcommand{\bdr}{{\mathbf{B}_{\mathrm{dR}}}}

\label{period ring: relative}

\label{period ring: Gao-Min-Wang}
 
\newcommand{\nht}{{\mathrm{nHT}}}

\label{period ring: bb font}
\newcommand*{\wt}[1]{\widetilde{#1}}
\newcommand*{\wh}[1]{\widehat{#1}}

\label{period ring: AB ring}
\newcommand{\wtb}{   {\widetilde{{\mathbf{B}}}}  }

\label{period ring: integral Hodge}
\def \ok {{\mathcal{O}_K}}
\def \ko {{K_{0}}}

 \newcommand{\cflat}{{C^\flat}}
 
\newcommand{\ocflat}{{\mathcal{O}_C^\flat}}

 \newcommand{\zetaflat}{{\zeta^\flat}}
   \newcommand{\piflat}{{\pi^\flat}}

\label{peirod module: D}

\label{imperf residue}

\label{Cats: Fontaine}

 \newcommand{\MF}{\mathrm{MF}^{\varphi, N}_{K_0}}
\newcommand{\MFwa}{\mathrm{MF}^{\varphi, N, \mathrm{wa}}_{K_0}}
\newcommand{\bigMF}{\mathrm{MF}^{\varphi, N}_{S_{K_0}}}
\newcommand{\bigMFwa}{\mathrm{MF}^{\varphi, N, \mathrm{wa}}_{S_{K_0}}}

\label{Cats: Rep and Higgs}

\label{••••••••1}
\label{ALL: site, sheaves}
\label{site: etale proetale}

\label{syntomic coho}

\label{font: roman rm}

 \newcommand{\rmh}{{\mathrm{H}}}

\label{font: mathbb}

\newcommand{\bQ}{{\mathbb Q}}

\newcommand{\bZ}{{\mathbb Z}}

  \newcommand{\bm}{\mathbb{M}}

  \newcommand{\bbz}{{\mathbb{Z}}}

\newcommand{\zp}{{\mathbb{Z}_p}}
\newcommand{\qp}{{\mathbb{Q}_p}}
\newcommand{\fp}{{\mathbb{F}_p}}

\label{font: mathcal}
\newcommand{\calI}{{\mathcal I}}

\newcommand{\calO}{{\mathcal O}}

   \renewcommand{\O}{{\mathcal{O}}} 
 \renewcommand{\o}{{{\mathcal{O}}}}

  \newcommand{\cd}{{\mathcal{D}}}
    \newcommand{\cald}{{\mathcal{D}}}

\newcommand{\cale}{{\mathcal{E}}}

\newcommand{\cm}{{\mathcal{M}} }

\newcommand{\caln}{{\mathcal{N}}}

\label{font: mathfrak}
\newcommand{\gs}{{\mathfrak{S}}}

\newcommand{\gm}{{\mathfrak{M}}}
  \newcommand{\gmast}{{\gm^\ast}}

\newcommand{\M}{{\mathfrak{M}}}

\newcommand{\gn}{{\mathfrak{N}}}

\newcommand{\minf}{{\mathfrak{M}_{\mathrm{inf}}}}

\newcommand{\fkc}{{\mathfrak{c}}}

\newcommand{\fkm}{{\mathfrak{m}}}

\newcommand{\fkt}{{\mathfrak{t}}}

 \label{font: hat}
\label{font: bar}
\newcommand{\barK}{{\overline{K}}}

\label{font: bold}
 \label{font: mathbf}
 \newcommand{\bfa}{\mathbf{A}}
\newcommand{\bfb}{\mathbf{B}}

\label{••••••••2}

\newcommand{\amax}{\mathbf{A}_{\mathrm{max}}}

 \newcommand{\gsmax}{{\gs_{\mathrm{max}}}}

\label{stack proj}

\label{words}

\label{Gao added to dR paper}

\newcommand{\cmast}{{\mathcal{M}^{\ast}}}

\newcommand{\bargm}{{\overline{\gm}}}
 \newcommand{\bargmast}{{\overline{\gm}^\ast}}
\renewcommand{\phi}{\varphi}
 \newcommand{\fon}{{\mathrm{Fon}}}
 \newcommand{\gmht}{{\gm_\HT}}
  \newcommand{\bargmht}{{\bargm_\HT}}
  
  \newcommand{\gminf}{\gm_{\mathrm{inf}}}
   \newcommand{\bargminf}{\overline{\gm}_{\mathrm{inf}}}
 
 \newcommand{\hattheta}{{\wh{\Theta}}}

\author[]{Hui Gao}   \address{Department of Mathematics and Shenzhen International Center for Mathematics, Southern University of Science and Technology, Shenzhen 518055, China}   \email{gaoh@sustech.edu.cn}
 
 \author[]{Tong Liu} 
\address{Department of Mathematics, Purdue University, West Lafayette, IN 47907}
\email{tongliu@math.purdue.edu}

\begin{document}


\title[]{Integral filtered Sen theory and applications}

 \begin{abstract}  \normalsize{ 
We study Nygaard-, conjugate-, and Hodge filtrations  on  the many variants of   Breuil--Kisin modules associated to integral  semi-stable Galois representations. 
This leads to an  integral    Sen operator  satisfying certain  ``$1$-degree shrinking"  on  the increasing conjugate filtration, and (in special cases) a   mod $p$ Sen operator satisfying certain ``$p$-degree shrinking". 
These constructions are related with  prismatic $F$-crystals, Hodge--Tate crystals and $F$-gauges, and have explicit relations with classical (non-prismatic) operators. 
As applications, we obtain vanishing and torsion bound results on graded of the integral  Hodge filtration; our explicit methods also recover results of Gee--Kisin and Bhatt--Gee--Kisin concerning  the mod $p$  Hodge filtrations and Frobenius structures. }
    \end{abstract}

  \date{\today} 
\maketitle
\tableofcontents

%\subsection*{new macros}

%%\newpage
\section{Introduction}

This paper studies Sen theory and its applications in the context of \emph{integral} $p$-adic Hodge theory. 
Classically, Sen theory \cite{Sen81} studies $C$-representations, whereas integral $p$-adic Hodge theory studies $\zp$-representations (including torsion ones); thus  previously these two subjects have only minimal intersections, since  passing from $\zp$- to $C$-coefficients  obviously  lose a lot of information (e.g., the Frobenius operator). 
Our work is inspired  by the many recent studies in prismatic cohomology \cite{BS22}---more specifically---its coefficient objects including $F$-crystals \cite{BS23, DL23}, Hodge--Tate crystals (\cite{BL1, AHLB1, GMWHT} etc.) and the most recent $F$-gauges \cite{Bha22}. 
In particular, we are inspired by a theorem of Gee--Kisin \cite{GK-ias} on reduction of crystalline representations which makes use of $F$-gauges.

In this paper, we show that many classical \emph{differential/monodromy} operators (\cite{Bre97, Kis06, Liu08, Gao23} etc.) from  integral $p$-adic Hodge theory are indeed closely related with Sen theory: in fact, their (Hodge--Tate) reductions (modulo a certain ideal) are precisely (refinements of) the Sen operator. 
More importantly, because the classical monodromy operators are intertwined with Frobenius operators---using which we can define Nygaard-, Hodge-, and conjugate filtrations---, we show that these reduced Sen operators have rich interplay with these filtrations. 
In some sense, these filtrations (in particular the conjugate filtration) can be regarded as   \emph{shadows} of the Frobenius operator. 
As applications of these structural results on \emph{filtered Sen theory}, we in turn  obtain structural results on the Frobenius operator, which in particular recover results of Gee--Kisin \cite{GK-ias} and Bhatt--Gee--Kisin \cite{BGK}. 

Compared with the stacky approach of \cite{Bha22, GK-ias, BGK}, our explicit approach has the extra benefit of direct relation with  classical operators: indeed, all our results on filtered Sen theory admit \emph{lifts} to variants of Breuil--Kisin modules; most interestingly, in the mod $p$ setting, we construct a certain (lifted) ``stabilized truncated  operator" where one can \emph{unify} a ``$p$-Griffiths transversality" property and a ``$p$-degree shrinking" property for the mod $p$ Sen operator.

In the following \S \ref{subsec:intro-filsen}, we first discuss our results on  filtered Sen theory; we then discuss their applications in \S \ref{subsec:intro-app}. 
We should already point out that many results have overlaps and strong connections with results by Bhatt--Lurie \cite{Bha22}, Gee--Kisin \cite{GK-ias} and Bhatt--Gee--Kisin \cite{BGK}: see in particular Remarks \ref{rem:intro-filSen},  \ref{rem history intro} and \ref{rem y matrix} for detailed comments and comparisons. 
 We   emphasize that the original observation that
$F$-gauges can be used to prove such results is first due to Gee and Kisin.
  We hope these remarks make  it clear our intellectual debt to the work of Bhatt--Lurie and Gee--Kisin.

\subsection{Filtered Sen theory}\label{subsec:intro-filsen}

\begin{notation}
We introduce some notations that are necessary for our discussions.
\begin{enumerate}
\item  Let $k$ be a perfect  field of characteristic $p$, let $W(k)$ be the ring of Witt vectors, and let $K_0 :=W(k)[1/p]$.
Let $K$ be a finite totally ramified  extension of $K_0$, let $\mathcal O_K$ be the ring of integers. Fix an algebraic closure $\overline {K}$ of $K$ and set $G_K:=\Gal(\overline{K}/K)$.
Let $\pi \in K$ be a fixed uniformizer, and let $E(u)$ be its minimal polynomial over $K_0$; one can use these to define the Breuil--Kisin prism $(\gs=W(k)[[u]], (E))$.

\item Let $T$ be an integral semi-stable representation of $\gk$ with  Hodge--Tate weights (of $T[1/p]$) being $0 \leq r_1 \leq \cdots \leq r_d$. One can associate   an (effective) Breuil--Kisin module, which is a finite free $\gs$-module $\gm$ equipped with $\phi: \gm \to \gm$ such that the linearization \[1\otimes \phi: \gs[1/E]\otimes_{\phi, \gs}\gm \to \gs[1/E]\otimes_{ \gs}\gm \] is an isomorphism.  Define the following (effective) $\bbz$-filtrations:
\begin{enumerate}
\item  Regard $\gm^\ast =\gs\otimes_{\phi, \gs} \gm$ as a submodule of $\gm$ via $1\otimes \phi$, and define the  (decreasing) Nygaard filtration
\[ \fil^n \gmast :=\gmast \cap E^n \gm \]
\item Induce the (decreasing) Hodge filtration on $\gm_\dR:=\gmast/E\gmast$   via  the surjection $\gmast \onto \gm_\dR$. 
\item Define the  conjugate filtration on $\gm_\HT:=\gm/E\gm$  where $\fil_n \gmht$ (or  $\fil_n^\rmconj \gmht$ for emphasis) is the image of 
\[ \fil^n \gm^\ast/\fil^{n+1} \gm^\ast \xhookrightarrow{E^{-n}} \gm/E\gm\] 
It is a standard fact (cf. Lemma \ref{lem: matching graded dR and HT}) that this is an increasing filtration (whence the subscript notation $\fil_\bullet$) and it has matching gradeds  with that of Hodge filtration:
\[ \fil^n \gm_\dR/ \fil^{n+1} \gm_\dR  \simeq \fil_n \gm_\HT/ \fil_{n-1} \gm_\HT\]
\end{enumerate}

\item Let $\bargm=\gm/p\gm$; similarly define $\bargm^\ast$ and its Nygaard filtration, then use it to induce $\fil^\bullet \bargm_\dR$ and   $\fil_\bullet \bargmht$. We warn that the natural filtered reduction map $\fil^\bullet \gmast \to \fil^\bullet \bargmast$ (and other filtered reduction maps) is in general not strict; thus these mod $p$ filtrations could behave very differently.
\end{enumerate}
 \end{notation}

Let $\o$ be the ring of holomorphic functions on the open unit disk (defined over $K_0$). Let $\cm=\gm\otimes_\gs \o$.
 In \cite{Kis06}, Kisin constructs a differential operator
\[ N_\nabla: \cm \to  \cm. \]
Take mod $E$ reduction, (and note $\o/E=K$), we obtain
\[  \overline{N}_\nabla: \gm_\HT[1/p] \to  \gm_\HT[1/p]. \] 
We start with a \emph{filtered} refinement of the above operator.

\begin{theorem}[cf. Thm. \ref{thm: rational sen shift}]\label{thm: intro rational sen} 
(Let $T$ be a semi-stable representation).  
There is a constant $\fkc \in K$ (with explicit expression $\fkc=(\theta_\fon(u\lambda'))^{-1}$, cf \S \ref{sec: fil sen}), such that the scaled operator 
$$\theta_\kinfty= \fkc \overline{N}_\nabla: \gm_\HT[1/p] \to \gm_\HT[1/p],$$
---which we call the \emph{negative $\kinfty$-Sen operator} (cf. Remark \ref{rem: negative Sen op})--- satisfies the following:
\begin{enumerate}
\item $\theta_\kinfty$ is semi-simple with eigenvalues $r_1, \cdots, r_d$;  

\item For each $n$, the $n$-th shifted operator $\theta_\kinfty-n$  satisfies  \emph{$1$-degree shrinking}  
on $\fil_n$ in the sense that:
\[ (\theta_\kinfty-n)\left(\fil_n^\rmconj  \gm_\HT[1/p] \right) \subset \fil_{n-1}^\rmconj  \gm_\HT[1/p] \]
 \end{enumerate} 
\end{theorem}

When $K$ is unramified and $T$ is crystalline, the structures in Theorem \ref{thm: intro rational sen} (indeed, also its integral version Theorem \ref{thm: intro integral sen}) are first known by the work of Bhatt--Lurie \cite{Bha22}, via a stacky approach; note however the relation with Kisin's $N_\nabla$-operator (i.e., the \emph{lift} of Sen operator) above   genuinely needs extra input (at least from \cite{GMWHT}). 
See Remark \ref{rem:intro-filSen} for detailed comments on Bhatt--Lurie's approach; here, let us first discuss our approach.

\begin{proof}[Idea of proof for Theorem \ref{thm: intro rational sen}. ]
The operator $\theta_\kinfty$ is already constructed in \cite{GMWHT} (indeed, even for any $C$-representation), and hence   Item (1) of Theorem \ref{thm: intro rational sen} is a direct  consequence.
For the ``$1$-degree shrinking" in Item (2) (cf. Remark \ref{rem:ds vs gt} for discussion of terminology), it turns out to be a ``shadow" of the ($1$-degree) \emph{Griffiths transversality} of Breuil's $N$-operator \cite{Bre97}.
To be more precise, we will  make use of the following property of  Kisin's $N_\nabla$-operator:
\[ N_\nabla(\fil^n \cmast) \subset E\fil^{n-1} \cmast, \]
which follows easily from its intertwining with $\phi$:
\[N_\nabla\varphi=\frac{pE(u)}{E(0)} \varphi N_\nabla. \]
Note the formula $N_\nabla(\fil^n \cmast) \subset E\fil^{n-1} \cmast$ looks like a ``twisted" form of Griffiths transversality; however the extra $E$-twist actually implies the stability  $N_\nabla(\fil^n \cmast) \subset  \fil^{n} \cmast$.  
\end{proof}

\begin{remark}
The properties of the $N$-operator and the $N_\nabla$-operator discussed above could be compared, leading to a certain ``\emph{unification}" in this framework, cf. Remark \ref{rem: nnabla and N}. 
This observation is indeed a strong psychological comfort for us, as initially we are very much confused by the many ``similar looking" yet different properties of these operators (cf. Remark \ref{rem:ds vs gt}); furthermore, it serves as a ``philosophical guide" for us to construct other ``lifted" operators to deduce results on the (reduced) Sen operator, cf. Theorem \ref{thm-intro-pGT} and Proposition \ref{prop:introtauA} for a further (even more intricate) example.
\end{remark}

\begin{remark}[``1-degree shrinking" vs. ``Griffiths transversality"] \label{rem:ds vs gt}
We  discuss (and justify) some of the terminologies used here.
\begin{enumerate}
\item In Hodge theory and its analogous use elsewhere (e.g.   \cite{Bre97}), ``Griffiths transversality" of an operator $f$  signifies its \emph{failure} of  preserving a  \emph{decreasing} filtration by at most one degree, i.e., it satisfies:
\[ f: \fil^\bullet \to \fil^{\bullet-1}\]
\item We say an operator $g_n$ satisfies $1$-degree shrinking on (degree $n$ of) an \emph{increasing} filtration  $\fil_\bullet$  if:
\[ g_n: \fil_n \to \fil_{n-1}.\] 
Note that $g_n$ not only \emph{preserves} $\fil_n$, it even \emph{shrinks} it; thus its effect is completely different from that of ``Griffiths transversality". In addition, we note that the terminology ``transversality" has geometric connotations; thus the terminology such as ``anti-" or ``opposite-" Griffiths transversality does not seem to be appropriate here.
Here, note we fix the degree $n$ above, since in our typical example such as Theorem \ref{thm: intro rational sen}, ``$\theta_\kinfty-n$" depends on $n$.

\item \label{item-intro-pgt}
 We shall also prove a certain ``$p$-degree shrinking" in Theorem \ref{thm-intro-pGT} (in the mod $p$ setting). 
In fact,  the method we use     also recovers a ``$p$-Griffiths transversality" of Bhatt--Lurie \cite{Bha22} (although unlike the ``$p$-degree shrinking", this will not be used in our applications). 
Most interestingly, we show that these two distinct phenomena follow from properties of a same \emph{lifted} operator.  

\item As discussed above, it turns out the phenomena of ($1$-degree resp. $p$-degree) shrinking and Griffiths transversality co-exist, and sometimes could be ``unified". 
This was indeed a source of great confusion in our initial investigation; we hope this work could clarify some of these aspects. In fact, most interestingly, we find  that the (heavy) application of  Griffiths transversality in previous works (in particular, \cite{GLS14}) could actually be \emph{replaced} by the use of $1$- resp. $p$-degree shrinking, which  leads to  a substantially more conceptual reproof of difficult results in \cite{GLS14}, cf. \S \ref{subsec: GLS reproof}.
\end{enumerate}
\end{remark}

We now refine Theorem \ref{thm: intro rational sen} to an  \emph{integral}  filtered version. Let   $\star \in \{ \emptyset, \log \}$, and suppose $T$ is $\star$-crystalline (where log-crystalline means semi-stable). 
Let $E'(\pi)$ be $\frac{d}{du}(E)$ evaluated at $\pi$.
Let
 \begin{equation} \label{eq: intro a constant}
a=
\begin{cases}
  E'(\pi), &  \text{if } \star=\emptyset \\
 \pi E'(\pi), &  \text{if } \star=\log
\end{cases}
\end{equation}

\begin{theorem}[cf. Theorems \ref{thm: integral sen integral conj fil} and \ref{thm: mod p Sen op fil}] \label{thm: intro integral sen} 
 Suppose $T$ is  $\star$-crystalline, and use constant $a$ in Eqn \eqref{eq: intro a constant}.  
The  \emph{amplified Sen operator}  
$$\Theta=a\theta_\kinfty: \gm_\HT[1/p] \to \gm_\HT[1/p]$$ 
satisfies the following:
\begin{enumerate}
 \item $\Theta$ is integral, that is:
\[ \Theta (\gm_\HT) \subset \gm_\HT\]

\item $\Theta-an=a(\theta_\kinfty-n)$ satisfies $1$-degree shrinking on (integral) $\fil_n$: that is, it induces a map  
\begin{equation*} \label{eq intro theta int}
 \Theta-an:\fil_n^\rmconj  \gm_\HT \to   \fil_{n-1}^\rmconj \gm_\HT
\end{equation*}  

\item In addition, consider the mod $p$ operator:
\[ \overline{\Theta}: \bargmht \to \bargmht\]
then $\overline{\Theta}-an$ also satisfies $1$-degree shrinking in the sense that it induces 
\begin{equation*} \label{eq intro theta modp}
\overline{\Theta}-an:\fil_n^\rmconj  \bargm_\HT \to   \fil_{n-1}^\rmconj \bargm_\HT
\end{equation*}  
\end{enumerate}
\end{theorem}
\begin{proof}[Idea of proof.] To obtain this integral version, we study the  integral  property of Kisin's $N_\nabla$-operator: indeed, it is known (cf. \cite{Liu08, Gao23}) that $N_\nabla$ is equal to a normalization of $\log(\tau)$ (cf. Notation \ref{nota lie hatg} for $\tau$); thus it suffices to carefully study the $p$-adic bounds of the log-expansion. 
As a technical comment, Item (3) does not directly follow from Item (2) as the filtered map $\fil_\bullet \gmht \to \fil_\bullet \bargmht$ is not strict; however, a similar bounding strategy works.  
(We mention a minor but subtle point here: one can certainly consider reduction of $\Theta$ modulo $\pi$ (instead of modulo $p$), however this is not the correct normalization in this paper, cf. Remark \ref{rem not mod pi}  for more comments.)
\end{proof}

\begin{remark} \label{rem:intro-filSen}
 When $K$ is unramified and $T$ is crystalline, Theorem  \ref{thm: intro integral sen} is first due to Bhatt--Lurie \cite{Bha22} via a stacky approach. It should be mentioned that both authors of this paper are not experts with stacks, particularly the (very deep) stacky techniques in Bhatt--Lurie's work  \cite{Bha22}. In particular, even after finishing a first draft of this paper, 
we do not fully understand its relation to \cite{Bha22} or \cite{GK-ias}.
We thank  Bhargav Bhatt, Toby Gee and Mark Kisin for many useful communications which help us to understand the following stacky picture.  
In the following, we freely use the stacky notations from \cite{Bha22}. 
\begin{enumerate}
\item Let  $T$ be a  crystalline representation, let $\cale$ be the corresponding \emph{reflexive} $F$-gauge constructed in \cite[Theorem 6.6.13]{Bha22}, which is a sheaf on $\o_K^\syn$. Consider the pull-back of $\cale$ to (the Hodge--Tate component) $(\o_K^\caln)_{t=0}$, which has an explicit  presentation by $\mathbf{A}^1/(\mathbf{G}_a^\sharp \rtimes \mathbf{G}_m)$ as in \cite[Proposition 5.3.7]{Bha22}. An explicit computation  of quasi-coherent sheaves on  $(\o_K^\caln)_{t=0}$  leads to the statements in Theorem \ref{thm: intro integral sen}. To be more precise, this is already carried out in \cite[\S 6.5.4]{Bha22} (in the mod $p$ case); cf.~ in particular (the diagram and ensuing argument in)   \cite[Remark 6.5.11]{Bha22}.  In addition, Bhatt explains to us that the argument of  \cite[Proposition 5.3.7]{Bha22} can be modified to accommodate the case with  $K$  ramified. 
\item 
We also note that the phenomenon of $1$-degree shrinking in Theorem \ref{thm: intro integral sen}  already appears in \cite{BL1} (pre-dating prismatic $F$-gauges), albeit then in a cohomological  setting, cf.~ e.g.~ \cite[Remark 4.9.10]{BL1}. 
Note in loc.~ cit., $K$ is unramified and hence admits $q$-de Rham prism, and thus  the  Sen operator  there is the ``classical" one (over the cyclotomic tower), cf. \cite[\S 3.9]{BL1}. This is ``compatible" with our Sen operator over the Kummer tower (say, after linear extension to $C$, or to $\o_C$ in the integral crystalline case), by \cite{GMWHT}, cf. also Theorem \ref{thmkummersenop} for a quick review.  

\item In some sense, Bhatt--Lurie's  approach is purely geometric: indeed, the $1$-degree shrinking phenomenon happens for any coherent sheaf ---even those not necessarily related with crystalline representations--- on   $(\o_K^\caln)_{t=0}$. Whereas our proof (as mentioned in proof of Theorem \ref{thm: intro rational sen}) crucially uses (a variant of) the Griffiths transversality of Breuil's $N$-operator. We hope our approach explicates the relation between the ``stacky" operator  and the classical ones. In addition, since our operator is also explicitly defined for general semi-stable representations, it informs us why some of the applications in this paper does not naively generalize for semi-stable representations.

\end{enumerate}
\end{remark}

Note when $a \neq 1$, then $\overline{\Theta}$ is necessarily nilpotent (and hence has limited use). However, when $K$ is unramified and $T$ is crystalline (hence $a=1$), then it turns out that $\overline{\Theta}$  has  very intricate relations with (mod $p$) filtrations, which we explain in the following.

\begin{notation}\label{nota:introkunram}
 Suppose $K$ is unramified and $T$ is crystalline. Note in this case, $a=1$ and hence 
\[ \Theta =\theta_\infty.\]
Consider $\phi(\bargm) \subset \bargmast$ as a $k[[u^p]]$-submodule, with the induced ``sub-Nygaard" filtration:
\[ \fil^\bullet \phi(\bargm) :=\phi(\bargm) \cap \fil^\bullet  \bargmast =\phi(\bargm) \cap u^\bullet \bargm \]
For $n, i\in \bbz$, denote $N_{n,i}$ as the image of the composite map
\[ u^i\fil^{n-i}\phi(\bargm)/u^i\fil^{n+1-i}\phi(\bargm) \into \gr^n \bargmast \xrightarrow{\simeq, u^{-n}} \fil_n \bargmht\] 
It is easy to see $N_{n,i}=N_{n-i,0}$, and each $N_{n,i}$ can be regarded as a piece of the conjugate filtration.
 Note since these pieces come from a $k[[u^p]]$-module, there is in general no relation between $N_{n-1,0}$ and $N_{n,0}$; rather, we only have $N_{n-p,0} \subset N_{n,0}$.
\end{notation}

 The following \emph{$p$-degree shrinking} result forms a  technical core of this paper. It bears some ``resemblance" but is completely different from a \emph{$p$-Griffiths transversality} property of Bhatt--Lurie \cite{Bha22}, although they could be ``unified" in some sense, cf. Remark \ref{rem:intro-unify}. As far as we understand, Theorem \ref{thm-intro-pGT} does not seem to   follow from constructions in \cite{Bha22}.
 
\begin{theorem}[Theorem \ref{prop: p griffiths}] \label{thm-intro-pGT}
 Suppose $K$ is unramified and $T$ is crystalline, and use Notation \ref{nota:introkunram}. Then $\overline{\Theta}-n$ satisfies a ``$p$-degree shrinking" in the sense that:
\[ (\overline{\Theta}-n)(N_{n,0}) \subset N_{n-p,0}. \]
In particular, the   eigenvalues of $\overline{\Theta}$ acting on $N_{n,0}$ are all $n$ (mod $p$). 
In addition, we have a direct sum decomposition
\[ \oplus_{i=0}^{p-1} N_{n,i} = \fil_n \bargmht,\]
where the left hand side is exactly the generalized eigenspace decomposition with respect to $\overline{\Theta}$-action on $\fil_n \bargmht$.
\end{theorem}
\begin{proof}[Idea of proof.]
For the  proof of Theorem \ref{thm: intro integral sen}, we needed to study \emph{integral} properties of the (normalized) $\log(\tau)$ expansion which involves $p$ in the denominators that nonetheless could be taken care of. 
However, for this (very intricate) mod $p$ theorem, we need to make use of a \emph{purely integral} operator: 
\[ \delta_\tau:= \frac{\tau-1}{u\frac{p}{E(0)}\fkt} \]
(where $\fkt \in \ainf$ is a ``normalizing" element, cf. Notation \ref{nota: ring for mod}), which is exactly the \emph{first term} of the  (normalized) $\log\tau$-expansion  mentioned above.
It is well-known (e.g.  \cite{Liu08, Gao23}) that this operator induces a map
\[ \delta_\tau: \gm \to \gm_\inf=\gm\otimes_\gs \ainf.\]
In addition, by a sharp observation of  Yu Min and Yupeng Wang (cf. Proposition \ref{prop MinWang}), the reduction mod $(E,p)$ of this operator is exactly the mod $p$ Sen operator $\overline{\Theta}$. 
Following the main spirit of this paper, and noting $\delta_\tau$ has an obvious intertwining with the Frobenius operator, we desire to establish some  similar ``$p$-degree" property  on this \emph{lifted} level. 
As such, it requires us to \emph{stabilize} the $\delta_\tau$ operator: that is, we will construct a ring $\gs \subset A \subset \ainf$ (cf. below) such that $\delta_\tau$ induces a stable operator:
\[ \delta_\tau: \gm_A \to \gm_A,\]
where  $\gm_A:=\gm\otimes_\gs A$.
In some sense, this could be regarded as a \emph{purely integral} variant of Kisin's $N_\nabla$-operator (which is stable on $\gm\otimes_\gs \o$). 
Once stabilized, it is relatively easy to deduce  the desired properties for $\delta_\tau$, cf. Proposition \ref{prop:introtauA}; 
after  reduction modulo a maximal ideal of $A/pA$, (and some extra work), this leads to the $p$-degree shrinking in Theorem \ref{thm-intro-pGT}.
\end{proof}

We expand on the stabilization of $\delta_\tau$ mentioned in above proof. This will have to be   complicated as we need to balance between the size of the ring and the properties that we desire to have.
Let $\gs^1_\log$ be the co-product of the (log-) Breuil--Kisin prism on the \emph{log}-prismatic site, and let $\gs^1_{\log, \max} =\gs^1_\log[E/p]^{\wedge_p}$. Denote
\[ A : = \gs^1_{\log, \max}[1/p] \cap \ainf. \]
The appearance of the ``max" ring here is partly inspired by a relatively easy ``max"  integral lift of Sen operator in \S \ref{sec: sen max classical}; the intersection with $\ainf$ comes from our desire to construct a flat ring (at least in mod $p$ case). 
Slightly unfortunately, this ring $A$ is rather implicit: in fact, it is not clear if $A$ is $p$-adic complete (cf. Remark 
\ref{rem implicit A}) although it is irrelevant for our purpose; in addition, the appearance of \emph{log}-prismatic rings (in the study of crystalline representation) is slightly strange although it turns out to be \emph{necessary} and natural, cf. Remark \ref{rem:logpris-necessary}.
Fortunately, one can easily check that  $A/pA$ is flat over $\gs/p\gs$, making   it convenient to study conjugate filtrations which satisfy \emph{flat base change} (Lemma \ref{lem: flat base change}). 
(We do speculate that it might be possible to  find a more explicit ring $A$, perhaps in a more ``explicit" (ring-theoretic)   syntomic  cohomology theory of Bhatt--Lurie \cite{Bha22}).

\begin{prop}[cf. Proposition \ref{prop-pGT-A-level}] \label{prop:introtauA}
Suppose $K$ is unramified and $T$ is crystalline. 
Let $\gm_A=\gm\otimes_\gs A$, then
\[ (\tau-1)(\gm_A) \subset u\fkt\gm_A\]
and thus $\delta_\tau$ is stable on $\gm_A$.
Denote $\bargm_A:=\gm_A/p\gm_A$, consider the mod $p$ operator 
\[ \delta_\tau= \frac{\tau-1}{u\frac{p}{E(0)}\fkt}: \bargm_A \to \bargm_A.\]
Then we have:
\begin{enumerate} 
\item \label{item2pgtaintro} $\delta_\tau(\Fil^n \varphi (\bargm_A)) \subset u^p \fil^{n-p}\phi(\bargm_A)$; 

\item  \label{item3pgtaintro} \emph{($p$-Griffiths transversality)}:  $\delta_\tau(\Fil^n  \bargm_A) \subset  \fil^{n-p} \bargm_A$. 
\end{enumerate}
Here $\fil^\bullet \phi(\bargm_A):=\phi(\bargm_A) \cap u^\bullet \bargm_A$, and $\fil^\bullet \bargm_A$ is induced from the bijective map $\bargm_A \xrightarrow{\phi} \phi(\bargm_A)$.
\end{prop}
\begin{proof}[Idea of proof.]
The construction of the operator $\delta_\tau$ follows from carefully bounding the range of $(\tau-1)$-action on $\gm$ and $A$. Once constructed, its (mod $p$) property quickly follows from its obvious intertwining with the Frobenius operator
 \[\delta_\tau \phi =\frac{p}{\phi(E(0))}u^p \phi \delta_\tau. \] 
Indeed, the $u^p$ above accounts for the ``twist" in Proposition \ref{prop:introtauA}\eqref{item2pgtaintro}.
Indeed, Proposition \ref{prop:introtauA}\eqref{item2pgtaintro} looks like a ``twisted" form of $p$-Griffiths transversality; however it actually implies the stability
\[\delta_\tau(\Fil^n \varphi (\bargm_A)) \subset \Fil^n \varphi (\bargm_A).\]
Note the similarities here with the properties of $N_\nabla$-operator  sketched   in   proof of Theorem \ref{thm: intro rational sen}: there the ``twisted Griffiths transversality" induces the $1$-degree shrinking; here the ``twisted $p$-degree Griffiths transversality" induces the $p$-degree shrinking in Theorem \ref{thm-intro-pGT}.
\end{proof}

\begin{remark} \label{rem:intro-unify}
As mentioned in Remark \ref{rem:ds vs gt}\eqref{item-intro-pgt}, our stabilized operator $\delta_\tau$ can  be also used to recover (via a certain ``de Rham specialization", in contrast to the ``Hodge--Tate specialization" in Theorem \ref{thm-intro-pGT}) a ``$p$-Griffiths transversality" of Bhatt--Lurie \cite{Bha22}, cf. Theorem \ref{thm:BL-p-GT}. 
Indeed, our $p$-degree shrinking in Theorem \ref{thm-intro-pGT} follows from Proposition \ref{prop:introtauA}\eqref{item2pgtaintro}, whereas the $p$-Griffiths transversality follows from Proposition \ref{prop:introtauA}\eqref{item3pgtaintro}: note the later is a  completely different phenomenon,  and is not used for our applications. See \S \ref{sec:pDS vs pGT} for more discussions and comparisons.
\end{remark}

\subsection{Applications: Hodge filtrations and Frobenius shapes} \label{subsec:intro-app}

We now discuss applications of above filtered Sen theory to (more classical) questions of Hodge filtrations as well as shapes of the Frobenius matrix. 
When $K$ is unramified and $T$ is crystalline with Hodge--Tate weights bounded in the range $[0,p]$, then \emph{all} results in this subsection are known by the work \cite{GLS14}: there, one makes heavy use of the \emph{Griffiths transversality} of Breuil's $N$-operator to deduce structures of the Hodge filtration, which in turn yield structures of the Nygaard filtration and hence the Frobenius matrix. 
The main innovation here is that we make use of the extra \emph{conjugate filtration} together with various \emph{shrinking} properties of the Sen operator  discussed in above subsection. This not only substantially simplifies the proof of many difficult results in \cite{GLS14}, it also allows us to obtain   results outside the range  $[0,p]$.

We start by explaining a torsion control and vanishing theorem on the integral resp. mod $p$ Hodge filtration. The mod $p$ result is first due to Gee--Kisin \cite{GK-ias}, and  our integral result is inspired by their theorem. See Remark \ref{rem history intro} for more historical comments, and see Remark \ref{rem: f gauge vb} for our motivations from \cite{GLS14}.

\begin{theorem}\label{thm: intro vanishing}
Suppose $K$ is unramified. Let $T$ be a   stable  $\zp$-lattice in a crystalline representation of $\gk$ with Hodge--Tate weights $\{r_1, \cdots, r_d\}$ where $0 \leq r_1 \leq \cdots \leq r_d$, and let $\gm$ be the associated Breuil--Kisin module.
\begin{enumerate}
    \item \label{item integral ver intro}
    (cf. Thm. \ref{thm: vanish integral}). Suppose $n$ is not in the set $\{ r_i+kp, k \geq 0, 1\leq i\leq d \} \cap [0, r_d]$,  then
    \[ \gr^n \gm_\dR =0. \]
    In addition, for each $n$,  $(\gr^n \gm_\dR)_\tor$ is uniformly killed by $(r_d-1)!$ and has number of generators  uniformly $\leq d$.
(See Theorems \ref{thm: bound expo} and \ref{thm: bound generator} for more precise torsion bounds.) 
%So in particular, when $K/\qp$ is a finite extension, $(\gr^n \gm_\dR)_\tor$ is finite for each $n$.
 
    \item \emph{(Gee--Kisin \cite{GK-ias})} \label{item mod p ver intro} (cf. Thm. \ref{thm: mod p jump}). Suppose $n$ is not   in the set $\{ r_i+kp, k \in \bbz, 1\leq i\leq d \} \cap [0, r_d]$, then 
\[ \gr^n \bargm_\dR =0. \]
More precisely, let $b_1\leq  \cdots \leq b_d$ be the jumps of $\fil^\bullet \bargm_\dR$ counted with multiplicities, then $0 \leq b_i \leq r_d$ for each $i$ and
\[ \{ b_1, \cdots, b_d\} \equiv  \{r_1, \cdots, r_d\} \pmod p\]
in the sense that both sides define a same (un-ordered) set of elements in $\bbz/p\bbz$ with same multiplicities. 
\end{enumerate}
 \end{theorem}
 
 We first give some quick technical comments.   As a special case (that is easier than the complicated-looking conditions in the theorem), whenever $n \not\equiv r_i \pmod p$ for all $i$, then \[ \gr^n \gm_\dR =0, \quad \text{ and } \gr^n \bargm_\dR =0.\]
Caution: the integral and mod $p$ results do not imply each other, as $\fil^\bullet \gm_\dR \to \fil^\bullet \bargm_\dR$ is not strict.
 
 \begin{proof}[Idea of proof for Theorem \ref{thm: intro vanishing}.]
 We   briefly discuss how to use  the integral filtered   Sen operator to prove integral vanishing in Theorem \ref{thm: intro vanishing}(1); the torsion control results also depend on studying these operators. (The mod $p$ case follows similar ideas, using the mod $p$ Sen operator.)  
Indeed, suppose $n$ satisfies the condition in Theorem \ref{thm: intro vanishing}\eqref{item integral ver intro}. We claim the composite
\[ \fil_n^\rmconj \gm_\HT \xrightarrow{\Theta-an} \fil_{n-1}^\rmconj \gm_\HT \into \fil_n^\rmconj \gm_\HT\]
 is bijective; this would imply 
$\fil_{n-1}^\rmconj \gm_\HT = \fil_n^\rmconj \gm_\HT$ thus vanishing of $\gr_n \gm_\HT$ (equivalently, of $\gr^n \gm_\dR$, by  Lemma \ref{lem: matching graded dR and HT}). 
To wit, it reduces to compute   eigenvalues of the above endomorphism (after inverting $p$, and hence Theorem \ref{thm: intro rational sen}  is applicable): they are precisely the $a(r_i-n)$'s where $a$ is the constant in \eqref{eq: intro a constant}---but $a=1$ precisely because $K$ is unramified and $T$ is crystalline---; these are $p$-adic units and we can conclude.  
%(This also explains why  Theorem \ref{thm: intro vanishing} is restricted to $a=1$ case. Indeed past experiences show that the other cases do not behave well in general; nonetheless, cf. Remark \ref{rem: future} for some comments on ramified case).
 \end{proof}

\begin{remark} \label{rem history intro}
 We give some historical remarks about Theorem  \ref{thm: intro vanishing}.
\begin{enumerate}
    \item (\emph{Mod $p$ results.}) Theorem \ref{thm: intro vanishing}(2)  is first announced by Gee--Kisin in \cite{GK-ias}; their key tool is the $F$-gauge attached to crystalline representations as constructed by Bhatt--Lurie \cite{Bha22}. Continuing the discussions in Remark \ref{rem:intro-filSen}, Gee and Kisin explained to us in detail that their key argument in proving their Thm. \ref{thm: intro vanishing}\eqref{item mod p ver intro}  hinges on realizing  (Rees construction associated to) the filtered modules $\fil_\bullet \gm_\HT \otimes_\ok k$ and $\fil_\bullet \bargmht$ as objects living over $(\o_K^\caln)_{t=0,p=0}$. 
 Indeed, by the \emph{stacky}  filtered Sen theory  of \cite{Bha22}, it is equivalent to construct certain  graded modules over a (non-commutative) ring $k\{x, D\}/(Dx-xD-1)$ (cf. \cite[\S 6.5.4]{Bha22}).
This allows Gee--Kisin to reduce the question to a concrete module-theoretic problem.

In comparison to Gee--Kisin's technique, our main argument is more explicit, and uses eigenvalue computations.  Nonetheless, it might be the case that the two proofs are essentially equivalent after unravelling all the details.

\item   (\emph{Integral results.}) Our   Theorem \ref{thm: intro vanishing}\eqref{item integral ver intro} (the integral vanishing and torsion bound) is inspired by Gee--Kisin's theorem, as  well as the known case in \cite{GLS14} when the Hodge--Tate weights is bounded in $[0,p]$ (cf. Remark \ref{rem: f gauge vb} for more comments on motivation).    
The  vanishing part (about $ \gr^n \gm_\dR$) in Theorem \ref{thm: intro vanishing} (1) was also known by the second author \cite{Liu24} via a different method.  
After an early draft of this paper, Gee and Kisin show us  that they can also build upon their stacky method and   module theoretic argument to reprove the integral vanishing result (cf.~Theorem \ref{thm: vanish integral})   
and strengthen our (earlier version of) Theorem \ref{thm: bound generator} on bound of generators (cf. Remark \ref{rem:GK improve}; we are grateful to Gee and Kisin for allowing us to include the strengthened version here).
In addition, during the preparation of this paper, Dat Pham \cite{Pham} informed us that he also independently observed one can use (filtered Sen theory from \cite{Bha22} and)  similar eigenvalue argument   to prove integral vanishing part in Theorem \ref{thm: intro vanishing}(1).
 It is natural to speculate that perhaps one can continue the stacky argument to reprove/improve Theorem  \ref{thm: bound expo} on bound of exponent (and hence completely recover  Theorem \ref{thm: intro vanishing}(1)).

  As a final remark, we   expect these torsion bound results to be useful in future  investigations such as the ramified case  (cf.~e.g.~ Example \ref{ex: torfree p 2p} and Remark \ref{rem: future}): in some sense, the integral vanishing result only isolates the ``bad" positions, but the torsion bound results control how bad they can be.
  \end{enumerate} 
\end{remark}

\begin{remark}\label{rem: f gauge vb}
 We   comment on our motivation and possible future applications of  Thm. \ref{thm: intro vanishing}.
\begin{enumerate} 
\item As mentioned above, for an integral crystalline representation $T$,  Bhatt--Lurie (cf.~\cite{Bha22}) constructs an $F$-gauge $\cale$ which is a \emph{reflexive} sheaf  on $(\ok)^\syn$; it is not necessarily a vector bundle. 
 Consider the filtered map
\begin{equation} \label{eq fil int strict}
\fil^\bullet \gm_\dR \to \fil^\bullet (\gm[1/p])_\dR =\fil^\bullet D_\dR(T[1/p])
\end{equation}
The content of \cite[Prop. 4.5]{GLS14} implies that  the filtered map \eqref{eq fil int strict}  is strict if and only if $\cale$ is a vector bundle  (equivalently, in more concrete terms, if and only if $\fil^\bullet \gmast$ admits \emph{adapted basis}; cf. Item (2) of \cite[Prop. 4.5]{GLS14}).  (See also Lemma \ref{lem: weak frob equiv} and Remark \ref{rem frob shape f gauge} for more discussions on this vector bundle condition.)  
As an example,   \cite[Prop. 4.16]{GLS14}  proves that when the Hodge--Tate weights are in the range $[0,p]$, then \eqref{eq fil int strict}  is strict, and hence $\cale$ is a vector bundle. Note \eqref{eq fil int strict}   being strict implies $\gr^n\gm_\dR=0$ if and only if $n \neq r_i$.
  Conversely, the \emph{failure}  of strictness of \eqref{eq fil int strict}    can be used as a measure of failure of $\cale$ being a vector bundle.  Theorem  \ref{thm: intro vanishing} informs us that to examine  \eqref{eq fil int strict}   (and hence its failure of strictness),   it suffices to concentrate at those $n$'s congruent to Hodge--Tate weights.

\item On a more classical note, the existence of adapted basis in \cite{GLS14}   has strong implications on the shape of Frobenius operator on the Breuil--Kisin module, crucially used there for the study of reduction of crystalline representations, which in turn has application to Serre weight conjectures. 
Indeed, based on the ideas to prove Theorem \ref{thm: intro vanishing} (that is:   filtered  Sen theory), we can give a  substantially  more conceptual reproof of a very difficult theorem  \cite[Prop. 4.16]{GLS14}  cited above  (which  requires the necessary assumption $r_d \leq p$); see \S \ref{subsec: GLS reproof}.

\item The above remarks show that the geometric structures  of $F$-gauges are strongly tied with algebraic  structures (e.g., filtration, Frobenius) of Breuil--Kisin modules. 
We expect these relations, and in particular Theorem \ref{thm: intro vanishing} to be useful in extending above results, e.g., to the case when $r_d >p$.
 \end{enumerate}
\end{remark}

 We also obtain another mod $p$ result which is technical looking at first glance,   but   is strongly inspired by the known case when $r_d \leq p$ in \cite{GLS14}, and is expected to be useful for   applications in Serre weight conjectures.   
Suppose $K$ is unramified and $T$ is crystalline.
  Take any basis $\vec{e}$ of the mod $p$ Breuil--Kisin module $\bargm$, and write $\phi(\vec e)=(\vec e) A$; as $k[[u]]$ is a valuation ring, the   matrix $A $ always have  a decomposition $$A=XDY$$
  with $X, Y \in \GL_d(k[[u]])$ and $D$ a diagonal matrix. One can easily compute (cf. Lemma \ref{lem: mod p conj fil}) that the  diagonal entries of $D$ (up to permutation) are exactly   $u^{b_1}, \cdots, u^{b_d}$  with $b_i$ as  in Theorem \ref{thm: intro vanishing}\eqref{item mod p ver intro}  above. Thus, the content of Theorem \ref{thm: intro vanishing}\eqref{item mod p ver intro} gives control on (these $b_i$ and hence) the matrix $D$.
   The following theorem, first due to ongoing work of Bhatt--Gee--Kisin \cite{BGK} (cf.~ Remark \ref{rem y matrix} for its history), gives control on $Y$. 
  
  %Let us quickly mention that although the statement is about a  technical condition on a matrix, the actual content (and the proof) is indeed about the \emph{mod $p$ Hodge filtration}; cf.~ \S \ref{sec:mod-p-shape} for details.

\begin{theorem} \label{Theorem intro y matrix}
\emph{(Bhatt--Gee--Kisin \cite{BGK})}
 (cf. Theorem \ref{thm1 mod p frob}).
 Suppose $K$ is unramified and $T$ is crystalline.
 Use notations in above paragraph. 
Then  $Y\in \GL_d(k[[u^p]])$; equivalently, after a change of basis, $Y$ can be made into the identity matrix. 
\end{theorem}

\begin{proof}[Idea of proof.] In earlier work of Bartlett \cite{Bar20a}, he shows that the validity of Theorem \ref{Theorem intro y matrix} is equivalent to the \emph{coincidence} of a \emph{sub-Hodge filtration} with the Hodge filtration on $\bargm_\dR$, cf. Lemma \ref{lem: mod p Frob two fil}. Following the main theme of this paper, we discover that Bartlett's equivalent condition can be translated (via non-trivial argument) into another condition concerning \emph{coincidence} of a ``\emph{sub-conjugate filtration}" with the conjugate filtration on $\bargm_\HT$. 
It turns out this coincidence is exactly the equality of the generalized eigenspace decomposition in Theorem \ref{thm-intro-pGT} (which in turn follows from the $p$-degree shrinking property), and thus we can conclude.
   We remark that the spirit here is similar to the proof of Theorem \ref{thm: intro vanishing}: it all boils down to the \emph{shrinking} property of the Sen operator and its consequence on \emph{eigenvalues}. 
\end{proof}

\begin{remark} \label{rem y matrix}  
  We remark on the history of Theorem \ref{Theorem intro y matrix}. 
  We do not know the validity of   Theorem \ref{Theorem intro y matrix} in our first draft of this paper; we then learn from Gee--Kisin that they can prove Theorem \ref{Theorem intro y matrix}  in dimension $2$. 
After learning of Gee--Kisin's result, our investigation and proof of Theorem \ref{Theorem intro y matrix}  is inspired by the methods in \cite{GLS14} as well as (unexpectedly) our desire to understand a certain \emph{$p$-Griffiths transversality} in \cite{Bha22}.
As discussed in above sketch, it turns out that  we instead could construct a  new   \emph{$p$-degree shrinking} structure, which leads  to a  conjugate-filtration version  of a lemma of Bartlett \cite{Bar20a} which in turn implies Theorem \ref{Theorem intro y matrix}.
 After we obtain the proof of Theorem \ref{Theorem intro y matrix}, we learn that Bhatt--Gee--Kisin already have a proof before us; as far as we are informed, their proof builds on a stacky approach, and does not seem to directly translate into our proof.

% Interestingly, we also construct certain $p$-degree shrinking, but now for certain \emph{increasing} filtration, in contrast to a certain \emph{decreasing} filtration in \cite{Bha22}; cf. Remark \ref{rem: BL p-griffiths}. 

  % Indeed, in this case, the matrix $D$ is \emph{completely} determined by the Hodge--Tate weights (without modulo $p$). However, the fact that $Y$ is a matrix over $k[[u^p]]$, even in the situation of \cite{GLS14}, is a most difficult one, and relies on very delicate ``approximation" techniques. 
   
 \end{remark}

%\subsection{Comparison to a stacky approach} \label{rem: stacky approach} 

\subsection{Structure of the paper}  

 We divide the paper into 4 parts.
 \begin{itemize}
 \item Part 1, \S \ref{sec: conj fil} -\S \ref{sec: sen max classical}. In this part, we use ``classical tools" to construct (integral) filtered Sen theory.  In \S \ref{sec: conj fil}, we review basic properties of conjugate filtrations. 
In \S \ref{sec: BK mod}, we review   modules attached to semi-stable representations; operators on these modules lead to integral Sen theory in \S \ref{sec: integral Sen}. Incorporating the structure of conjugate filtrations, we obtain an upgrade to (integral)  \emph{filtered  Sen theory} in \S \ref{sec: fil sen}.
In \S \ref{sec: sen max classical}, we show one can ``lift" the integral Sen  operator to (a variant of) the Breuil--Kisin module level.

\item Part 2, \S \ref{sec: vanishing}-\S \ref{sec: frob matrix}. This part is the application of Part 1. 
In \S \ref{sec: vanishing}, we deploy filtered Sen theory to prove torsion bound and vanishing results on graded  of Hodge filtrations.
In  \S \ref{sec: frob matrix}, we use Sen theory to study shapes of Frobenius matrices; this in particular leads to a substantially more conceptual reproof of a technical result from \cite{GLS14} (assuming some mod $p$ results that will be proved in \S \ref{sec:mod-p-shape}).

\item  Part 3, \S \ref{sec: pris interpretation}-\S \ref{sec:p-GT-A}.   In this part, we  incorporate ideas of prismatic crystals to study filtered Sen theory.   In \S \ref{sec: pris interpretation},  we begin with prismatic interpretation of   filtered Sen theory studied in Part 1. In \S \ref{sec:truncated-sen}, we use the idea of Hodge--Tate prismatic crystals to construct a \emph{truncated} Sen operator. 
This will be \emph{lifted} to a \emph{stabilized} truncated operator in \S \ref{sec:p-GT-A}, using rings studied in \S \ref{sec:pris-max-ring}. 

\item  Part 4, \S \ref{sec:mod-p-shape}-\S \ref{sec:pDS vs pGT}. 
This part is the application of Part 3. 
We specialize the stabilized  truncated operator above to the mod $p$ Sen operator, establishing the  $p$-degree shrinking property   in \S \ref{sec:pDS vs pGT} and the Frobenius shape theorem in \S \ref{sec:mod-p-shape}; note these results are closely related: we separately write out \S \ref{sec:pDS vs pGT} for clarification and comparison with results of Bhatt--Lurie. 
 \end{itemize}

\subsection{Notations and  conventions} \label{subsec notation}
\begin{notation} \label{notafields}
 We   introduce some field (and ring) notations. 
\begin{itemize}
\item Let $\zeta_p$ be a primitive $p$-th root of unity, and inductively fix $\zeta_{p^n}$ so that $\zeta_{p^n}^p=\zeta_{p^{n-1}}$. Let $K_{p^\infty}=  \cup _{n=1}^\infty
K(\zeta_{p^n})$.

\item Let $\pi_0=\pi$, and inductively fix some $\pi_n$ so that $\pi_n^p=\pi_{n-1}$. 
Let $K_{\infty}   = \cup _{n = 1} ^{\infty} K(\pi_n)$.
When $p\geq 3$, \cite[Lem. 5.1.2]{Liu08} implies $\kpinfty \cap \kinfty=K$; when $p=2$, by \cite[Lem. 2.1]{Wangxiyuan}, we can and do choose some $\pi_n$ so that $\kpinfty \cap \kinfty=K$.
\item Let $C$ be the completion of $\barK$, with ring of integers $\o_C$; let $\ocflat$ be the tilt of the perfectoid ring $\o_C$. The sequence $(1, \zeta_p, \cdots, \zeta_{p^n}, \cdots)$ defines an element $\zetaflat \in \ocflat$; the sequence $\{\pi_n\}_{n \geq 0}$ defines an element $\piflat \in \ocflat$.
\item  Let $\gs=W(k)[[u]]$.  Let $\ainf=W(\o_C^\flat)$. 
Define a $W(k)$-linear embedding  $\gs \into \ainf$ via $u \mapsto [\pi^\flat]$. Let $E=E(u)=\mathrm{Irr}(\pi, W(k)) \in \gs$.
 Let   $[\zetaflat] \in \ainf$  be the  Teichm\"uller lift of $\zetaflat$.
\end{itemize}  
Let $L=\kpinfty \kinfty$. 
Let $$G_{\kinfty}:= \gal (\overline K / K_{\infty}), \quad G_{\kpinfty}:= \gal (\overline K / K_{p^\infty}), \quad G_L: =\gal(\overline K/L).$$
Further define $\Gamma_K, \hat{G}$ as in the following diagram:
\[
\begin{tikzcd}
                                       & L                                                                                             &                             \\
\kpinfty \arrow[ru, "\langle \tau\rangle", no head] &                                                                                               & \kinfty \arrow[lu, no head] \\
                                       & K \arrow[lu, "\Gamma_K", no head] \arrow[ru, no head] \arrow[uu, "\hat{G}"', no head, dashed] &
\end{tikzcd}
\]
Here we let $\tau \in \gal(L/K_{p^\infty})$ be  the  topological generator such that $\tau(\pi_i)=\pi_i\zeta_{p^i}$ for each $i$.
\end{notation}

We shall freely use the notion of locally analytic vectors in this paper; their relevance in $p$-adic Hodge theory is first discussed in \cite{BC16}. We also refer to \cite[1.4.2]{Gao23} for a very quick summary. Here, we recall the following Lie algebra operators.

\begin{notation} \label{nota lie hatg}
For $g\in \hat{G}$, let $\log (g)$ denote  the (formally written) series $(-1)\cdot \sum_{k \geq 1} (1-g)^k/k$. Given a $\hat{G}$-locally analytic representation $W$, the following two Lie-algebra operators (acting on $W$) are well defined:
\begin{itemize}
\item  for $g\in \gal(L/\kinfty)$ enough close to 1, one can define $\nabla_\gamma := \frac{\log(g)}{\log(\chi_p(g))}$;
\item for $n \gg 0$ hence $\tau^{p^n}$ enough close to 1, one can define $\nabla_\tau :=\frac{\log(\tau^{p^n})}{p^n}$.
\end{itemize}
These two Lie-algebra operators form a $\qp$-basis of $\Lie(\hat{G}$).
\end{notation}

\begin{convention}[covariant functors, Hodge--Tate weights vs. Sen weights] \label{conv: ht and sen effective} This is a paper on \emph{integral} $p$-adic Hodge theory (which also treats torsion representations), hence various  ``normalizations" are needed to simplify discussions (i.e., to \emph{stay positive}). 
    \begin{enumerate}      
 \item In this paper we   use many categories of modules and the functors relating them; we will always use \emph{covariant}  functors. This makes the comparisons amongst them easier (i.e., using tensor products, rather than $\Hom$'s).
  
\item Our $D_{\dR}(V)$ is    defined as the covariant functor $(V\otimes_{\Qp} \bdr)^{G_K}$. The (covariant) cyclotomic Sen operator is the Lie algebra operator $\nabla_\gamma$ in Notation \ref{nota lie hatg} (acting on Sen modules, cf. Construction \ref{notaSenop}). Thus, for the cyclotomic character $\chi_p=\zp(1)$, the \emph{Hodge-Tate weight} (filtration jumps of $D_\dR$) is $-1$, whereas the   \emph{Sen weight} (eigenvalue of Sen operator)   is 1.  That is: our  convention of Hodge-Tate weight   and Sen weight are \emph{opposite} to each other. In our paper, we will use the \emph{negative Sen operator} (particularly over the Kummer tower, cf. \S \ref{sec: integral Sen}) to reconcile this, cf.  also the next item.

\item In this paper, we only work with
\begin{itemize}
    \item (semi-stable) representations with  Hodge-Tate weights (equivalently, \emph{negative} Sen weights) $\geq 0$, for example $\chi_p^{-1}=\zp(-1)$;  
    
    \item As a consequence, their associated Breuil-Kisin modules  are effective, i.e.,  have  $E(u)$-heights $\geq 0$, and hence $\phi$ is defined without inverting $E(u)$. 
\end{itemize} 
   \end{enumerate} 
\end{convention}

\begin{convention}[More  on  $\pm$ signs] \label{conv minus plus}
 We summarize some other $\pm$-sign conventions made in this paper.
\begin{enumerate}
\item The $N_\nabla$ operator (cf. Construction \ref{cons: loc ana nnabla}) is the same as the one in \cite{Kis06}, hence is opposite to the one in \cite{Gao23} (thus also \cite{GMWHT}), cf. \cite[Rem. 4.1.3]{Gao23}. The operator $N_S$ in Notation \ref{nota: ring for mod} is the same as in \cite{Liu08}. ($N_S$ is not used in \cite{Gao23}).

\item As a consequence of previous item, the operator  $\frac{1}{\theta_\fon(u\lambda')}\cdot N_\nabla$ in Theorem \ref{thmkummersenop} is the \emph{negative} Sen operator over the Kummer tower. As discussed in Convention \ref{conv: ht and sen effective}, this is convenient for us.

\item In align with above item, our convention of the constant $a$ in Definition \ref{def: amplified Sen} is also opposite to that in \cite{GMWHT}. This makes it possible to have identification \[ \theta_\kinfty=\Theta \]
  in the key case where $K$ is unramified and $T$ is crystalline.
\end{enumerate} 
\end{convention}

\begin{convention}[``$\theta$-notations"] \label{conv: theta notation} We shall slightly abuse the symbol $\theta$ in this paper.
\begin{enumerate}
\item We   use $\theta_\kinfty$ to denote the negative $\kinfty$-Sen operator, cf. Remark \ref{rem: negative Sen op}. We then use $\Theta$ (the ``amplified" $\theta$) to denote the \emph{amplified}   Sen operator in Definition \ref{def: amplified Sen}.
\item We use $\theta_\fon$ to denote Fontaine's ``$\theta$-map"; this is the map $\theta_\fon: \ainf \to \o_C$ and $\theta_\fon: \bdrplus \to C$.
\end{enumerate}
\end{convention}

\subsection*{Acknowledgement} 
We thank 
Robin Bartlett, 
Haoyang Guo, 
Naoki Imai, 
Mark Kisin, 
Shizhang Li, 
Yu Min,
and Yupeng Wang for valuable discussions and feedback.  
We thank Florian Herzig for carefully reading an earlier draft and for providing  detailed comments. 
Our special thanks go to Bhargav Bhatt and Toby Gee, who patiently explained ideas of prismatic $F$-gauge and the Gee--Kisin theorem to the second named author.  
Part of the work was carried out when both authors were visiting IAS and BICMR, and when H.G.~ was visiting Purdue and T.L.~ was visiting SUSTech; we thank these institutions for excellent working conditions. 
During the stay at IAS, the first named author is supported by Infosys Member Fund, and the second named author is supported by the Shiing-Shen Chern Membership.
Hui Gao is partially
supported by the National Natural Science Foundation of China under agreements   NSFC-12071201, NSFC-12471011.

 %%\newpage 
 \addtocontents{toc}{\ghblue{Sen theory: classical approach}}
 \section{Review of conjugate filtration} \label{sec: conj fil} 

We review basic properties of conjugate filtrations. Some results presented here should be well-known in the literature (the  variants in \S \ref{subsec filphi} seem to be less well-known); we include brief proofs. 
The main applications are when the triple $(A,d, \phi)$ forms a prism, but we have presented a more axiomatized version. The readers could first read \S \ref{subsec:conj-fil} to familiarize with the conjugation filtrations.  
\S \ref{subsec axiomGLS} axiomatizes some argument in \cite{GLS14}; it is not too difficult but will only be used in \S \ref{sec: frob matrix}.
\S \ref{subsec filphi} concerns filtrations related to $\phi(M)$: it will be very critically used in the mod $p$ theory starting from \S \ref{sec:p-GT-A}.

\subsection{Conjugate filtrations and   gradeds} \label{subsec:conj-fil}
\begin{notation} \label{nota: isogeny}
    Let $A$ be a ring equipped with a ring endomorphism
\[ \phi: A\to A \]
Let $d \in A$ and suppose $d$ and $\phi(d)$ are non-zero-divisors. 
An \emph{effective isogeny}  with respect to the triple $(A,d, \phi)$ is a finite free  $A$-module $M$ equipped with an (effective) $d$-isogeny in the sense there is a $\phi$-semi-linear map
\begin{equation} \label{phieffeisog}
    \phi:   M \to M 
\end{equation}
such that the linearization
\[ 1\otimes \phi: A[1/d]\otimes_{\phi, A}  M \to A[1/d]\otimes_{A}  M \]
is an isomorphism.
Denote $M^\ast=A \otimes_{\phi, A} M$.
Since $\phi(d)$ is   a non-zero-divisor, $M^\ast$ can be regarded as a submodule of $A[1/d]\otimes_{\phi, A}  M$; 
 the bijection $1\otimes \phi$ sends $M^\ast$ into $M$ (which can be regarded as a submodule of $A[1/d]\otimes_{A}  M$ since $d$ is   a non-zero-divisor); henceforth, we can and shall  regard $M^\ast$ as a submodule of $M$.
\end{notation}

\begin{defn} \label{defn: conj fil} 
Use Notation \ref{nota: isogeny}. 
 Denote   
\[ M_\HT:=M/dM, \quad M_\dR:=M^\ast/dM^\ast \]
We define several   $\bbz$-filtrations.
\begin{enumerate}
\item Define the decreasing Nygaard filtration $\fil^n M^\ast: = M^\ast \cap d^n M$.  
%\{ x\in A \otimes_{\phi, A} M, (1\otimes \phi)(x) \subset d^n M  \}=

\item Define the decreasing Hodge filtration  $\fil^n M_\dR$ as the quotient filtration via $M^\ast \onto M_\dR$; one can check \[\fil^n M_\dR=\fil^n M^\ast/d\fil^{n-1} M^\ast.\]

\item The map
\[ \fil^n M^\ast \xrightarrow{d^{-n}} M\]
induces an injective map
\[  \fil^n M^\ast/\fil^{n+1} M^\ast  \xhookrightarrow{d^{-n}} M/dM. \] 
 Define $\fil_n  M_\HT=\fil_n^\rmconj M_\HT$  as the image of the  above map; these form the \emph{increasing} (cf. Lemma \ref{lem: matching graded dR and HT}) conjugate filtration.  
\end{enumerate} 
Since $\phi$ on $M$ is an \emph{effective} $d$-isogeny, all above filtrations are \emph{effective} in the sense $\gr^n=0$ (or $\gr_n=0$) for $n \leq -1$ (caution: for the increasing conjugate filtration, $\gr_n^\rmconj:=\fil_n^\rmconj/\fil_{n-1}^\rmconj$).
\end{defn}

\begin{rem}
    One can relax   condition \eqref{phieffeisog} to  
    \[ \phi:   M[1/d] \to M[1/d] \]
    (such that $1\otimes \phi$ is an isomorphism), the one obtains possibly non-effective filtrations. In this  paper, we shall only use effective $d$-isogenies in align with Convention \ref{conv: ht and sen effective}; thus all filtrations in this paper are effective ones.
\end{rem}

\begin{lemma}[Matching of graded] \label{lem: matching graded dR and HT}
  The conjugate filtration on $M_\HT$ is increasing. 
   In addition, the map
$\fil^n M^\ast \xrightarrow{1/d^n} M$  induces an isomorphism
    \[ \gr^n M_\dR  \simeq \gr_n M_\HT\]
    where LHS is $\fil^n/\fil^{n+1}$ and RHS is $\fil_n/\fil_{n-1}$. 
 \end{lemma}
 \begin{proof}
 One sees the conjugate filtration is increasing from the following diagram where all arrows are injective
\[
\begin{tikzcd}
\fil^n M^\ast/\fil^{n+1} M^\ast \arrow[rr, "d", hook] \arrow[d, "d^{-n}"', hook] &  & \fil^{n+1} M^\ast/\fil^{i+2} M^\ast \arrow[d, "d^{-n-1}", hook] \\
M/dM \arrow[rr, "="]                                                                                           &  & M/dM                                                                                           
\end{tikzcd}
\]
To match the gradeds, consider the following diagram
\begin{equation}\label{eq: matching of graded}
    \begin{tikzcd}
0 \arrow[r] & d\fil^n M^\ast \arrow[r] \arrow[d, hook] & d\fil^{n-1} M^\ast \arrow[r, "d^{-n} "] \arrow[d, hook] & \fil_{n-1}  M_\HT \arrow[d, hook] \arrow[r] & 0 \\
0 \arrow[r] & \fil^{n+1} M^\ast \arrow[d] \arrow[r]      & \fil^n M^\ast \arrow[d] \arrow[r, "d^{-n} "]            & \fil_n  M_\HT \arrow[d] \arrow[r]           & 0 \\
0 \arrow[r] & \fil^{n+1} M_\dR \arrow[r]                        & \fil^n M_\dR \arrow[r, ""]                                                 & \gr_n M_\HT \arrow[r]                      & 0
\end{tikzcd}
\end{equation} 
Here, the top two rows are short exact by definition. The bottom row is defined as cokernels of the top two rows, hence is also short exact; this implies $\gr^n M_\dR  \simeq \gr_n M_\HT$.
 \end{proof}

\begin{lemma}[Flat base change, and intersections] \label{lem: flat base change}
Let $(A, d)$ and $M$ be as in Definition \ref{defn: conj fil}. Let $A \into B$ be a flat embedding; suppose $\phi$ extends to $B$, and the triple $(B, d, \phi)$ still satisfies the assumptions in Notation \ref{nota: isogeny}.
One can then apply Definition \ref{defn: conj fil} to  $M_B=M\otimes_A B$ with respect to $(B, d, \phi)$. 
\begin{enumerate}
\item We have base change isomorphisms:
\[  \fil^n M^\ast\otimes_A B \xrightarrow{\simeq}  \fil^n M_B^\ast \]
\[ \fil^n  M_\dR \otimes_A B \xrightarrow{\simeq} \fil^n  M_{B, \dR}\] 
\[ \fil_n  M_\HT \otimes_A B  \xrightarrow{\simeq}  \fil_n  M_{B, \HT}\]

\item Suppose  the induced map $A/d \to B/d$ is    injective. Then we have the following identity (with intersection taken inside $M_B^\ast$):
\[ \fil^n M^\ast= M^\ast\cap \fil^n M_B^\ast \]

\item \label{flatbcitem3} Suppose  $A/d \to B/d$ is   injective (thus we can regard $M_\HT$ resp. $M_\dR$ as a subspace of $M_{B, \HT}$ resp. $M_{B, \dR}$), and further suppose one of the following:
\begin{enumerate}
\item $A/d$ is a field; or
\item $A=\gs, B=\ainf$ with $d=E(u)$ (cf. Notation \ref{notafields}),
\end{enumerate}
then we   have   the following identity (with intersection taken inside $M_{B, \HT}$ resp.  $M_{B, \dR}$):
\[ \fil_n  M_\HT=M_\HT \cap \fil_n  M_{B, \HT}  \]
\[ \fil^n  M_\dR= M_\dR \cap\fil^n  M_{B, \dR}\] 
\end{enumerate}
\end{lemma}
\begin{proof}
 Item (1). Note $\fil^n M^\ast=M^\ast \cap d^n M$, and note intersection commutes with flat base change; this leads to $ \fil^n M_B^\ast=\fil^n M^\ast\otimes_A B$. The other base change results follow by definition.

 Item (2). Injectivity of $A/d \to B/d$ implies $A \cap d^nB=d^nA$; thus $M \cap d^nM_B=d^nM$. Thus we have
\[M^\ast\cap \fil^n M_B^\ast = M^\ast\cap d^n M_B = M^\ast\cap d^n M =\fil^n M^\ast \]

 Item (3). We only consider the $M_\HT$-case since the the $M_\dR$-case is similar.
 If $A/d$ is a field, then any basis of the $A/d$-vector space $\fil_n M_\HT$ extends to a basis of $M_\HT$; one can then easily conclude using the base change isomorphism in Item (1). 
  Finally, consider the concrete case $A=\gs, B=\ainf$. Since $(E, p)$ is a regular sequence in both $\gs$ and $\ainf$, it is easy to check (cf. the discussion in Notation \ref{nota:adapted}(1), which is applicable here) that we have
\[ (\fil_n  M_\HT)[1/p]=M_\HT[1/p] \cap (\fil_n  M_{B, \HT})[1/p]\]
Note also $\fil^n M_\HT$ is a finite free $\ok$-module (although not necessarily a direct summand of $M_\HT$), the desired statement follows from the base change isomorphism in Item (1) and the fact that $K\cap \o_C=\ok$.
\end{proof}

\subsection{Adapted bases of filtrations} \label{subsec axiomGLS}
 
 We axiomatize   \cite[Proposition 4.5]{GLS14} concerning existence of ``adapted bases"; our new proof here is simpler, using conjugate filtration as a new   input. The typical example for Notation \ref{nota:adapted} is the triple $(\gs, E, \phi)$,  but the axiomatic set-up is convenient for our future work.

 \begin{notation} \label{nota:adapted}
Use notations in Notation \ref{nota: isogeny} and Definition \ref{defn: conj fil}.
Suppose furthermore that
\begin{itemize}
    \item $d$ and $\varphi(d)$ are non-zero-divisors of $A[1/p]$;
    \item  $d$ is contained in the Jacobson radical of $A$, and $A/d$ is a mixed characteristic   discrete valuation ring.
\end{itemize} 
Denote the fraction field as $F=(A/d)[1/p]$.  
\begin{enumerate}
\item Since $M[1/p]=M\otimes_A A[1/p]$ is now an effective isogeny over $(A[1/p], d, \phi)$, one can define various filtrations as in Definition \ref{defn: conj fil}. For simplicity, denote the relevant modules as
\[ M[1/p]^\ast, \quad D_\HT=M[1/p]/dM[1/p], \quad D_\dR=M[1/p]^\ast/dM[1/p]^\ast\]
Since $(d, p)$ is a regular sequence in $A$, it is easy to check that:
\[ \fil^\bullet  M[1/p]^\ast = (\fil^\bullet M^\ast)[1/p] \]
\[\fil^\bullet  D_\dR =(\fil^\bullet  M_\dR)[1/p] \]
\[\fil_\bullet  D_\HT =(\fil_\bullet  M_\HT)[1/p] \]
Further denote the jumps of the filtered $F$-vector space  $\fil^\bullet D_\dR$ as $0 \leq r_1 \leq \cdots \leq r_d$ (and call them the weights of $D_\dR$).

\item  
Say an $A/d$-basis $e_1, \cdots, e_d$ of $M_\dR$ is an adapted basis for  $\fil^\bullet M_\dR$, if for each $k \in \bbz$,
 \[ \fil^k M_\dR=\bigoplus_{i \text{ such that } k \leq r_i} A/d\cdot e_i \]
One can alternatively write above equation as
\[
\fil^k M_\dR=\oplus_{i=1}^d (A/d)\cdot 0^{\max\{0, k-r_i\}}e_i 
\]
with the convention that ``$0^0=1$".
Note this implies that  for each $i$,  $e_i \in \fil^{r_i}M_\dR -\fil^{r_i+1} M_\dR$.

\item Say an $A$-basis $\hat{e}_1, \cdots, \hat{e}_d$ of $M^\ast$ is an adapted basis for $\fil^\bullet M^\ast$ if for each $k \in \bbz$,
\begin{equation}\label{eqadptaed}
\fil^k M^\ast=\oplus_{i=1}^d A\cdot d^{\max\{0, k-r_i\}}\hat{e}_i 
\end{equation}  
 One can easily check that (using $d$ is a non-zero-divisor): $\fil^\bullet M^\ast$ has an adapted basis if and only if Eqn. \eqref{eqadptaed} holds for $k=r_d$.
 Note this implies for each $i$,  $\hat{e}_n \in \fil^{r_i} M^\ast -\fil^{r_i+1}M^\ast$.
\end{enumerate}
\end{notation}

\begin{lemma} \label{lem:adapted} 
Use Notation \ref{nota:adapted}. Then the following conditions are equivalent:
\begin{enumerate}
\item $\fil^\bullet M^\ast$ has an adapted basis.
\item  $\fil^\bullet M_\dR$ has an adapted basis.
\item $\fil^\bullet M_\dR \to \fil^\bullet D_\dR$ is strict.
\item $\gr^\bullet M_\dR =\gr_\bullet M_\HT$ (as $A/d$-modules) are torsion-free.
\end{enumerate}
Furthermore, suppose now condition (2) holds (equivalently, all conditions hold), let $(e_1, \cdots, e_d)$ be the basis of $M_\dR$ as in Notation \ref{nota:adapted}(2), and for each $i$, let $\hat{e}_i \in \fil^{r_i} M^\ast$ be any lift of $e_i$ along the surjection $\fil^{r_i} M^\ast \onto \fil^{r_i} M_\dR$, then  $(\hat{e}_1, \cdots, \hat{e}_d)$ forms an adapted basis of $\fil^\bullet M^\ast$.
\end{lemma}
\begin{proof}  
The implication (1) $\Rightarrow$ (2) is obvious. The equivalences (2) $\Leftrightarrow$ (3) $\Leftrightarrow$ (4) follows from the fact that $A/d$ is DVR (hence finitely generated torsion-free modules are finite free).
  We now prove (2) $\Rightarrow$ (1): that is, we prove the $(\hat{e}_1, \cdots, \hat{e}_d)$ constructed in the ``furthermore" part forms an adapted basis of $\fil^\bullet M^\ast$. Denote 
\[ \wt{\fil}^k M^\ast=\oplus_i A\cdot d^{\max\{0, k-r_i\}}\hat{e}_i\]
it is clear 
\[\wt{\fil}^k M^\ast \subset {\fil}^k M^\ast \]
To prove equality, we use induction on $k$. It is   true for $k \leq 0$ since by Nakayama Lemma, $\hat{e}_i$'s form a basis of $M^\ast$. Suppose the equality holds for the $(k-1)$-th filtration; to prove the case for the $k$-th filtration,   it suffices to prove that the surjection
\[ \wt{\fil}^{k-1} M^\ast/\wt{\fil}^k M^\ast \to  {\fil}^{k-1} M^\ast/{\fil}^k M^\ast \]
is an isomorphism. Note the LHS is finite free over $A/d$ by construction, and the RHS is also finite free over $A/d$ since it is isomorphic to   $\fil^{k-1} M_\HT$ (which is a submodule of a the finite free module $M_\HT$). 
Thus it suffices to show both sides have the same dimension after inverting $p$. On the LHS, the dimension is equal to $d-\dim_F (\fil^{k} D_\dR)$; the dimension on the RHS is exactly  $\dim_F \fil^{k-1} D_\HT$, hence is also $d-\dim_F (\fil^{k} D_\dR)$ (using Lemma \ref{lem: matching graded dR and HT}).
\end{proof}

\subsection{A sub-conjugate filtration} \label{subsec filphi}

In this section, we discuss filtrations related with $\phi(M)$ (instead of $M^\ast$); they will be mainly used in the mod $p$ situation. 

\begin{defn} \label{defn fil phiM}
Use Notation \ref{nota: isogeny}. Suppose furthermore that $\phi(d)=d^p$ (e.g., when $\phi$ is the Frobenius map on a $\fp$-algebra $A$).
\begin{enumerate}
\item The injective map $\phi: M\to M$ factors as
\[ M \xrightarrow{\phi} \phi(M) \into M^\ast \into M\]
where the first map is bijective. 
  Define a ``sub-Nygaard" filtration
\[\fil^n \phi(M):  = \phi(M) \cap \fil^n M^\ast =\phi(M) \cap d^n M\]
\[  \fil^n M:=\{x\in M | \phi(x) \in d^n M \}\]
and thus $\phi$ induces a bijection from $\fil^n M$ to $\fil^n \phi(M)$.

\item For each $i\in \bbz$, define $N_{n,i} \subset M_\HT$ as the image   the composite map:
\[ d^i\fil^{n-i} \phi(M)/ d^i\fil^{n+1-i} \phi(M) \into  \fil^n M^\ast/\fil^{n+1} M^\ast \xrightarrow{d^{-n}, \simeq} \fil_n M_\HT\] 
It is easy to see the equality (inside $M_\HT$):
\[ N_{n,i} =N_{n-i,0}.\]
(Thus in practice, it would suffice  to only consider $N_{\bullet, 0}$'s; but we shall also make use of general $N_{n,i}$-notations). 

\item Note in general $N_{n-1,0}$ and $N_{n,0}$ have no relations, because in general there is no relation between $d\fil^{n-1}\phi(M)$ and $\fil^n \phi(M)$; cf. Remark \ref{rem noflat}. However, since $\phi(d)=d^p$, we have $d^p\fil^{n-p}\phi(M) \subset \fil^n \phi(M)$; thus it is easy to check (cf. Lemma \ref{lem match grade phiM} below):
\[ N_{n-p,0} \subset N_{n,0}\]
As a consequence, after \emph{fixing} an $n$, we can produce a sub-filtration (the ``sub-conjugate filtration" with fixed $n$):
\[ N_{\bullet, [\bullet-n]} \subset \fil_\bullet M_\HT \]
where we use $[x] \in \{0, 1, \cdots, p-1\}$ to denote the residue   modulo $p$; more explicitly, it is of the form:
\[\cdots = N_{n-p-1, p-1} \subset  N_{n-p,0} = N_{n-p+1, 1}=\cdots =N_{n-1, p-1} \subset N_{n, 0} =N_{n+1, 1} =\cdots \]
that is: it can only jump at degrees congruent to $n$ modulo $p$.

\item   Use $\fil^\bullet \phi(M)$ to induce a (decreasing)  filtration $\fil^\bullet_\mathrm{H}$ on $\phi(M/dM)  =\phi(M)/d^p\phi(M)$ (using $\phi(d)=d^p$); one can check
\[ \fil^n_\mathrm{H} \phi(M/dM) =\fil^n \phi(M)/d^p \fil^{n-p}\phi(M)\]
Equivalently, one can use $\fil^\bullet M$ to induce a  filtration $\fil^\bullet_\mathrm{H}$ on $M_\HT=M/dM$; we have
\[ \fil^n_\mathrm{H} M_\HT =\fil^n M/d \fil^{n-p}M\]
\end{enumerate}
\end{defn}

\begin{lemma} \label{lem match grade phiM}
The maps
$\fil^n M \xrightarrow{\simeq, \phi}\fil^n \phi(M) \xrightarrow{d^{-n}} M$  induce  isomorphisms
\[ \fil^n_\mathrm{H} M_\HT/\fil^{n+1}_\mathrm{H} M_\HT \simeq \fil^n_\mathrm{H} \phi(M/dM)/\fil^{n+1}_\mathrm{H} \phi(M/dM) \simeq N_{n,0}/N_{n-p,0}\]
\end{lemma}
\begin{proof}
One can form a  similar   diagram as \eqref{eq: matching of graded}, where all rows and columns are short exact
\[
\begin{tikzcd}
0 \arrow[r] & d^p\fil^{n+1-p} \phi(M) \arrow[r] \arrow[d, hook] & d^p\fil^{n-p} \phi(M) \arrow[r, "d^{-n} "] \arrow[d, hook] & {N_{n-p,0}} \arrow[d, hook] \arrow[r] & 0 \\
0 \arrow[r] & \fil^{n+1} \phi(M) \arrow[d] \arrow[r]            & \fil^n \phi(M) \arrow[d] \arrow[r, "d^{-n} "]              & {N_{n,0}} \arrow[d] \arrow[r]         & 0 \\
0 \arrow[r] & \fil^{n+1} \phi(M/dM) \arrow[r]                   & \fil^n \phi(M/dM) \arrow[r, ""]                                   & \gr \arrow[r]                         & 0
\end{tikzcd}
\]
It is then easy to conclude.
\end{proof}

\begin{remark} \label{rem noflat}
In contrast to the situation in Lemma \ref{lem: flat base change}, the filtrations in Definition \ref{defn fil phiM} in general do  not satisfy flat base change. For example, consider the flat embedding
\[  k[[u]] \into \ocflat \]
which is the mod $p$ reduction of the usual flat embedding $\gs \into \ainf$ (cf. Notation \ref{notafields}). Let $M=k[[u]\cdot e$  with $\phi(e)=e$; denote $\hat{M}=M\otimes_{k[[u]]} \ocflat$.
It is easy to check (using $\phi$ is bijective on $\ocflat$):
\[ \fil^n \phi(M) =u^{p\lceil n/p \rceil} k[[u^p]] \cdot e\]
\[\fil^n \phi(\hat{M}) =u^n \ocflat \cdot e \]
Thus, there is no flat base change, and the inclusion $\fil^n \phi(M) \subset M \cap \fil^n \phi(\hat{M})$ is in general not equality.
\end{remark}

 %%\newpage 
\section{Modules attached to semi-stable representations} \label{sec: BK mod}
In this section, we review three categories of modules attached to semi-stable representations:
\begin{itemize}
\item We review Kisin's $\o$-modules \cite{Kis06} attached to rational semi-stable representations. The $N_\nabla$-operator will be used to construct the \emph{rational filtered Sen operator} in \S \ref{sec: fil sen}.

% We make careful analysis of the $N_\nabla$-operator (certain argument is only  implicit  in \cite{Kis06}); this will be crucially used to  construct the \emph{rational filtered Sen operator} in \S \ref{sec: fil sen rational}.

\item We review Breuil's $S_\ko$-modules \cite{Bre97} and the relations with Kisin's $\o$-modules (constructed in \cite{Liu08}); these $S_\ko$-modules are convenient for dimension computations. In addition, as will be revealed in \S \ref{sub-BeuilN}, they are very closely related with the \emph{shifted Sen operator}.

\item We review Breuil--Kisin $\gk$-modules   \cite{Gao23} attached to integral semi-stable representations. We analyse the $\tau$-operators, which will be used to  construct the \emph{integral (and mod $p$) filtered Sen operator} in \S \ref{sec: fil sen}.

\end{itemize}

\begin{notation} \label{nota: ring for mod}
We recall some rings and elements used throughout the paper.
\begin{enumerate}

\item \label{item 2 for o} Let $\o$ be the ring of analytic functions on the open unit disk defined over $K_0$. Explicitly,
\[
\mathcal{O} = \{f(u)= \sum_{i=0}^{\infty} a_i u^i, a_i \in K_0 \mid   f(u) \text{ converges for all } u \in \mathfrak{m}_{\mathcal{O}_{\overline{K}}}   \},
\] 
  One can extend  $\phi$ on $\gs$ to $\o$.
Define an element
\[
\lambda :=\prod_{n \geq 0} (\varphi^n(\frac{E(u)}{E(0)}))  \in \o.\]  
Define an operator  $N_\nabla:=-u\lambda\frac{d}{du}$ on $\o$. 

\item Let $S$ be  the $p$-adic completion of the PD envelope of $\gs$ with respect to the ideal  $(E(u))$ (cf.  Notation \ref{notafields}). The Frobenius $\phi$ on $\gs$ extends to $S$. Define the operator $N_S=-u\frac{d}{du}$   on $S$. Let  $\Fil ^j S\subset S $ be the $p$-adic completion of the ideal generated by $\gamma_i (E(u)):= \frac{E(u)^i}{i!}$ with $ i \geq j$. Let $S_{K_0}=S[\frac 1 p]$, and extend $\varphi, N_S$ actions on $S$ to $S_{K_0}$ ($\Qp$-linearly).
Let $\Fil^i S_{K_0}:= \Fil^i S\otimes_{\Zp}\Qp$.

 \item Recall $[\zetaflat]$ is defined in  Notation \ref{notafields}. Let $\xi=\frac{[\zetaflat]-1}{\phi^{-1}([\zetaflat]-1)}$. 
Let $t=\log([\zetaflat])  \in \bcrisplus$ be the usual element. 
Define the element (with quotient taken inside $\bdrplus$),
$$ \mathfrak{t} = \frac{t}{p\lambda}=\frac{p\phi^{-1}([\zetaflat]-1) \prod_{n \geq 0} \phi^n(\xi/p)}{p\prod_{n \geq 0} (\varphi^n(\frac{E(u)}{E(0)}))} \in \ainf.$$
(Recall we use $\theta_\fon$ to denote Fontaine's map, cf. Convention \ref{conv: theta notation}). We have:
\[v_p(\theta_\fon(\fkt))= v_p(\zeta_p-1)=\frac{1}{p-1}\]  

\item Following \cite[\S 5.1]{Fon94}, for any $\phi$-stable subring $A \subset \bcris$, define 
\[ I^{[1]}A :=\{ a\in A; \theta_\fon(\phi^n(a)) =0, \forall n \geq 0 \} \]
Then $I^{[1]}\ainf$ is a principal ideal by \cite[Proposition 5.1.3]{Fon94},  with $\phi(\fkt)$ as a generator by \cite[Lemma 3.2.2]{Liu10}. Denote
\[ w= \frac{[\zeta^\flat]-1}{E} \in \ainf \]
Then it is easy to check that $\phi(w)$ is also a generator of $I^{[1]}\ainf$. Thus ${\fkt}/{w}$ is a unit in $\ainf$.
\end{enumerate}
\end{notation}

 \begin{construction}[{cf. \cite[\S 4.1]{Gao23}}] \label{cons: loc ana nnabla}
 Consider the (well-known) perfect Robba ring $\wtb^\dagger_\rig$ (as in \cite{Ber02}); denote 
\[ \wtb^\dagger_{\rig, L}:= (\wtb^\dagger_\rig)^{G_L},\]
which is a LF representation of the Lie group  $\hat{G}$.
Use $(\wtb^\dagger_{\rig, L})^{\hat{G}\dpa}$ to denote the set of pro-locally analytic vectors, which admit $\nabla_\tau$-action (cf. Notation \ref{nota lie hatg}). By \cite[Lem. 5.1.1]{GP21}, $\fkt$ is a unit inside $(\wtb^\dagger_{\rig, L})^{\hat{G}\dpa}$, thus we can define a (normalized) operator
\begin{equation}
N_\nabla: =-\frac{1}{p\fkt}\nabla_\tau
\end{equation}
This still acts on $(\wtb^\dagger_{\rig, L})^{\hat{G}\dpa}$. One checks that it is stable on the subring $\o$; indeed,  \cite[Lemma 4.1.2]{Gao23} (with Convention \ref{conv minus plus} in mind) shows that it coincides with   $-u\lambda\frac{d}{du}$ as in Notation \ref{nota: ring for mod}\eqref{item 2 for o} (whence the coincidence of notation).
 \end{construction}

\subsection{Breuil's $S_\ko$-modules and Kisin's   $\mathcal{O}$-modules} \label{subsecKis}
\begin{defn}\label{deffilnmod}
Let $\MF$ be the category of
 \emph{(effective) filtered $(\varphi,N)$-modules over $K_0$} which consists of  finite dimensional $K_0$-vector spaces $D$ equipped with
\begin{enumerate}
\item an injective Frobenius $\varphi: D \to D$ such that $\varphi(ax)=\varphi(a)\varphi(x)$ for all $a \in K_0, x \in D$;
\item a monodromy $N: D \to D$, which is a $K_{0}$-linear map such that $N\varphi=p\varphi N$;
\item a filtration $(\Fil^{i}D_{K})_{i\in\mathbb{Z}}$ on $D_{K}=D\otimes_{K_0} K$, by decreasing $K$-vector subspaces such that $\Fil^{0}D_{K}= D_K$  and $\Fil^{i}D_{K}=0$ for $i \gg 0$.
\end{enumerate} 
\end{defn}

\begin{defn}[\cite{Kis06}] \label{def: Kis o mod}
 Let $\textnormal{Mod}_{ \mathcal{O}}^{\varphi, N_\nabla }$ be the category   consisting of finite free $\mathcal{O}$-modules $\cm$ equipped with
\begin{enumerate}
\item a $\varphi_{\mathcal O}$-semi-linear morphism $\varphi : \cm \to \cm$ such that   the cokernel of $1 \otimes \varphi : \varphi ^*\cm \to \cm $ is killed by $E(u)^h$ for some $h \in \mathbb{Z}^{\geq 0}$;
\item $N_\nabla: \cm \to \cm$ is a  map such that $N_\nabla(fm)=N_\nabla(f)m+fN_\nabla(m)$ for all $f\in \mathcal O$ and $m \in \cm$, and $N_\nabla\varphi=\frac{pE(u)}{E(0)} \varphi N_\nabla$.
\end{enumerate}
\end{defn}

\begin{defn}[\cite{Bre97}] \label{def: breuil mod}
Let $\bigMF$ be the category whose objects are finite free $S_{K_0}$-modules $\cd$ with:
\begin{enumerate}
\item a $\varphi_{S_{K_0}}$-semi-linear morphism $\varphi_{\cd}: \cd \to \cd$ such that the determinant of $\varphi_{\cd}$ is invertible in $S_{K_0}$;
 \item a decreasing filtration $\{\Fil^i\cd\}_{i=0}^{\infty}$ of $S_{K_0}$-sub-modules of $\cd$ such that $\Fil^0\cd=\cd$ and $\Fil^i S_{K_0} \Fil^j \cd \subseteq \Fil^{i+j}\cd$;
 \item a $K_0$-linear map $N: \cd \to \cd$ such that $N(fm)=N(f)m+fN(m)$ for all $f\in S_{K_0}$ and $m \in \cd$, $N\varphi=p \varphi N$ and $N (\Fil^i \cd) \subseteq \Fil^{i-1}\cd$.
\end{enumerate} 
\end{defn}

\begin{construction}[\cite{Bre97}] \label{ConstructionD to cd}
For $D \in \MF$, we can associate an object in $\bigMF$ by $\cd:= S_\ko\otimes_{\ko}D$ and
 \begin{itemize}
 \item $\varphi: =\varphi_S \otimes \varphi_D$;
 \item $N:= N\otimes Id + Id\otimes N$;
 \item $\Fil^0\cd :=\cd$ and inductively,
 $$\Fil^{i+1}\cd := \{ x \in \cd | N(x) \in \Fil^i \cd \text{ and } f_{\pi}(x) \in \Fil^{i+1}D_K  \},$$
 where $f_{\pi}: \cd \twoheadrightarrow D_K$ by $s(u)\otimes x \mapsto s(\pi)\otimes x$.
 \end{itemize}
\end{construction}

\begin{construction}[{cf. \cite[\S 3.2]{Liu08}}] \label{Constructioncm to cd}
Given $\cm \in \textnormal{Mod}_{ \mathcal{O}}^{\varphi, N_\nabla}$, we can associate an object in $\bigMF$ by $\cd:=S_\ko \otimes_{\phi, \o} \cm$ together with 
\begin{itemize}
\item $\phi_\cd=\phi_S \otimes \phi_{\cm}$
\item $N_\cd =N_S\otimes 1+ \frac{p}{\phi(\lambda)}1\otimes N_\nabla$.
\item $\fil^i \cd = \{ m \in \cd | (1\otimes \phi)(m) \subset \fil^i S_\ko \otimes_\o \cm\}$.
\end{itemize}
\end{construction}

\begin{theorem}[{cf. \cite{Bre97}, \cite{Kis06}, and \cite[\S 3.2]{Liu08}}] 
    The functors in Constructions \ref{ConstructionD to cd} and \ref{Constructioncm to cd} induce a  diagram
   \[ \begin{tikzcd}
\MFwa \arrow[d, hook] \arrow[rr, "\simeq"] &  & \bigMFwa \arrow[d, hook] &  & {\textnormal{Mod}_{ \mathcal{O}}^{\varphi, N_\nabla, 0}} \arrow[ll, "\simeq"'] \arrow[d, hook] \\
\MF \arrow[rr, "\simeq"]                   &  & \bigMF                   &  & {\textnormal{Mod}_{ \mathcal{O}}^{\varphi, N_\nabla }} \arrow[ll, "\simeq"']                  
\end{tikzcd}\]
where all horizontal arrow are equivalence of categories. Here $\MFwa \subset \MF$ is the subcategory of weakly admissible objects; $\textnormal{Mod}_{ \mathcal{O}}^{\varphi, N_\nabla, 0}$ is the subcategory of $\textnormal{Mod}_{ \mathcal{O}}^{\varphi, N_\nabla}$ consisting of objects whose base change to the Robba ring is pure of slope $0$  in the sense of Kedlaya  (cf. \cite[\S 1.3]{Kis06}); and $\mathrm{MF}_{S_\ko}^{\phi, N, \mathrm{wa}}$ is the essential image of $\MFwa$ under the equivalence $\MF \simeq \bigMF$. 
\end{theorem}

 %equivalences of categories:  \[ \MF \xrightarrow{\simeq} \bigMF \xleftarrow{\simeq} \textnormal{Mod}_{ \mathcal{O}}^{\varphi, N_\nabla}\]
%Then we have equivalences of c ategories: \[ \MFwa \xrightarrow{\simeq} \bigMFwa \xleftarrow{\simeq} \textnormal{Mod}_{ \mathcal{O}}^{\varphi, N_\nabla, 0}\]

We record some facts on $N_\nabla$ for future use.
\begin{lemma}  \label{lem: Kis nyg fil}
 Following  Notation \ref{nota: isogeny} applied the triple $(\o, E, \phi)$,  regard $\cmast$ as a submodule of   $\cm$. We have: 
    \begin{enumerate}
        \item \label{kis1} $N_\nabla(\cmast) \subset E\cmast$;
        \item $ N_\nabla(E^n \cm) \subset  E^n \cm$;
        \item \label{kis2}  $N_\nabla(\fil^n \cmast) \subset E\fil^{n-1} \cmast$.
    \end{enumerate} 
\end{lemma}
\begin{proof}
This is (easy) strengthening of argument in  \cite[Lem. 1.2.12]{Kis06}. Using the relation $N_\nabla\varphi=\frac{pE(u)}{E(0)} \varphi N_\nabla$ on $\cm$, it is easy to see $N_\nabla(\cmast) \subset E\cmast$.
It is also easy to check (as in \emph{loc. cit.}) that $ N_\nabla(E^n \cm) \subset  E^n \cm$. Thus $N_\nabla(\fil^n \cmast)$ is contained in $E\cmast \cap E^n \cm=E\fil^{n-1} \cmast$. 
\end{proof}

%\begin{remark} \label{rem:twisted-GT} The ``twisted Griffiths transversality"  \[ N_\nabla(\fil^n \cmast) \subset E\fil^{n-1} \cmast\]    is \emph{stronger} than a Griffiths transversality, with the extra $E$-twisting. We shall see in Theorem \ref{thm: rational sen shift} that this is exactly the ``origin" of the ``1-degree shrinking" for the Sen operator. A similar (twisted $p$-Griffiths transversality) phenomenon happens in the mod $p$ case, cf. Proposition \ref{prop-pGT-A-level}.  \end{remark}

\begin{remark} \label{rem cm is lav}
Given $\cm \in \textnormal{Mod}_{ \mathcal{O}}^{\varphi, N_\nabla, 0}$, \cite[Theorem 5.3.4]{Gao23} implies that 
\[ \cm \otimes_\o \bfb^\dagger_{\rig, \kinfty}  \simeq D^\dagger_{\rig, \kinfty}(V) \]
where the right hand side is the rigid-overconvergent $(\phi, \tau)$-module associated to $V$. (cf. loc. cit.  for unfamiliar notations). 
In addition, loc. cit. implies that the $N_\nabla$-operator on $\cm$ is \emph{coincides} with the locally analytic (i.e., Lie algebra theoretic) operator $N_\nabla$ that we constructed in Construction \ref{cons: loc ana nnabla}.
\end{remark}

\subsection{Integral semi-stable representations and Breuil--Kisin $\gk$-modules}
 
  Recall an effective Breuil--Kisin module  is an effective isogeny (cf. Notation \ref{nota: isogeny}) with respect to the tripe $(\gs, E, \phi)$.
  
\begin{defn}\label{defwr11}
Let $\textnormal{Mod}_{\gs, \ainf}^{\varphi, G_K, \log-\crys}$
be the category consisting of triples $(\gm, \varphi_{\gm}, G_K)$, which we call the (effective) \emph{Breuil--Kisin $G_K$-modules}, where
\begin{enumerate}
\item $(\gm, \varphi_\gm)$ is an effective  finite free Breuil--Kisin module;
\item  $G_K$ is a continuous $\varphi_{\minf}$-commuting $\ainf$-semi-linear   $G_K$-action on $\minf:= \ainf
\otimes_{ \gs} \gm$, such that
\begin{enumerate}
\item $\gm \subset (\minf)^{G_\kinfty}$ via the embedding $\gm \hookrightarrow \minf$;
\item  $\gm/u\gm \subset (\minf/W(\mathfrak{m}_{\o_C^\flat})\minf)^{G_K}$ via the embedding $\gm/u\gm \hookrightarrow \minf/W(\mathfrak{m}_{\o_C^\flat})\minf$.
\end{enumerate} 
\end{enumerate}
Let $\textnormal{Mod}_{\gs, \ainf}^{\varphi, G_K, \crys}$
be the sub-category consisting of objects  such that $(g-1)(\gm) \subset \fkt W(\fkm_\ocflat) \gm_\inf$ for all $g\in \gk$ (cf. \cite[Prop. 7.1.10]{Gao23}).
  \end{defn}
 
 \begin{theorem}[\emph{\cite[Theorem 1.1.11]{Gao23}}]
 Let $\star \in \{\log, \emptyset\}$. 
 The functor sending $(\gm, \gm_\inf)$ to $(\gm_\inf \otimes_\ainf W(C^\flat))^{\phi=1}$ induces an equivalence
\[ \mathrm{Mod}_{\gs, \ainf}^{\phi, \gk, \star-\crys}  \simeq \rep_\zp^{\star-\crys, \geq 0}(\gk) \]
 \end{theorem}

We now study bounds on the range of $(\tau-1)^i$-actions; these can be regarded as integral counterparts of Lemma \ref{lem: Kis nyg fil}. We start with an easy lemma.

\begin{lemma} \label{lem:ring ainf} 
Let $a,b,c \geq 0$, then
\begin{equation*} \label{eq:u fkt cap E}
u^a \ainf \cap \fkt^b \ainf \cap E^c \ainf =u^a\fkt^bE^c \ainf
\end{equation*}  
\end{lemma}
\begin{proof}
 Using the fact that $\phi(\fkt)$ is a generator of the ideal $I^{[1]}\ainf$, cf. Notation \ref{nota: ring for mod}, one can easily check 
\begin{equation*} \label{eq:u fkt cap}
u^a \ainf \cap \fkt^b \ainf =u^a\fkt^b \ainf.
\end{equation*}
 Then one can   conclude using   that  $E$ is a generator of $\ker \theta_\fon$ on $\ainf$.
 \end{proof}

\begin{lemma} \label{lem: tau range gmast}
  Let $ (\gm, \gm_\inf) \in \textnormal{Mod}_{\gs, \ainf}^{\varphi, G_K, \log-\crys }$. Let $i \geq 1, n \geq 0$.  
  \begin{enumerate}
  \item  We have
\begin{eqnarray}
\label{taugm}   (\tau -1)^i (\gm) & \subset &  \fkt^i \minf \\
\label{taugmast}  (\tau -1)^i (\gmast) &\subset &   E\fkt^i \gm_\inf^\ast \\
\label{tauEgm}  (\tau-1)^i( E^n \gm) &\subset  &\fkt^i E^n \gm_\inf \\
\label{taufilgm}  (\tau-1)^i(\fil^n \gmast) &\subset & \fkt^i E \fil^{n-1}\gm_\inf^\ast  
\end{eqnarray}  
   
  \item If $(\gm, \gm_\inf)$ is furthermore crystalline, then 
  \begin{eqnarray}
(\tau -1)^i (\gm) & \subset &  u\fkt^i \minf \\
(\tau -1)^i (\gmast) &\subset &  uE\fkt^i \gm_\inf^\ast \\
(\tau-1)^i( E^n \gm) &\subset  &u\fkt^i E^n \gm_\inf \\
(\tau-1)^i(\fil^n \gmast) &\subset & u\fkt^i E \fil^{n-1}\gm_\inf^\ast  
\end{eqnarray}   
  \end{enumerate}
\end{lemma}
\begin{proof}
  Consider Eqn \eqref{taugm} in the semi-stable case. When $i=1$, this is proved in Step 1 of \cite[Prop. 7.1.10]{Gao23}; the general case follows from similar argument.  Indeed, by \cite[Lem. 7.1.9]{Gao23}, it suffices to show that 
\[ (\tau-1)^i (\gm) \subset \gm \otimes_\gs \fkt^i \wtb^{[0, \frac{r_0}{p}]} \]
Note, 
 \[ (\tau-1)^i (\gm)=(\sum_{j \geq 1} \frac{\nabla_\tau^j}{j!})^i(\gm). \]
It is already proven in \cite[Eqn. (7.1.18)]{Gao23} that
\[\nabla_\tau^j(\gm) \subset \gm\otimes_\gs \fkt^j \cdot \o \]
thus each summand of above summation falls inside $\gm\otimes_\gs \fkt^i \cdot \o$.
Consider Eqn. \eqref{taugmast}: by \eqref{taugm} and Lemma \ref{lem:ring ainf}, it reduces to prove
\[ (\tau-1)^i(\gmast) \subset E\gm_\inf^\ast\]
It suffices to treat $i=1$ case as other cases follow by induction; but it reduces to the fact that
\[ (\tau-1)(\gs) \subset \phi(\fkt) \ainf, \text{ and } (\tau-1)(\phi(\gm)) \subset  \phi(\fkt)\gm_\inf^\ast\]
where the second inclusion follows from Eqn \eqref{taugm}. For Eqn. \eqref{tauEgm}, similarly, it suffices to note that the ideal $(E)=\ker \theta_\fon$ is $\gk$-stable, and by induction,
\[ (\tau-1)^i(E^n\gm) \subset E^n\minf.\] 
Finally, Eqn. \eqref{taufilgm} follows by combining \eqref{taugmast} and \eqref{tauEgm}.

 Consider  the crystalline case in Item (2).
With the semi-stable case in hand, and using Lemma \ref{lem:ring ainf}, it remains  to prove that in the crystalline case, we further have:
 \[  (\tau-1)^i (\gm) \subset  u \minf \]
 Since $N_\nabla(\gm) \subset \gm\otimes_\gs \o$, and since $N_\nabla/uN_\nabla=N_{D_\log(V)}=0$, we must have
\[ N_\nabla(\gm) \subset \gm\otimes_\gs u\cdot \o \]
Thus,
\[\nabla_\tau(\gm) \subset \gm\otimes_\gs u\fkt \cdot \o \]
One can check $\nabla_\tau(u\fkt^k) \subset u\fkt^{k+1}\cdot \o$; thus inductively, we can show that
\[\nabla_\tau^j(\gm) \subset \gm\otimes_\gs u\fkt^j \cdot \o \]
Then we can use the same argument as in semi-stable case to conclude. 
\end{proof}

We record the following crystalline criterion (as an   addendum to \cite[Prop. 7.1.10]{Gao23}).

\begin{cor} \label{prop: crys criterion}
Let $T$ be a semi-stable representation, with $\gm$ the associated Breuil--Kisin module. 
The following are equivalent.
\begin{enumerate}
\item $(\tau -1)^i (\gm) \subset  u\fkt^i \minf, \forall i \geq 1$
\item   $(\tau -1) (\gm) \subset  u\fkt  \minf$
\item   $(\tau -1) (\gm) \subset \fkt W(\fkm_\ocflat) \minf$
\item $T$ is crystalline.
\end{enumerate}
\end{cor}   
\begin{proof}
Obviously, $(1) \Rightarrow (2) \Rightarrow (3)$. The equivalence of (3) and (4) is proved in \cite[Prop. 7.1.10]{Gao23}.   The implication from (4) to (1) is proved in Lem. \ref{lem: tau range gmast}. 
\end{proof}

 \subsection{Relation of Nygaard filtrations}
 
 Let $T$ be an integral semi-stable representation, and $V=T[1/p]$. Let $\gm, \cm, \cd$ etc. be the associated modules from above subsections.
    Construction \ref{Constructioncm to cd} implies $\cd=S_{\ko}\otimes_\o \cmast$; also note the map $\cmast \into \cm$ induces $\cd \into S_{\ko}\otimes_\o \cm$.

\begin{lemma} \label{lem: strict morph of mod}
    We have a commutative diagram of morphisms of filtered modules
    \[
    \begin{tikzcd}
{(\gmast, \fil^i \gmast)} \arrow[d, hook] \arrow[rr, hook] &  & {(\cmast, \fil^i \cmast)} \arrow[d, hook] \arrow[rr, hook] &  & {(\cald , \fil^i \cald )} \arrow[d, hook] \\
{(\gm, E^i\gm)} \arrow[rr, hook]                           &  & {(\cm, E^i\cm)} \arrow[rr, hook]                           &  & {(S_{\ko}\otimes_\o \cm, \fil^i S_{\ko} \otimes_\o \cm)}               
\end{tikzcd}
\]
where all arrows are strict (as filtered morphisms). Furthermore, both squares are Cartesian squares (of filtered objects).
\end{lemma}
\begin{proof}
    To check the arrows on the bottom row are strict, it reduces to check the morphisms of filtered rings
    \[ (\gs, E^i\gs) \into (\o, E^i\o) \into (S_{\ko}, \fil^i S_{\ko}) \]
are strict; the proof is standard, and  is omitted.
Consider the vertical arrows: the left two arrows are strict by definition; the right vertical arrow is strict by Construction \ref{Constructioncm to cd}.
It now suffices to prove both squares are Cartesian squares (of filtered objects), as this will imply strictness on top row. In fact, since we already know all vertical arrows are strict, it remains to prove the unfiltered version. That is, we only need to prove
\[ \cmast \cap \gm = \gmast, \quad \cald^\ast \cap \cm  =\cmast\] 
Consider the first case (i.e., the left square).  Let $\vec e$ be basis of $\gm$, and $\phi(\vec{e})=A\vec{e}$ where $A$ is matrix over $\gs$; let $B$ be the matrix such that $AB=E^h$.
    An element in $ \cmast \cap \gm$ is of the form
    \[ \sum a_i \phi(e_i) =(a_1, \cdots, a_d)A\vec{e} \]
    where $a_i \in \o$ and $(a_1, \cdots, a_d)A \in \gs$; multiply by $B$, we see $(a_1, \cdots, a_d)E^h \in \gs$. Using $E^h\o \cap \gs=E^h\gs$, we see $a_i\in \gs$; this implies $\cmast \cap \gm=\gmast$. 
    The case $\cald^\ast \cap \cm=\cmast$ can be similarly proved using $E^h S_{\ko} \cap \o=E^h\o$. 
\end{proof}
 
 \begin{lemma} \label{lem: compa of graded}
We have
\begin{enumerate}
\item $ \gr^i \gm^\ast \into  \gr^i \cm^\ast \into  \gr^i \cald.$ (Caution: the first injection is not strict with respect to the inclusion $\gm_\HT \into \cm_\HT$).
\item $ \gr^i \gm^\ast \otimes_\ok K =\gr^i \cm^\ast = \gr^i \cald$, with dimension (over $K$) equal to $ d-\dim_K(\fil^{i+1} D_K)$.
\end{enumerate}
\end{lemma}  
\begin{proof} 
 Lem. \ref{lem: strict morph of mod} implies (1).
The equality of spaces in Item (2) follows from  \cite[Lem. 4.3(3)]{GLS14}; the dimension formula follows from the fact that $\cald$ has adapted basis (proved in \cite{Bre97}), cf. end of proof in \cite[Proposition 4.5]{GLS14}. (The two cited results from \cite{GLS14} are valid for any $K$ and for all semi-stable representations). (Note the embedding $\gs \to \o$ is not flat; thus one cannot apply Lemma \ref{lem: flat base change} to study these questions.)
\end{proof}

\begin{rem} \label{rem: no cdht}  The   discussions in this section prompt the question that if one can define a certain ``conjugate filtration" related to $\cd$. The diagram in Lemma \ref{lem: strict morph of mod} suggests that this filtration can only be defined on $(S_{\ko}\otimes_\o \cm)/(\fil^1 S_{\ko}\otimes_\o \cm)$, which is the same as $\cm_\HT$; in addition, Lemma \ref{lem: compa of graded} tells us $\gr^i\cmast=\gr^i\cd$, so the  ``conjugate filtration" can only be exactly the same as the one on $\cm_\HT$.  
\end{rem}

 %%\newpage
  \section{Integral Sen theory  for semi-stable representations} \label{sec: integral Sen}
  
We first review Sen theory over the Kummer tower (Theorem \ref{thmkummersenop})  for general $C$-representations. 
When the $C$-representation comes from an integral semi-stable representation, we upgrade this Sen theory to the $K$-rational level in Proposition \ref{prop: K ratitonal sen}, and finally to the $\ok$-integral level in  Theorem \ref{thmintsen}.

 \subsection{Sen theory over the Kummer tower} \label{sec: sen kummer}

Let $\rep_\gk(C)$ be the category of $C$-representations; an object is a finite dimensional $C$-vector space with a continuous semi-linear $\gk$-action.

 \begin{construction}\label{notaSenop}
 Let $W\in \rep_\gk(C)$ of dimension $d$. Define
 \begin{equation}\label{senlav}
D_{\Sen, \kpinfty}(W): =(W^{G_\kpinfty})^{\gammak\dla};
\end{equation}
cf. \S \ref{subsec notation} for notation $\kpinfty$ and the notion of locally analytic vectors. 
By \cite{Sen81} and then reformulated in \cite{BC16} using locally analytic vectors,  this is a $\kpinfty$-vector space of dimension $d$, such that the natural map
\[ D_{\Sen, \kpinfty}(W)\otimes_\kpinfty C \to W\]
is an isomorphism.
Thus, the  operator $\nabla_\gamma$ in Notation \ref{nota lie hatg}   induces an operator
\begin{equation}\label{eqsenclassical}
\nabla_\gamma: D_{\Sen, \kpinfty}(W) \to D_{\Sen, \kpinfty}(W).
\end{equation}
This is called the \emph{Sen operator}: it is $\kpinfty$-linear because $\nabla_\gamma$ kills $\kpinfty$.
 Extending  $C$-linearly, we obtain a $C$-linear operator
\begin{equation}\label{eqsenextendtoc}
\nabla_\gamma: W \to W;
\end{equation}
we still call it the \emph{Sen operator}.
The eigenvalues of the Sen operator are called Sen weights.
  \end{construction}

\begin{theorem}[{\cite[Theorem 7.12]{GMWHT}}] \label{thm331kummersenmod}  
Let  $W\in \rep_\gk(C)$. Define
 \begin{equation*}
 D_{\Sen, \kinfty}(W):= (W^{G_L})^{\tau\dla, \gamma=1}.
 \end{equation*} 
 Here, the right hand side   denotes the subset of  $\gal(L/\kpinfty)$-locally analytic vectors that are furthermore fixed by $\gal(L/\kinfty)$.
 Then it is a $\kinfty$-vector space, and the natural map $$ D_{\Sen, \kinfty}(W) \otimes_\kinfty C \to W$$  
 is an isomorphism. 
 \end{theorem}

\begin{construction} \label{Constructionla Sen op}
Consider $\fkt$ introduced in Notation \ref{nota: ring for mod}. By \cite[7.4]{GMWHT}, $\theta_\fon(\fkt)$  is a unit in $\hat{L}^{\hat{G}\dla}$. Similar to Construction \ref{cons: loc ana nnabla}, one can define
\begin{equation}
N_\nabla: =-\frac{1}{p\theta_\fon(\fkt)}\nabla_\tau
\end{equation}
which acts on $\hat{L}^{\hat{G}\dla}$, and indeed any $\hat{G}$-locally analytic representations over  $\hat{L}^{\hat{G}\dla}$. Thus,   there is an operator
\begin{equation} \label{eqn nnabla lhatla}
N_\nabla: D_{\Sen, \kinfty}(W) \otimes_\kinfty  \hat{L}^{\hat{G}\dla} \to D_{\Sen, \kinfty}(W) \otimes_\kinfty  \hat{L}^{\hat{G}\dla}
\end{equation} 
\end{construction} 
 
\begin{theorem}[{\cite[Theorem 7.13]{GMWHT}}]\label{thmkummersenop}
  After linear scaling, the operator in Eqn \eqref{eqn nnabla lhatla} induces a $\kinfty$-linear operator, which we call the \emph{negative Sen operator over the Kummer tower}
\begin{equation}\label{eqnnablanorm}
\frac{1}{\theta_\fon(u\lambda')}\cdot N_\nabla: D_{\Sen, \kinfty}(W) \to D_{\Sen, \kinfty}(W).
\end{equation} 
(Here: $\lambda'=\frac{d}{du}\lambda$.)
Extend it $C$-linearly to a $C$-linear operator on $ D_{\Sen, \kinfty}(W) \otimes_\kinfty C=W$, and denote it by the same notation:
\begin{equation}
\frac{1}{\theta_\fon(u\lambda')}\cdot N_\nabla:  W \to W
\end{equation}
Then this is precisely the \emph{negative} of the (uniquely defined) \emph{Sen operator} in Eqn. \eqref{eqsenextendtoc}.
\end{theorem} 
%In particular, The two operators $\frac{1}{\theta(u\lambda')}\cdot N_\nabla$ in Eqn. \eqref{eqnnablanorm} and $\nabla_\gamma$ in Eqn. \eqref{eqsenclassical} have the same eigenvalues, and have the same semi-simplicity property.

 \begin{remark} \label{rem: negative Sen op}
We shall use this \emph{negative} Sen operator in the rest of the paper. cf. Convention \ref{conv: ht and sen effective}.
 For brevity, we also  simply call it the negative $\kinfty$-Sen operator, and denote it by (cf. Convention \ref{conv: theta notation}),
  $$\theta_\kinfty:=\frac{1}{\theta_\fon(u\lambda')}N_\nabla$$
 \end{remark}

   \subsection{$K$-rational Sen operator}
\begin{notation} \label{nota: rep and mod}
\begin{enumerate}
    \item Let $\star \in \{ \emptyset, \log \}$.  Let $T$ be a  $\star$-crystalline $\zp$-representation, and let $V=T\otimes_\zp \qp$.
    Suppose the Hodge--Tate weights (cf. Convention \ref{conv: ht and sen effective}) of $V$ are $0 \leq r_1 \leq \cdots \leq r_d$.
  %  Thus the Sen weights associated to $V\otimes_\qp C$ are the non-positive numbers $-r_i$,.
    \item Let $(\gm, \gm_\inf)$ be the (effective) Breuil--Kisin $\gk$-module associated  to $T$. Let $\cm$ be the $\o$-module associated to $V$. 
   We consider $\gm$ resp. $\gminf$ resp. $\cm$ as an effective isogeny (Notation \ref{nota: isogeny}) over the triple
\[ (\gs, E, \phi), \quad \text{resp. } (\ainf, E, \phi), \quad  \text{resp. } (\o, E, \phi). \]      
    To   these modules, one can apply terminologies and notations in Definition \ref{defn: conj fil}.

%Note $(\gm/E)[1/p]=\cm/E$.
\end{enumerate} 
\end{notation}

\begin{prop} \label{propgmsenmod}
Use Notation \ref{nota: rep and mod}. There are $\gk$-equivariant isomorphisms
\begin{equation}\label{eqmodetenc}
\gm/E\otimes_\ok C =\gm_\inf\otimes_{\ainf} C \simeq T\otimes_\zp C 
\end{equation}
In addition,
\begin{equation}\label{eqgmsenmod}
\gm/E\otimes_\ok \kinfty = D_{\Sen, \kinfty}(T\otimes_\zp C)
\end{equation}
and
\begin{equation}\label{eqgmsenmod: cm ver}
\cm/E\otimes_K \kinfty = D_{\Sen, \kinfty}(T\otimes_\zp C)
\end{equation}
\end{prop}
\begin{proof}
By \cite[Prop. 3.1.3]{Liu10}, the $\gk$-equivariant isomorphism
\[ (\gm\otimes_\gs \ainf)\otimes_{\ainf} W(\cflat) \simeq T\otimes_\zp  W(\cflat) \]
can be refined as a $\gk$-equivariant isomorphism
\[ (\gm\otimes_\gs \ainf)\otimes_{\ainf} \ainf[\frac 1 \fkt] \simeq T\otimes_\zp   \ainf[\frac 1 \fkt] \]
One can modulo $E$ to get \eqref{eqmodetenc}. 

To prove \eqref{eqgmsenmod} (which is the same as \eqref{eqgmsenmod: cm ver}), it suffices to verify that the $\tau$-action on $\gm/E$ is locally analytic; but as reviewed in Rem. \ref{rem cm is lav}, the $\tau$-action on $\gm$ is already locally analytic. ((Alternatively, one can use the explicit formula for $\tau$-action in Construction \ref{cons:GMWHT-tau} from \cite{GMWHT}).)
\end{proof}

 \begin{proposition}[$K$-rational Sen operator]\label{prop: K ratitonal sen} 
Use notations in Prop. \ref{propgmsenmod}. Via \eqref{eqgmsenmod: cm ver}, the negative $\kinfty$-Sen operator (cf. Rem. \ref{rem: negative Sen op})  induces an endomorphism
\[ \theta_\kinfty: \cm/E\otimes_K \kinfty \to \cm/E\otimes_K \kinfty\]
This operator stabilizes $\cm_\HT =\cm/E=(\gm/E)[1/p]$, inducing a $K$-linear operator
\[\theta_\kinfty: \cm_\HT \to   \cm_\HT   \] 
In addition, the operator $\theta_\kinfty$ is semi-simple with   eigenvalues $r_1, \cdots, r_d$. 
 \end{proposition}
 \begin{proof}
Note that Kisin's operator $N_\nabla$ stabilizes $\cm$ hence also $\cm/E$. As discussed in  Cons. \ref{cons: loc ana nnabla} and Rem. \ref{rem cm is lav}, Kisin's operator coincides with the normalized Lie algebra operator, which remains so modulo $E$: that is to say, Kisin's operator modulo $E$ coincides with the operator from Cons. \ref{Constructionla Sen op}.
 Finally, $\theta_\kinfty$ is semisimple because  $T[1/p]$ is Hodge--Tate; its eigenvalues are negative Sen weights by Theorem \ref{thmkummersenop}, which are the Hodge--Tate weights, cf. Convention \ref{conv: ht and sen effective}.
 \end{proof}

\subsection{Amplified integral Sen operator}\label{subsec-cons-theta}

In this subsection, we construct an \emph{integral} Sen operator attached to $T$.

We first record a lemma which axiomatizes an argument repeatedly used in the paper.  
 Recall  
 \[ v_p(\theta_\fon(\frac{\fkt^i}{i \fkt}))=\frac{i-1}{p-1}- v_p(i) \geq 0 \text{ and goes to infinity as } i \to \infty. \] 
The following   lemma is an easy consequence of the above fact. 

\begin{lemma} \label{lem: log tau converge}
Let $b \in \ainf$.
\begin{enumerate}
\item Let $x\in \gm$. Let $Y \subset \gm_\inf$ be an additive closed  subset.
Suppose for each $i\geq 1$,
\begin{equation} \label{eq: cond fkti}
 (\tau-1)^i(x) \in \fkt^ib Y.
\end{equation}
Then the element $\frac{1}{b} \cdot \frac{(\tau-1)^i}{i\fkt}(x) \pmod{E}$ in $(\gminf/E)[1/p]$ lands inside $\gminf/E$. In addition, the summation
\[ \frac{1}{b}\cdot \frac{\log \tau}{\fkt}(x) \pmod{E} 
:= \sum_{i=1}^\infty \frac{1}{b} \cdot \frac{-(1-\tau)^i}{i\fkt}(x) \pmod{E} \]
 converges   inside $\gm_\inf/E$, and falls inside the image of
 \[ Y \to \gm_\inf/E\]
 
 \item  Let $z\in \bargm$. Let $W \subset \bargm_\inf$ be an additive closed  subset.
Suppose for each $i\geq 1$,
\begin{equation} \label{eq: cond fktimodp}
 (\tau-1)^i(z)= \fkt^i b w_i \in \fkt^i b W.
\end{equation}
Define the expression $\frac{1}{b} \cdot \frac{(\tau-1)^i}{i\fkt}(z) \pmod{E}$ as  $\theta_\fon(\frac{\fkt^{i-1}}{i})\left (w_i \pmod{E} \right)$,  which is a well-defined element in $\bargminf/E$. Then the summation 
 \[ \frac{1}{b}\cdot \frac{\log \tau}{\fkt}(z) \pmod{E} 
:= \sum_{i=1}^\infty \frac{1}{b} \cdot \frac{-(1-\tau)^i}{i\fkt}(z) \pmod{E} \]
is a finite summation inside $\bargminf/E$, and falls inside the image of 
 \[W \to \bargm_\inf/E.\]
\end{enumerate}
\end{lemma}

\begin{remark} \label{rem analytic than lav}
We will verify   Condition \eqref{eq: cond fkti} in various situations. As we shall see in the following Theorem \ref{thmintsen}, Condition \eqref{eq: cond fkti} is verified for \emph{any} $x\in \gm$ with $b=1,Y=\minf$; the fact that the (normalized) sequence $\log \tau$ converges implies that elements in $ \gm/E$ are indeed   \emph{analytic vectors} (not just locally analytic) inside the $\qp$-Banach representation $T\otimes_\zp C$. In particular, on elements  inside $\gm/E$, $\log\tau$ \emph{coincides} with $\nabla_\tau$ in Notation \ref{nota lie hatg}.
\end{remark}

\begin{thm}[Integrality of Sen operator] \label{thmintsen}  
 Consider the operator in Proposition \ref{prop: K ratitonal sen},
   \[\theta_\kinfty: (\gm/E)[1/p] \to (\gm/E)[1/p]. \]
We have:
\[ \theta_{\kinfty}(\gm/E) \subset \frac{1}{\pi E'(\pi)}\cdot \gm/E   \]
 When $T$ is crystalline, we further have
\[\theta_{\kinfty}(\gm/E) \subset \frac{1}{E'(\pi)}\cdot \gm/E   \]
\end{thm}
 \begin{proof} 
We first treat the (general) semi-stable case. 
Denote
\[ \fkc_\log:=   -\pi E'(\pi)\cdot \frac{1}{\theta_\fon(u\lambda')} \cdot \frac{1}{p}
 = \frac{E(0)}{ p \cdot  \theta_\fon(\phi(\lambda))}  \in \o_K^\times \]
Here we use $ \theta_\fon( \lambda') =\frac{E'(\pi)}{E(0)}\theta_\fon(\phi(\lambda))$, and use the fact that  $\theta_\fon(\phi(\lambda))$ and $ E(0)/p$ are in $\o_K^\times$.
Thus, one can write
\[ \pi E'(\pi)\theta_\kinfty= \fkc_\log   \frac{\nabla_\tau}{ \fkt}\]
So now it suffices to prove 
\[ \frac{\nabla_\tau}{\fkt}(\gm/E) \subset \gm/E \]
 Prop. \ref{prop: K ratitonal sen} implies $\theta_\kinfty$ and hence $ \frac{\nabla_\tau}{\fkt}$  stabilizes the rational object $(\gm/E)[1/p]$. Thus, it suffices to prove that
\begin{equation} \label{eqintstable}
\frac{\nabla_\tau}{\fkt}(\gm/E) \subset \minf/E   
\end{equation}
Lemma  \ref{lem: tau range gmast} implies
  \[  (\tau-1)^i (\gm) \subset  \fkt^i \minf \]
Thus we can apply Lem. \ref{lem: log tau converge} (with $b=1, Y=\minf$) to conclude \eqref{eqintstable}. See also Remark \ref{rem analytic than lav} about coincidence between $\log\tau$ and $\nabla_\tau$.

Now, suppose $T$ is furthermore crystalline.
Similar as in the semi-stable case, we are reduced to prove the analogue of \eqref{eqintstable} in the crystalline case, which is
\begin{equation} \label{eqstablecrys}
\frac{\nabla_\tau}{u\fkt}(\gm/E) \subset \minf/E
\end{equation}
Note in the crystalline case, Lemma  \ref{lem: tau range gmast} implies
 \[  (\tau-1)^i (\gm) \subset u \fkt^i \minf \]
 then we can apply Lem. \ref{lem: log tau converge} (with $b=u, Y=\minf$)  to conclude \eqref{eqstablecrys}.
\end{proof}

The above theorem leads to following definition.

\begin{defn} \label{def: amplified Sen}
   Suppose $T$ is $\star$-crystalline. Let
 \begin{equation*}
a=
\begin{cases}
  E'(\pi), &  \text{if } \star=\emptyset \\
 \pi E'(\pi), &  \text{if } \star=\log
\end{cases}
\end{equation*} 
Define the (integral negative)  \emph{amplified  Sen operator}:
\[ \Theta = a\theta_\kinfty: \gm/E \to \gm/E \] 
\end{defn}

\begin{remark}
Note the adjective ``negative" for the amplified operator is slightly misleading; for example, when $K$ is unramified and consider the  log-crystalline case, one could  well choose $\pi=-p$ which is a ``negative" number. Nonetheless, in our application (in the integral case) in \S \ref{sec: vanishing} and \S \ref{sec: frob matrix}, we   only consider $K$ unramified and $T$  crystalline; in that case, we do have
\[ \Theta=\theta_\kinfty. \]
\end{remark}

 %%\newpage 
\section{Filtered Sen theory} \label{sec: fil sen}
 
In this section, we construct filtered Sen theory, which works in  rational case, integral case and also the mod $p$ case. 
We show that (amplified) Sen operator  stabilizes conjugate filtration, and the ``shifted" (amplified) Sen operator  satisfies $1$-degree shrinking (cf. Remark \ref{rem:ds vs gt}  for discussion of terminology). 
 In the interlude \S \ref{sub-BeuilN}, we discuss the relations with Breuil's $N$-operator, which--- particularly the Griffiths transversality it satisfies--- was   a strong motivation in our initial investigations.
As a continuation of the previous section, we keep using Notation \ref{nota: rep and mod}; that is, we let $\star \in \{ \emptyset, \log \}$ and let $T$ be a  $\star$-crystalline representation with Hodge--Tate weights $0 \leq r_1 \leq \cdots \leq r_d$.

\subsection{Rational filtered Sen theory}  
\begin{construction} \label{cons: translate gr and conj fil}   
 Since $N_\nabla$ is stable on $E^n\cm$, it induces a $K$-linear operator on $E^n\cm/E^{n+1}\cm$; we claim the following  diagram is commutative.
\begin{equation} \label{diag: cmht twist}
\begin{tikzcd}
E^n\cm/E^{n+1}\cm \arrow[rr, "\fkc N_\nabla"] \arrow[d, " \times E^{-n}"] &  & E^n\cm/E^{n+1}\cm \arrow[d, "  \times E^{-n}"] \\
\cm/E\cm \arrow[rr, "\fkc N_\nabla-n"]                                                                       &  & \cm/E\cm                                                                        
\end{tikzcd}
\end{equation}
Here, $\mathfrak{c}=\frac{1}{\theta_\fon(u\lambda')}$. The commutativity of the diagram follows from standard computation that for $m \in \cm$:
\[  \fkc E^{-n}N_\nabla(E^nm) = \fkc N_\nabla(m)-n m \pmod{E}\] 
Indeed, this diagram simply says that $E^n\cm/E^{n+1}\cm$ is the $K$-rational Sen module corresponding to the $C$-representation $T(n)\otimes_\zp C$, where $T(n)$ is the Tate twist; cf. Proposition \ref{prop: K ratitonal sen}.
\end{construction}

 \begin{theorem} \label{thm: rational sen shift}
 Consider the negative Sen operator
 \[ \theta_\kinfty: \cm_\HT \to \cm_\HT\]
\begin{enumerate} 
 \item The operator $\theta_\kinfty-n$ sends $\fil_n^\rmconj \cm_\HT$ to $\fil_{n-1}^\rmconj \cm_\HT$, leading to the map which we call the ($n$-th) \emph{shifted (negative) Sen operator:}
 \[ \theta_\kinfty-n: \fil_n^\rmconj \cm_\HT  \to \fil_{n-1}^\rmconj \cm_\HT\]
 
     \item       $\theta_\kinfty(\fil_n^\rmconj \cm_\HT) \subset \fil_n^\rmconj \cm_\HT$ and the induced action of $\theta_\kinfty$ on $\gr_n \cm_\HT$ is scaling by $n$.

 \end{enumerate} 
 \end{theorem}

 %; thus we have an operator       \[ -\theta_\kinfty -n:  \fil_n^\rmconj \cm_\HT \to \fil_{n-1}^\rmconj \cm_\HT;\]    We  call it the $n$-th \emph{shifted (negative) Sen operator}. 
 \begin{proof}
Use the injection $\fil^n \cmast/\fil^{n+1} \cmast \into E^n\cm/E^{n+1}\cm$, and recall $N_\nabla(\fil^n \cmast) \subset E \fil^{n-1} \cmast$ (Lem. \ref{lem: Kis nyg fil}). Also use the definition of conjugate filtration Definition \ref{defn: conj fil}, we have the following commutative diagram, where each term is a sub-module of   diagram \eqref{diag: cmht twist}
\begin{equation} \label{diag: cmht twist sub}
\begin{tikzcd}
\fil^n \cmast/\fil^{n+1} \cmast \arrow[rr, "\fkc N_\nabla"] \arrow[d, "{ \times E^{-n}, \simeq}"] &  & E \fil^{n-1} \cmast/E \fil^{n} \cmast \arrow[d, "{ \times E^{-n}, \simeq}"] \\
\fil_n^\rmconj \cm_\HT \arrow[rr, "\fkc N_\nabla-n"]                                                                                 &  & \fil_{n-1}^\rmconj \cm_\HT                                                                                    
\end{tikzcd}
\end{equation}
This implies Item (1), because $\fkc N_\nabla-n$ on the bottom row is precisely $\theta_\kinfty-n$. 
Since conjugate filtration is increasing, $\theta_\kinfty-n$ hence $\theta_\kinfty$ stabilizes $\fil_{n}\cm_\HT$.
We can thus form the following commutative diagram (where the dotted arrow comes from Item (1)):
\begin{equation} \label{diag: theta cmht}
\begin{tikzcd}
0 \arrow[r] & \fil_{n-1}\cm_\HT \arrow[r] \arrow[d, "\theta_\kinfty-n"'] & \fil_{n}\cm_\HT \arrow[r] \arrow[d, "\theta_\kinfty-n"] \arrow[ld, "\theta_\kinfty-n"', dashed] & \gr_{n}\cm_\HT \arrow[r] \arrow[d, "\theta_\kinfty-n"] & 0 \\
0 \arrow[r] & \fil_{n-1}\cm_\HT \arrow[r]                        & \fil_{n}\cm_\HT \arrow[r]                                                       & \gr_{n}\cm_\HT \arrow[r]                       & 0
\end{tikzcd}
\end{equation} 
The dotted arrow implies the right most vertical arrow is the zero map (via diagram chasing), concluding Item (2).
 \end{proof}
   
   \begin{cor} \label{cor: eigenval fil}
   The conjugate filtration $\fil_\bullet \cm_\HT$ is the same as $\theta_\kinfty$-eigenvalue filtration in the sense that for each $i$
\[ \fil_i \cm_\HT =\bigoplus_{j \leq i} ( \cm_\HT)^{\theta_\kinfty= j} \]
   \end{cor}
     %  $\theta_\kinfty$ is  semi-simple on $\fil^\bullet \cm_\HT$ in the sense that for each $i$
   \begin{proof}
It follows from the fact that $\theta_\kinfty$ is semi-simple, and the fact that $\theta_\kinfty-n$ kills $\gr_n \cm_\HT$ which is of dimension equal to that of $\gr^n \cm_\dR=\gr^n D_\dR$.
   \end{proof}

\subsection{Relation with Breuil's $N$-operator}\label{sub-BeuilN}

In this subsection, we discuss the relation between the shifted (negative) Sen operator in Thm. \ref{thm: rational sen shift} and the $N$-operator on Breuil's $S_\ko$-module (Definition \ref{def: breuil mod}). These discussions will not be further used in this paper, but we would like to point out that it serves as a strong motivation (and psychological inspiration) in our initial construction   of Thm. \ref{thm: rational sen shift}, cf. Remark \ref{rem: nnabla and N}.

Recall $N$ on $\cd$ satisfies Griffiths transversality $N(\fil^n \cd) \subset \fil^{n-1}\cd$.
This induces an operator on graded: 
\[N: \gr^n \cd \to \gr^{n-1}\cd \]
Recall Lem. \ref{lem: compa of graded} implies that
 \[\gr^n \cd \simeq \gr^n \cm \simeq \fil_n \cm_\HT\]
 Thus we have an operator
 \begin{equation} \label{eq: n on  cmht}
 N: \fil_n \cm_\HT \to \fil_{n-1} \cm_\HT
 \end{equation}

\begin{prop} \label{prop: compa Breuil}
 After linearly scaling   \eqref{eq: n on  cmht}, the map 
\[ \frac{\phi(\lambda)}{p}\cdot N: \fil_n \cm_\HT \to \fil_{n-1} \cm_\HT \]
is exactly the same as
 \[ \theta_\kinfty-n: \fil_n \cm_\HT  \to \fil_{n-1} \cm_\HT\]
 in Theorem \ref{thm: rational sen shift}.
\end{prop}
\begin{proof}  
Use Construction \ref{Constructioncm to cd}, we have a commutative diag
\begin{equation} \label{eq two op 22}
   \begin{tikzcd}
\cmast \arrow[d, hook] \arrow[rr, " N_\nabla"] &  &    \cmast  \arrow[d, "\frac{p}{\phi(\lambda)}", hook] \\
\cald \arrow[rr, "N"]                               &  &    \cald                                                  
\end{tikzcd} 
\end{equation}
(Note we are regarding  $\cmast \subset \cald$ as a submodule, thus there are no more $\phi$-twists as in the formulae   in \cite[\S 3.2]{Liu08}.)    
Note on right vertical arrow, $\phi(\lambda) \in S^\times$.
Take filtrations, and note 
$ N_\nabla(\fil^n \cmast) \subset  E \fil^{n-1} \cmast $ by  Lem. \ref{lem: Kis nyg fil}; thus we have
 \begin{equation} \label{eq two op fil}
  \begin{tikzcd} 
\fil^n \cmast \arrow[d, hook] \arrow[rr, " N_\nabla"] &  & E\fil^{n-1} \cmast \arrow[d, "\frac{p}{\phi(\lambda)}", hook] \\
\fil^n\cald \arrow[rr, "N"]                                &  & \fil^{n-1}\cald                                                
\end{tikzcd}  
\end{equation}  
taking graded, then we get
\begin{equation} \label{eq two op gr}
  \begin{tikzcd}
\gr^n \cmast \arrow[d, "="] \arrow[rr, " N_\nabla"] &  & E\fil^{n-1} \cmast/E\fil^{n} \cmast \arrow[d, "{\frac{p}{\phi(\lambda)}, \simeq}"] \\
\gr^n\cald \arrow[rr, "N"]                               &  & \gr^{n-1}\cald                                               
\end{tikzcd}
\end{equation} 
Here the   vertical isomorphisms are proved in Lemma \ref{lem: compa of graded}.
This diagram translates into the desired statement (using Construction \ref{cons: translate gr and conj fil}).
\end{proof}
  
\begin{remark}\label{rem: nnabla and N}
 We make some observations related with Proposition \ref{prop: compa Breuil}.
\begin{enumerate}
%\item Recall $N$ on $\cd$ satisfies Griffiths transversality, but $N_\nabla$ on $\cm$ does not. 

\item In previous development of integral $p$-adic Hodge theory, the two operators $N$ and $N_\nabla$ are used in rather different fashions:
\begin{itemize}
\item The $N$-operator is crucially used in the construction of $(\phi, \hat{G})$-modules in \cite{Liu10}; in addition, the   Griffiths transversality is \emph{heavily} used in its subsequent applications e.g. in \cite{GLS14} (cf. also our \S \ref{subsec: GLS reproof}).

\item In comparison, the $N_\nabla$-operator on Kisin's $\o$-modules in \cite{Kis06} is indeed a special case of a similar $N_\nabla$-operator on all \emph{overconvergent} $(\phi, \tau)$-modules, constructed in \cite{GL20, GP21},  which in turn is a generalization of (locally analytic) ``Sen theory" (cf. \cite{BC16}). The  $N_\nabla$-operator plays the key role in constructing the Breuil--Kisin $\gk$-modules in \cite{Gao23}, which can be regarded as a ``non-$\phi$-twisted" version of  $(\phi, \hat{G})$-modules in \cite{Liu10}.
\end{itemize}

\item  Proposition \ref{prop: K ratitonal sen} implies that after normalization, $N_\nabla$ modulo $E$ is the Sen operator on $\cm_\HT$.
By \cite{Bre97}, the operator $N/u$ on $\cd/u$ is precisely $N_{D_\st(V)}$ on the Fontaine module $D_\st(V)$, hence in particular is \emph{nilpotent}.
This makes it hard to see if $N$ could induce any (non-nilpotent) Sen operator.

\item \label{item nnabla inherit}  What Proposition \ref{prop: compa Breuil} tells us is that: along the conjugate filtration $\fil_\bullet^\rmconj \cm_\HT$, the two operators $N$ and $N_\nabla$ are \emph{identified} into the same \emph{shifted Sen operators},  inheriting/incorporating the good features of both operators: they satisfy $1$-degree shrinking  and can read off Sen weights.
\end{enumerate}
\end{remark}
  
\subsection{Integral filtered Sen theory}  
 
\begin{theorem}  \label{thm: integral sen integral conj fil}  
 Consider the amplified Sen operator (and the constant $a$)  in Definition \ref{def: amplified Sen}, 
\[ \Theta: \gm_\HT \to \gm_\HT. \] 
We have
\begin{enumerate}
\item  The operator $\Theta -n a$  sends $\fil_n^\rmconj \gm_\HT$ to $\fil_{n-1}^\rmconj \gm_\HT$.
    \item    $\Theta (\fil_n^{\mathrm{conj}} \gm_\HT) \subset \fil_n^{\mathrm{conj}} \gm_\HT $ and the induced action on $\gr_n \gm_\HT$ is scaling by $na$. 
\end{enumerate} 
 \end{theorem} 
 \begin{proof}
The argument is very similar to that in Theorem \ref{thm: rational sen shift};  although we   caution that we cannot use an intersection argument to conclude since the filtered map $\fil_\bullet \gm_\HT \to  \fil_\bullet \cm_\HT$ is not strict. 
 
Indeed, similarly to Construction \ref{cons: translate gr and conj fil}, the $\Theta$-operator  induces a commutative diagram
 \begin{equation} \label{diag gmtwist} 
\begin{tikzcd}
E^n\gm/E^{n+1}\gm \arrow[d, "\times E^{-n}"] \arrow[rr, "\Theta "] &  & E^n\gm/E^{n+1}\gm \arrow[d, "\times E^{-n}"] \\
\gm/E\gm \arrow[rr, " \Theta -na "]                                       &  & \gm/E\gm                                    
\end{tikzcd}
 \end{equation}
Lemma \ref{lem: tau range gmast} implies
\begin{equation}\label{eq: intconj3}
     (\tau-1)^i(\fil^n \gmast) \subset \fkt^i E\fil^{n-1} \gm_\inf^\ast
\end{equation} 
Thus Lemma \ref{lem: log tau converge} implies that $\Theta$ induces an operator
\[\Theta: \fil^n \gmast/\fil^{n+1} \gmast \to  E \fil^{n-1} \gm_\inf^\ast/ E \fil^{n} \gm_\inf^\ast \]
The top row of \eqref{diag gmtwist} says that $\Theta$ is stable on $E^n\gm/E^{n+1}\gm$; thus the image of the above map lands inside
\[ \left( E \fil^{n-1} \gm_\inf^\ast/ E \fil^{n} \gm_\inf^\ast \right) \cap 
\left(E^n\gm/E^{n+1}\gm \right) =E \fil^{n-1} \gmast/ E \fil^{n} \gmast; \]
here the intersection equality follows from Lemma \ref{lem: flat base change}\eqref{flatbcitem3}(b). 
Thus, we can construct the following commutative diagram (as a sub-diagram of \eqref{diag gmtwist})
\begin{equation} \label{diag shift sen gmht}
\begin{tikzcd}
\fil^n \gmast/\fil^{n+1} \gmast \arrow[d, "{\times E^{-n}, \simeq}"] \arrow[rr, "\Theta"] &  & E \fil^{n-1} \gmast/ E \fil^{n} \gmast  \arrow[d, "{\times E^{-n}, \simeq}"] \\
\fil_n^\rmconj \gm_\HT \arrow[rr, "\Theta-na"]                                            &  & \fil_{n-1}^\rmconj \gm_\HT                                                  
\end{tikzcd}
\end{equation} 
This leads to
\begin{equation} \label{diag: theta gmht}
\begin{tikzcd}
0 \arrow[r] & \fil_{n-1}\gm_\HT \arrow[r] \arrow[d, "\Theta-na"'] & \fil_n\gm_\HT \arrow[r] \arrow[d, "\Theta-na"] \arrow[ld, "\Theta-na"', dashed] & \gr_{n}\gm_\HT \arrow[r] \arrow[d, "\Theta-na"] & 0 \\
0 \arrow[r] & \fil_{n-1}\gm_\HT \arrow[r]                        & \fil_n\gm_\HT \arrow[r]                                                       & \gr_{n}\gm_\HT \arrow[r]                       & 0
\end{tikzcd}
\end{equation} 
and we can conclude as   in Theorem \ref{thm: rational sen shift}.
 \end{proof}

%\begin{rem}     Note in general, the inclusion     \[ \fil^i_{\mathrm{conj}} \gm/E \subset \gm/E \cap \fil^i_{\mathrm{conj}} \gm/E[1/p]  \]     is not equality, thus we cannot use intersection argument. \end{rem}

 \subsection{Mod $p$ filtered Sen theory}
 
In this subsection, we consider filtered Sen theory in the mod $p$ case.
Let $\overline{\gs}=\gs/p\gs=k[[u]]$, let $\bargm =\gm\otimes_\gs \overline{\gs}$ be the mod $p$ Breuil--Kisin module, which we regard as an effective isogeny (Notation \ref{nota: isogeny}) over the triple
\[ (\overline{\gs}, \bar{E}=u^e, \phi) \]
Thus, in the terminology of Definition \ref{defn: conj fil}, we have
\[ \bargm_\HT=\bargm/u^e\bargm \simeq \gm_\HT/p\gm_\HT\]
which is an $\ok/p\ok$-module.
(See Remark \ref{rem not mod pi} below to see why we do not use the $k$-vector space $\gm_\HT/\pi \gm_\HT$).

\begin{theorem} \label{thm: mod p Sen op fil}
Consider amplified Sen operator  
\[ \Theta: \gm_\HT \to  \gm_\HT.\] 
 Modulo $p$,  we obtain
\[ \overline \Theta: \bargm_\HT \to \bargm_\HT.\]
We have
\begin{enumerate}
 \item  The operator $\overline \Theta-na$ sends $\fil_n^\rmconj \bargm_\HT$ to $\fil_{n-1}^\rmconj \bargm_\HT$.
 
    \item $\overline \Theta$ stabilizes $\fil_n^\rmconj \bargm_\HT$   and the induced action on $\gr_n  \bargm_\HT$ is scaling by $na$.   
\end{enumerate} 
\end{theorem}
\begin{proof}
  Lemma \ref{lem: tau range gmast} implies
\[   (\tau-1)^i( \bargmast) \subset (\phi(\fkt))^i  \bargm_\inf^\ast \]
and
\[   (\tau-1)^i( E^n \bargm) \subset \fkt^i E^n \bargm_\inf \]
Take intersection, we have
\[   (\tau-1)^i( \fil^n \bargmast) \subset \fkt^iE\fil^{n-1}\bargm_\inf^\ast \]
(Note this cannot be directly implied by \eqref{taufilgm}: the map $\fil^n \gmast  \to \fil^n \bargmast$ might not be surjective.)
  Lem. \ref{lem: log tau converge} (the mod $p$ case) induces an operator
  \[ \overline{\Theta}: \fil^n \bargmast/\fil^{n+1} \bargmast \to  E \fil^{n-1} \bargm_\inf^\ast/ E \fil^{n} \bargm_\inf^\ast \]
  The top row of \eqref{diag gmtwist} implies that $\overline{\Theta}$ is stable on $E^n\bargm/E^{n+1}\bargm$, thus Lemma \ref{lem: flat base change}(3)(a) implies that we obtain an operator:
    \[ \overline{\Theta}: \fil^n \bargmast/\fil^{n+1} \bargmast \to  E \fil^{n-1} \bargm^\ast/ E \fil^{n} \bargmast \]
    Similar to the integral case in Theorem \ref{thm: integral sen integral conj fil}, we have mod $p$ version of diagram \eqref{diag shift sen gmht}:
    \begin{equation} \label{diag shift sen bargmht}
\begin{tikzcd}
\fil^n \bargmast/\fil^{n+1} \bargmast \arrow[d, "{\times E^{-n}, \simeq}"] \arrow[rr, "\overline{\Theta}"] &  & E \fil^{n-1} \bargmast/ E \fil^{n} \bargmast  \arrow[d, "{\times E^{-n}, \simeq}"] \\
\fil_n^\rmconj \bargm_\HT \arrow[rr, "\overline{\Theta}-na"]                                            &  & \fil_{n-1}^\rmconj \bargm_\HT                                                  
\end{tikzcd}
\end{equation}  
as  well as the diagram
\begin{equation} \label{diag: theta bargmht}
\begin{tikzcd}
0 \arrow[r] & \fil_{n-1}\bargm_\HT \arrow[r] \arrow[d, "\overline{\Theta}-na"'] & \fil_{n}\bargm_\HT \arrow[r] \arrow[d, "\overline{\Theta}-na"] \arrow[ld, "\overline{\Theta}-na"', dashed] & \gr_{n}\bargm_\HT \arrow[r] \arrow[d, "\overline{\Theta}-na"] & 0 \\
0 \arrow[r] & \fil_{n-1}\bargm_\HT \arrow[r]                        & \fil_{n}\bargm_\HT \arrow[r]                                                       & \gr_{n}\bargm_\HT \arrow[r]                       & 0
\end{tikzcd}
\end{equation} 
 Thus we can conclude.
\end{proof}

\begin{remark} \label{rem unram crys mod p}
When $a$ is not a unit, i.e., when $K$ is ramified or we are in the semi-stable case, then  $\overline{\Theta}$ is   nilpotent.
\end{remark}

\begin{remark} \label{rem not mod pi} 
Suppose $K$ is ramified (that is, $e>1$), then $\bargm$ is also an effective isogeny with respect to the triple 
\[ (k[[u]], u, \phi). \] 
In addition, $\overline{\Theta}$ on $\bargm_\HT=\bargm/u^e\bargm$ also induces (simply by mod $\pi$) an operator on $\bargm/u\bargm$ (which is a $k$-vector space). However,   the tripe $(k[[u]], u, \phi)$  would induce \emph{different} filtrations (on $\bargmast$ and its reductions), and would  break several normalizations in the discussion of Theorem \ref{thm: mod p Sen op fil}. 
Also, note this mod $\pi$ operator  is still nilpotent and hence is not so useful in this paper. Nonetheless, we point out that certain other (differently normalized) ``mod $\pi$ operators"   indeed would appear in our future work, cf. Remark \ref{rem: future}.
\end{remark}

%%\newpage 
\section{Lifting Sen operator over max rings} \label{sec: sen max classical}

In \S \ref{sec: integral Sen}, we constructed an \emph{integral} (amplified) Sen operator $\Theta$, using the \emph{rational} operator $N_\nabla$ defined on $\o$-modules. In this section, we show that $\Theta$ also admits a certain \emph{integral lift}, defined on a certain integral ring $\gsmax$. The results here are no longer used in this paper, cf. Remark \ref{limitation max lift}; however, these constructions serve as important motivations for our constructions in \S \ref{sec:pris-max-ring} and \S \ref{sec:p-GT-A}.

\begin{notation} We review some ``max"-rings.
\begin{enumerate}
\item 
Following the original notation in \cite[III.2]{Col98} \footnote{In comparison, the \emph{notation} ``$\ainf$" seems to be first introduced in \cite{Fon94} (although the ring itself is used in as early as  \cite{Fon82} (denoted as $W(R)$ there); cf. also \cite[1.5.2]{Col19} for a quick history of other period rings.}, define a  ring  $\amax=\ainf[ \frac{E}{p} ]^{\wedge_p}$ where $[ \cdot ]^{\wedge_p}$ signified $p$-adic completion, and equip it with $p$-adic topology. By the first paragraph in \cite[III.2]{Col98}, we also have 
\[ \amax =\ainf[ \frac{\fkt^{p-1}}{p} ]^{\wedge_p} =\ainf [\frac{u^e}{p} ]^{\wedge_p} \]
One has $\phi(\amax) \subset \acris \subset \amax$, hence $\amax$ can be regarded a certain ``Frobenius descent" of $\acris$; and it is well-know that this ring behaves better than $\acris$ in many aspects, yet is also sufficient to study crystalline representations.
%Denote $\bmax:=\amax[1/p]$.

\item 
Define \[ \gsmax:=\gs[\frac{E}{p} ]^{\wedge_p} =\gs[\frac{u^e}{p} ]^{\wedge_p},\]
 which can be regarded a certain ``Frobenius descent" of $S$. It is easy to check $\lambda \in \gsmax$, and $\o \subset \gsmax[1/p]$.

\item We can re-normalize the differentiation $N_\nabla  =-\frac{1}{p\fkt}\log\tau$ on $\o$ to define the following  differentiations.
Let $\star \in \{ \emptyset, \log \}$. When $\star=\emptyset$, define
\[\wh{\Theta}_\max: =\frac{E}{ut}\cdot \log\tau=E\frac{d}{du}: \gsmax \to \gsmax\]
When $\star=\log$, define
\[\wh{\Theta}_{\max}: =\frac{E}{t}\cdot \log\tau=uE\frac{d}{du}: \gsmax \to \gsmax\]
\end{enumerate}
\end{notation}

\begin{thm} \label{Theorem bk conn}
Let $\star \in \{ \emptyset, \log \}$.  Let $T$ be a $\star$-crystalline $\zp$-representation. Let $\gm_\max:=\gm\otimes_\gs \gsmax$. Then the operator $\wh{\Theta}_\max$ (which is $\frac{E}{ut}\cdot \log\tau$ resp. $\frac{E}{t}\cdot \log\tau$) is stable on $\gm_\max$. In addition, the following diagram is commutative; here, the vertical arrows are induced by reduction modulo $E/p$, and the bottom row is defined in Definition \ref{def: amplified Sen}.
\[
\begin{tikzcd}
\gm_\max \arrow[d, two heads] \arrow[r, "\wh{\Theta}_\max"] & \gm_\max \arrow[d, two heads] \\
\gmht \arrow[r, "\Theta"]                          & \gmht                      
\end{tikzcd}
\]
The operator $\wh{\Theta}_\max$ satisfies Leibniz rule with respect to $\wh{\Theta}_\max$ on $\gsmax$, and satisfies the relation 
\[ \frac{p}{\phi(E)}u^{p-1}E \cdot \phi \circ \hattheta_\max=\hattheta_\max \circ \phi, \text{ when $\star=\emptyset$ }\]
\[ \frac{p}{\phi(E)} E \cdot \phi \circ \hattheta_\max =\hattheta_\max \circ \phi, \text{ when $\star=\log$} \]
\end{thm}
\begin{proof} 
 We only treat the log-crystalline case (the crystalline case is similar). In this case, Lemma \ref{lem: tau range gmast} implies 
\[  (\tau -1)^i (\gm)   \subset    \fkt^i \minf \]
Since $\fkt^{i-1}/i \in \amax$ and converges to zero, it is easy to check that 
\[ \frac{\log \tau}{\fkt}(\gm) \subset \gm \otimes_\gs \amax \]
Note $E/(p\lambda)$ is a unit in $\gsmax$; thus we have
\[ \hattheta_\max(\gm) \subset  \gm \otimes_\gs \amax \]
As $N_\nabla(\gm) \subset \gm\otimes_\gs \o \subset \gm\otimes_\gs \gsmax[1/p]$, it is easy to see
\[ \hattheta_\max(\gm) \subset  \gm \otimes_\gs \gsmax[1/p]\]
We can conclude the stability of $\hattheta_\max$ on $\gm_\max$ since
 \[ \gs_\max[1/p] \cap \amax =\gs_\max.\]
The commutativity of the diagram can be readily checked using the equation (cf. Theorem \ref{thmintsen})
\[ \Theta=\fkc_\log   \frac{\nabla_\tau}{ \fkt}.\]
Relations with $\phi$ are formal.
\end{proof}

\begin{remark} \label{limitation max lift}
Unfortunately, this \emph{integral} lift of the (integral) Sen operator has its limitations. Indeed, the map $\gs \into \gsmax$ is not flat hence Lemma \ref{lem: flat base change} is not applicable, making it  difficult to study Nygaard (and other) filtrations on $\gm_\max$. In addition, after modulo $p$, $E=0$ in $\gsmax/p$, and hence the finite height condition  is lost. The existence of this (unsatisfactory) lift inspires and forces us to discover the lift in \S \ref{subsec-flat-subring}.
\end{remark}

%%\newpage 
\addtocontents{toc}{\ghblue{Applications: integral case}}
  
\section{Torsion bound and vanishing  of graded pieces} \label{sec: vanishing}

We apply filtered Sen theory to study    graded pieces of conjugate/Hodge   filtrations. 
Recall Lem. \ref{lem: matching graded dR and HT} implies the graded pieces of these two filtrations are the same; that is: $\gr^n \gm_\dR=\gr_n \gm_\HT$. Thus, in this section, 
we only consider the conjugate filtration and its graded. We obtain control of  $(\gr_n \gm_\HT)_\tor$  in Theorem \ref{thm: bound expo} and Theorem \ref{thm: bound generator}; these imply vanishing results in Theorem \ref{thm: vanish integral}. Similar method leads to reproof of mod $p$ vanishing results of Gee--Kisin \cite{GK-ias}, cf. Theorem \ref{thm: mod p jump}.

\begin{notation} \label{nota: crys final}
Suppose  $K$ is unramified, and  $T$ is a   crystalline $\zp$-representation with Hodge--Tate weights $0 \leq r_1 \leq \cdots \leq r_d$. Let $(\gm, \gm_\inf)$ be the Breuil--Kisin $\gk$-module associated  to $T$. In this case, $a=1$ in Definition \ref{def: amplified Sen};  we simply denote
\[\theta:= \Theta =\theta_\kinfty \]
\end{notation}

The following easy lemma will be repeatedly used.

\begin{lemma} \label{lem: det}
Fix $n$ and vary $m$.
Consider the endomorphism
\[\theta-n: \fil_{m} \gm_\HT\to \fil_{m} \gm_\HT\]
\begin{enumerate}
\item The determinant  is $\prod_{r_i \leq m}(r_i-n)$.
\item It is injective if all appearing $r_i-n$ (i.e., those $r_i \leq m$) are nonzero; for example, when $m<n$.
\item It is bijective if all appearing $r_i-n$ are $p$-adic units. 
\end{enumerate} 
\end{lemma}
\begin{proof}
One only needs to prove Item (1). To compute the determinant, one can invert $p$ and then apply  Corollary  \ref{cor: eigenval fil}. 
\end{proof}

\subsection{Bound of torsion} 
 In this subsection, we  bound torsion in $\gr_n \gm_\HT$. 
  Given a finitely generated $W(k)$-module $M$, let $r(M)$ be the minimal number of its generators (equivalently, $k$-dimension of $M/pM$); when $M$ is torsion, let $e(M)$ be its exponent which is the smallest integer such that $p^{e(M)}$ kills $M$.

\begin{lemma} \label{lem: to bound tor}
Denote the cokernel of $\theta-n: \fil_{m} \gm_\HT\to \fil_{m} \gm_\HT$ as
\[ C_m =\fil_{m} \gm_\HT/(\theta-n) \]
\begin{enumerate}
\item \label{eq91}  There is a left exact sequence $
0 \to (\fil_n  \gm_\HT)^{\theta=n} \to \gr_n   \gm_\HT \to C_{n-1}$.

\item   \label{eq92} For any $b$, we have $
C_{n-bp-1}  =C_{n-bp-2} =\cdots =C_{n-bp-p}  $
 
\item \label{eq94} For any $b$, we have a  right  exact sequence $
 C_{n-bp-p} \to C_{n-bp} \to (\gr_{n-bp} \gm_\HT)/(bp) \to 0$. 
\end{enumerate}
\end{lemma}
\begin{proof} 
We know $\theta=m$ on $\gr^m  \gm_\HT$ by Theorem \ref{thm: integral sen integral conj fil}. 
Consider following diagram
\[
\begin{tikzcd}
0 \arrow[r] & \fil_{m-1}\gm_\HT \arrow[r] \arrow[d, "\theta-n"'] & \fil_{m}\gm_\HT \arrow[r] \arrow[d, "\theta-n"] & \gr_{m}\gm_\HT \arrow[r] \arrow[d, "(\theta-n)=(m-n)"] & 0 \\
0 \arrow[r] & \fil_{m-1}\gm_\HT \arrow[r]                        & \fil_{m}\gm_\HT \arrow[r]                       & \gr_{m}\gm_\HT \arrow[r]                           & 0
\end{tikzcd}
\] 
 The case $m=n$ leads to \eqref{eq91}; note $\theta-n$ is injective on $\fil_{n-1}$ by Lemma \ref{lem: det}.  
 When  $p \nmid n-m$, the right most vertical arrow is an isomorphism, thus
$C_m=C_{m-1}$ which inductively implies \eqref{eq92}.
 The case $m=n-bp$ leads to   \eqref{eq94}.
\end{proof}

\begin{theorem}[Bound of exponent]  \label{thm: bound expo}  
\begin{enumerate}
\item  For each $n$, $(\gr_n \gm_\HT)_\tor$ is killed by $ n!$. 
\item \label{item nbig vanish} If $n \geq r_d+1$, then $\gr_n \gm_\HT=0$.
 If $n \geq r_d$, then $(\gr_n \gm_\HT)_\tor=0$.

\item Uniformly for all $n$, $(\gr_n \gm_\HT)_\tor$ is killed by $(r_d-1)!$.
\end{enumerate}   
\end{theorem}
\begin{proof}
 Write $n=a+pk$ with $0 \leq a \leq p-1$. 
   Use exact sequence in  Lemma \ref{lem: to bound tor}\eqref{eq91}; note $(\fil_n  \gm_\HT)^{\theta=n}$ is torsion-free, thus we have ``torsion control" (note $C_m$ is   torsion if  $m <n$): 
\[  (\gr_n\gm_\HT)_\tor  \subset  C_{n-1}  = C_{n-p}  \]
 Lemma \ref{lem: to bound tor}\eqref{eq94} shows that the exponent is bounded by
\[ e(C_{n-p}) \leq  e(C_{n-2p}) + v_p(p)   \leq  e(C_{n-3p}) +v_p(2p)+ v_p(p) =  e(C_{n-3p}) +v_p((2p)!) \leq \cdots  \] 
\[\leq   e(C_{a}) + v_p((n-a-p)! ) \leq e(C_{a-p}) + v_p((n-a)! )= v_p((n-a)! ) =v_p(n!) \]  

Consider Item (2). Note $\gm$ has Frobenius height $r_d$, thus $\fil^i \gmast=E^i\gm$ for $i \geq r_d$.
This implies that when $n >r_d$,
 $\fil^{n}\gm_\dR=0$ and  $\gr_n \gm_\HT=0.$
 In the border case  $n=r_d$, $\gr_{r_d} \gm_\HT=\gr^{r_d} \gm_\dR=\fil^{r_d} \gm_\dR$ is torsion-free. Item (3) then follows.
\end{proof}
%(This argument and hence conclusion holds verbatim for ramified $K$ and semi-stable $T$.) The uniform bound in (2) comes from the fact that $\gr_n \gm_\HT=0$ when $n>r_d$.

\begin{example} \label{ex: torfree p 2p}
\begin{enumerate}
\item  \label{item torfree p case} If $r_d \leq p$, then $(\gr_n \gm_\HT)$  is torsion-free for all $n$.
As we shall discuss in \S \ref{subsec: GLS reproof}, this fact has quick implications to Frobenius matrix of $\gm$ which was previously proved in \cite{GLS14} in connection with Serre weight conjecture for $\GL_2$.
 
\item If $r_d \leq 2p$, then $(\gr_n \gm_\HT)_\tor$  is killed by $p$ for all $n$. 
 This is particularly interesting as the range $[0,2p]$ is the range of Hodge--Tate weights appearing in Serre weight conjectures for $\GL_3$; cf. e.g. \cite{LLLMInv} and other related works. It would be very interesting to see if one can exploit this $p$-torsion fact in relevance to Serre weight conjectures.
% (See also Remark \ref{rem: gls example} for a further example). 
\end{enumerate}
\end{example}

The following theorem provides bound on number of generators on the torsion part; the current bound is first observed by Gee and Kisin and improves our previous one, cf. Remark \ref{rem:GK improve}.

\begin{theorem}[Bound of number of generators] \label{thm: bound generator} 
 Write $\alpha(x)=\sharp\{r_i \equiv x \pmod{p}, r_i \leq x\}$, which is number of Hodge--Tate weights congruent to $x$ and $\leq x$. We have
\begin{enumerate} 
\item  \label{item bnd length}   $r(\gr_n  \gm_\HT) \leq  \alpha(n)$.

\item  \label{item bnd torsion length}   $r((\gr_n \gm_\HT)_\tor) \leq   \alpha(n-p)$.

\item Uniformly for all $n$,  $r(\gr_n  \gm_\HT) \leq  d$. 
\end{enumerate}
\end{theorem}
\begin{proof} 
Item (2) implies Item (1) as the rank of the free part of $\gr_n  \gm_\HT$ is exactly $\alpha(n)-\alpha(n-p)$. Item (1) obviously implies (3). Thus it suffices to just prove Item (2). 

Consider left exact sequence in Lemma \ref{lem: to bound tor}(1):
\[ 0 \to (\fil_n  \gm_\HT)^{\theta=n} \to \gr_n   \gm_\HT \to C_{n-1}\]
 take $p$-torsion, then we have a left exact sequence
\[ 0 \to 0 \to (\gr_n \gmht)[p] \to C_{n-1}[p] \]
So
\[ r((\gr_n \gmht)_\tor) =\dim_k (\gr_n \gmht)[p]  \leq \dim_k  C_{n-1}[p] =\dim_k C_{n-1}\otimes_{W(k)} k \]
Here the last equality is because  $C_{n-1}$ is a torsion module.
From definition of $C_{n-1}$, we have a right exact sequence
\[ (\fil_{n-1} \gmht) \otimes k \xrightarrow{\theta -n} (\fil_{n-1} \gmht) \otimes k \to C_{n-1}\otimes k \to 0 \]
Recall eigenvalues of $\theta -n$ acting on $\fil_{n-1} \gmht$ are $r_i -n$ with $r_i \leq n-1$;   the dimension of the cokernel $ C_{n-1}\otimes k$ is bounded by the multiplicity of zero as eigenvalue of $\theta -n \pmod{p}$ (to see this linear algebra fact,  base change above right exact sequence to algebraic closure of $k$, and consider Jordan blocks for the matrix of $\theta-n$), which is precisely $\alpha(n-p)$.
\end{proof}

\begin{remark}\label{rem:GK improve} 
In an earlier draft, we could only prove the weaker bounds (which suffice for applications in Theorem \ref{thm: vanish integral}): 
\[ r((\gr_n \gm_\HT)_\tor) \leq 2^{\lceil \frac{n}{p} \rceil} \alpha(n-p), \quad r(\gr_n  \gm_\HT) \leq 2^{\lceil \frac{n}{p} \rceil}\alpha(n) \]
Toby Gee and Mark Kisin then show (via their module theoretic argument, mentioned in Remark \ref{rem history intro}), that one could indeed remove all the $2$-powers in these bounds. We thank them for this sharp observation and for their generosity for allowing us to include the strengthened version here. (We also note that (as far as we understand), our proof is not ``direct" translation of their proof, pointing to usefulness of different approaches to these questions.)
\end{remark}

\subsection{Integral vanishing}

We glean the following  cleaner  (torsion) vanishing results from above arguments.

\begin{theorem} \label{thm: vanish integral}
Use Notation \ref{nota: crys final} (so $K$ is unramified, and $T$ is crystalline).

\begin{enumerate}
\item If $n \geq r_d+1$, then $\gr_n \gm_\HT=0$.
\item \label{item vanish n int} If $ n \notin \{ r_i+kp, k \geq 0, 1\leq i\leq d \}$, then $\gr_n \gm_\HT=0$.
\item \label{item torfree n} If $n \notin \{ r_i+kp, k \geq 1, 1\leq i\leq d \}$, then $\gr_n \gm_\HT$ is torsion-free.
\end{enumerate} 
\end{theorem} 
\begin{proof}
(1) is already proved in  Theorem \ref{thm: bound expo}\eqref{item nbig vanish}.

(2). Apply Theorem \ref{thm: bound generator}\eqref{item bnd length}, noting $\alpha(n)=0$. Alternatively, using slightly more concrete argument, it suffices to prove the stronger statement that the composite \[  \fil_n  \gm_\HT \xrightarrow{\theta-n} \fil_{n-1}  \gm_\HT  \into \fil_{n}  \gm_\HT \]   is bijective, where the first map comes from Theorem \ref{thm: integral sen integral conj fil}. As argued in Lemma \ref{lem: det}, all the eigenvalues  here  are $r_i -n$ with $r_i \leq n$ by Corollary  \ref{cor: eigenval fil}; these are all units: otherwise $p\mid n-r_i \geq 0$, then $n=r_i+kp$ for some $k \geq 0$.

(3). Apply Theorem \ref{thm: bound generator}\eqref{item bnd torsion length}, noting $\alpha(n-p)=0$. Alternatively, one can also use more concrete argument, proving $ \theta -r_k$ is bijective on $ \fil_{r_k-1}  \gm_\HT$ (but not on $ \fil_{r_k}  \gm_\HT$!); five lemma then implies that $\gr_{r_k}  \gm_\HT=(\fil_{r_k}  \gm_\HT)^{\theta=r_k}$ and hence is torsion-free, proving cases not covered by Item (2). 
\end{proof}
%With Item (2) at hand, it remains to prove the following statement:  If $n=r_k$ is a Hodge--Tate weight, with no \emph{(strictly) smaller} Hodge--Tate weight congruent to it modulo $p$, then $\gr^{r_k} \gm_\HT$ is  $p$-torsion-free (in fact, it is also non-zero, as $\gr^{r_k} \cm_\HT$ is non-zero by Corollary  \ref{cor: eigenval fil}). In this situation, the (eigenvalue) argument in Item (2) shows that $ \theta -r_k$ is bijective on $ \fil_{r_k-1}  \gm_\HT$ (but not on $ \fil_{r_k}  \gm_\HT$!);  five lemma implies that $\gr^{r_k}  \gm_\HT=(\fil_{r_k}  \gm_\HT)^{\theta=r_k}$ and hence is torsion-free (being a submodule of a torsion-free module).

\begin{remark} \label{rem: liu24}
The  vanishing   results in Theorem \ref{thm: vanish integral} are also  proved in \cite{Liu24} by the second named author; the proof there uses similar (but more involved) technical computations as in \cite{GLS14} (which treated the case $r_d\leq p$). 
The proof presented in this paper is much more conceptual, much easier and cleaner; in addition, the method here also leads to torsion bound results in Theorem \ref{thm: bound expo}  and Theorem \ref{thm: bound generator}, which are not covered in  \cite{Liu24}. 
Furthermore, the methods also inspire the treatment in the mod $p$ case, cf. \S \ref{sec:mod-p-shape}. In summary, we regard the methods and results of this paper as more complete and more useful for future applications.
\end{remark}

\begin{remark} \label{rem: not working ram case}
If $K$ is ramified or if $T$ is semi-stable (non-crystalline), we can still run similar argument as in  Thm. \ref{thm: vanish integral}\eqref{item vanish n int}. However, the relevant eigenvalues  would be $a(r_i-n)$ which are never $p$-adic units; thus the argument becomes vacuous.
\end{remark}

\subsection{Structure of mod $p$ filtration: reproof of a theorem of Gee--Kisin} \label{subsec: mod p vanish}

 In this subsection, we reprove a theorem of Gee--Kisin announced in \cite{GK-ias}. In \cite{GK-ias}, the theorem is first stated in the form of Theorem \ref{thm: gee kisin}, but is (easily seen to be) equivalent to Theorem \ref{thm: mod p jump} via Lemma \ref{lem: mod p conj fil}. All   results in this subsection   were first proved by  Gee--Kisin; comparison of methods is briefly discussed in  Remark \ref{rem history intro}.

\begin{theorem}[{Gee--Kisin, cf. \cite{GK-ias}}] \label{thm: mod p jump}
Use Notation \ref{nota: crys final} (so $K$ is unramified, and $T$ is crystalline).
\begin{enumerate}
    \item Suppose  $n \notin \{ r_i+kp, k \in \bbz, 1\leq i\leq d \} \cap [0, r_d]$  (that is: if $n \geq r_d+1$ or $n \not\equiv r_i \pmod{p}$ for all $i$), then  
 \[\gr_n  \bargm_\HT =0 \]

 \item More precisely, let $0 \leq b_1\leq  \cdots \leq b_d$ be the jumps of $\fil_\bullet \bargm_\HT$ counted with multiplicities, then $b_i \leq r_d$ for each $i$ and
\[ \{ b_1, \cdots, b_d\} \equiv  \{r_1, \cdots, r_d\} \pmod p\]
in the sense that both sides define a same (un-ordered) set of elements in $\bbz/p\bbz$ with same multiplicities.
\end{enumerate} 
\end{theorem}
\begin{proof} 
It suffices to prove the stronger Item (2). Note the Frobenius height of $\bargm$ is still $r_d$, and hence $\gr_n  \bargm_\HT =0$ for $n >r_d$; thus $b_i \leq r_d$.  Now denote the sets $B=\{ b_1, \cdots, b_d\}$ and $R=  \{r_1, \cdots, r_d\}$. For $s \in \bbz/p\bbz$, let 
$\mu_B(s)$ be the multiplicity of $s$ in $\{ b_1, \cdots, b_d\} \pmod p$. Define $\mu_R(s)$ similarly. Thus, we want to prove that for each $s$,
\[ \mu_B(s) =\mu_R(s) \] 
The characteristic polynomial of $\theta$ on $\gm_\HT$ and hence also on $\bargmht$ is $\Pi_i (x-r_i)$. 
Thus for $ s\in \bbz/p\bbz$, the dimension of generalized eigenspace of eigenvalue $s$   is exactly $\mu_R(s)$.
By Theorem \ref{thm: mod p Sen op fil}, the induced action of $\theta$ on $\gr_n \bargm_\HT$ is scaling  by $n$. Thus  the dimension of generalized eigenspace of eigenvalue $s$ (which can be computed by taking graded pieces of $\bargm_\HT$) is also $\mu_B(s)$. Thus we conclude.
\end{proof}
%To see vanishing for $n >r_d$, it suffices to note the Frobenius height of $\bargm$ is still $r_d$. To see the other vanishing, use exactly the same concrete argument in Theorem \ref{thm: vanish integral}\eqref{item vanish n int}. Indeed, it suffices to prove the composite \[ \fil_n  \bargm_\HT \xrightarrow{\theta-n}\fil_{n-1}  \bargm_\HT  \into\fil_{n}  \bargm_\HT \] is bijective, where the first map comes from Theorem \ref{thm: mod p Sen op fil}. But $\theta-n$ is already bijective on $\bargm_\HT$: indeed, the characteristic polynomial of $\theta$ on $\gm_\HT$ is $\Pi_i (x-r_i)$; thus the eigenvalues of $\theta-n$ on $\bargm_\HT$ are precisely $r_i-n \pmod  p$, which are units.

%\begin{remark}\label{rem: delete}  \ghred{will delete} \ghgray{  As an interesting comparison to the   bound in Theorem \ref{thm: bound generator}, consider   the mod $p$ case, then   Theorem \ref{thm: mod p jump} implies: \[ \dim_k\gr_n \bargm_\HT\leq \sharp \{r_i \equiv n \pmod{p}\} \leq d \] } \end{remark}

\begin{notation} \label{nota: mod p bk}
Let $\bargm$ be a mod $p$ Breuil--Kisin module (not necessarily from reduction of a crystalline representation). 
Fix a basis $\vec e$, write $\phi(\vec e)=\vec e A$.
Since $k[[u]]$ is a PID, the matrix $A$ has a decomposition $A=X \cdot \diag(u^{a_1}, \cdots, u^{a_d}) \cdot Y$, where $X, Y$ are invertible matrices and $\diag(u^{a_1}, \cdots, u^{a_d})$ is a diagonal matrix with $0 \leq a_1 \leq \cdots \leq a_d$. (The elements $a_i$ are uniquely determined by $\bargm$.)
\end{notation}

\begin{lemma} \label{lem: mod p conj fil}
 Use Notation \ref{nota: mod p bk}.
 For $j \in \bbz$, let $\mathrm{mult}(j)$ be the multiplicity of $j$ in the set $\{a_1, \cdots, a_d\}$.
 Then
 \[ \dim_k \gr_n \bargm_\HT  =\mathrm{mult}(n)\]  
\end{lemma}
\begin{proof} One can find a basis $f_i^\ast$ of $\bargm^\ast$, such that
    \[ f_i^\ast =u^{a_i}g_i \]
    with $g_i$ forming a basis of $\bargm$. Thus
    \[ \fil^n \bargmast = \oplus_{i=1}^d \overline{\gs} \cdot  u^{\max\{ n-a_i, 0\}} \cdot f_i^\ast \] 
 Then one can easily compute other filtrations to conclude.
\end{proof}

\begin{theorem}[{Gee--Kisin, cf. \cite{GK-ias}}] \label{thm: gee kisin} 
Use Notation \ref{nota: crys final} (so $K$ is unramified, and $T$ is crystalline). Let $\bargm$ be the reduction of $\gm$, and   use Notation \ref{nota: mod p bk}. Then
\[ \{ a_1, \cdots, a_d\} \equiv  \{r_1, \cdots, r_d\} \pmod p\]
in the sense that both sides define a same (un-ordered) set of elements in $\bbz/p\bbz$ with same multiplicities.
  \end{theorem}
\begin{proof} 
Use Notation in Theorem \ref{thm: mod p jump}, then $\{ b_1, \cdots, b_d\}=\{a_1, \cdots, a_d\}$ (not just modulo $p$) by  Lemma \ref{lem: mod p conj fil}, and thus we can conclude using  Theorem \ref{thm: mod p jump}.
\end{proof}
%Denote the sets $A=\{ a_1, \cdots, a_d\}$ and $R=  \{r_1, \cdots, r_d\}$. For $n \in \bbz/p\bbz$, let  $\mu_A(n)$ be the multiplicity of $n$ in $\{ a_1, \cdots, a_d\} \pmod p$. Define $\mu_R(n)$ similarly. Thus, we want to prove that for each $s$, \[ \mu_A(s) =\mu_R(s) \]   Lemma \ref{lem: mod p conj fil} implies \[ \mu_A(s)   =\sum_{n \equiv s \pmod p} \dim_k \gr^n \bargm_\HT\] Since the characteristic polynomial of $\theta$ on $\bargm_\HT$ is $\Pi_i (x-r_i)$, $\mu_R(s)$ is equal to the summation of dimension of generalized $\theta$-eigenspaces of $\bargm_\HT$ with generalized eigenvalues congruent to $s$. But we know the induced action of $\theta$ on $\gr^n \bargm_\HT$ is $n$ by Theorem \ref{thm: mod p Sen op fil}; thus we must have \[ \mu_R(s)   =\sum_{n \equiv s \pmod p} \dim_k \gr^n \bargm_\HT\]

 %%\newpage
\section{Shape of integral Frobenius} \label{sec: frob matrix}

We first introduce  various Frobenius matrix conditions in \S \ref{subsec frob cond}. 
We then focus on the (weak) \emph{integral} conditions in \S \ref{subsec: weak frob}. 
The mod $p$ conditions will only be discussed in   \S \ref{sec:mod-p-shape}. 
The results of this section and  \S \ref{sec:mod-p-shape} leads to a conceptual reproof of a   difficult theorem    \cite{GLS14} by Gee, Savitt and the second named author, cf. \S \ref{subsec: GLS reproof}.

\subsection{Frobenius matrix conditions} \label{subsec frob cond}
  
  We introduce various Frobenius matrix conditions for Breuil--Kisin modules to facilitate discussions. These definitions only concern the $\phi$-operator; thus in this subsection, the modules do not necessarily come from integral semi-stable representations.
  
  \begin{notation} \label{nota general BK mod}
 (Allow $K$ to be ramified). Let $\gm$ be an (effective) Breuil--Kisin module (that is not necessarily attached to a semi-stable representation). Define $\gm^\ast, \gm_\dR, \gm_\HT$ etc.~ as in \S \ref{sec: conj fil}. As the triple $(\gs, E, \phi)$ satisfies the axioms in Notation \ref{nota:adapted}, one can define
 \[\fil^\bullet D_\dR:= \fil^\bullet \gm_\dR[1/p], \] 
 which is a filtered $K$-vector space; denote the jumps as $\{r_1 \leq \cdots \leq r_d\}$ and call them the Hodge--Tate weights of $\gm$.
  \end{notation}

\begin{defn}  \label{def: strong frob cond}
Use Notation \ref{nota general BK mod}.
 Let $\Lambda=\diag(E^{r_i})$ denote the diagonal matrix with diagonal  entries $E^{r_i}$. 
Consider the following conditions.
\begin{enumerate}
\item   \label{item weak frob old} Say $\gm$ satisfies the \emph{weak  Frobenius condition}, if there exists a (hence any) basis   $\vec e_1$ of $\gm$, such that $\phi(\vec e_1)=(\vec e_1) X_1\Lambda Y_1$ with $X_1, Y_1 \in \GL_d(\gs)$.
 
\item  \label{item strong frob} Say $\gm$ satisfies the \emph{strong  Frobenius condition}, if there  exists a (hence any) basis   $\vec e_2$ of $\gm$, such that $\phi(\vec e_2)=(\vec e_2) X_2\Lambda Y_2$ with $X_2, Y_2 \in \GL_d(\gs)$, and $Y_2 \pmod p \in \mathrm{GL}_d(k[[u^p]])$.

\item  \label{item strong mod p} 
Let $\bargm$ be the mod $p$ reduction of $\gm$. 
Say $\bargm$ satisfies the  \emph{strong mod $p$ Frobenius condition}, if there exists a (hence any) basis $\vec e_3$ of $\bargm$, such that $\phi(\vec e_3)=(\vec e_3) X_3\Lambda Y_3$ with $X_3  \in \GL_d({k[[u]]})$ and $Y_3 \in \mathrm{GL}_d(k[[u^p]])$.

 \item \label{item unaligned mod p} 
Let $\overline{\gn}$  be a mod $p$ Breuil--Kisin module (that is not necessarily the reduction of an integral Breuil--Kisin module; hence there is \emph{a priori} no notion of Hodge--Tate weights as in Notation \ref{nota general BK mod}.)
    Say $\overline{\gn}$  satisfies the  \emph{unaligned mod $p$ Frobenius condition}, if there exists a (hence any) basis $\vec e_4$ of $\overline{\gn}$, such that $$\phi(\vec e_4)=(\vec e_4) X_4DY_4, \text{ with } X_4  \in \GL_d(k[[u]]), Y_4 \in \mathrm{GL}_d(k[[u^p]]),$$
     and $D$ is a diagonal matrix with diagonal entries $u^{a_i}$ for some $a_i \geq 0$. Note the $a_i$'s are uniquely determined up to permutation, using the fact that $k[[u]]$ is a PID; the emphasis of this condition is on $Y_4$ since \emph{a priori} it is just some invertible matrix over $k[[u]]$. 
(Caution: the weak  Frobenius condition does not imply this condition; thus we refrain from using ``weak" here.) 
 \end{enumerate}
\end{defn}

\begin{remark} \label{rem: weak frob cond equiv}
\begin{enumerate}
\item  \label{item equiv strong} Definition \ref{def: strong frob cond}\eqref{item strong frob} is equivalent to the condition that there exists a (but not necessarily any) basis $\vec e$ of $\gm$, such that $\phi(\vec e)=(\vec e) X\Lambda Y$ with $X, Y \in \GL_d(\gs)$, and furthermore $Y\equiv I_d \pmod p$ (here $I_d$ is the identity matrix). Indeed, if Definition \ref{def: strong frob cond} \eqref{item strong frob} is satisfied, then  there exists some $C\in \mathrm{GL}_d(\gs)$ such that $ Y_2 (\phi(C))^{-1} \equiv 1 \pmod p$. One can use $\vec e =\vec{e_2}C^{-1} $ to verify the condition here. This is the condition used in \cite[Theorem 4.1]{GLS14}.

\item    Similar to the discussion in above, one can require $Y_3$ and $Y_4$ in  Definition \ref{def: strong frob cond}  to be the identity matrix, but now only for some (and not all) bases $\vec{e_3'}, \vec{e_4'}$. We also caution the differences here with the matrix decomposition in Notation \ref{nota: mod p bk}.
\end{enumerate}
 \end{remark}
% between Items \eqref{item strong frob} and (2)'). Item \eqref{item strong mod p} is equivalent to the condition that there exists  a basis $\vec e_5$ of $\bargm$, such that $\phi(\vec e_5)=(\vec e_5) X_5\Lambda$ with $X_5  \in \GL_d(\overline{\gs})$. 

The following lemma is obvious.

\begin{lemma} \label{lem: equiv strong frob} %\ghblue{updated, added 3rd item}
Use Notations in Definition \ref{def: strong frob cond}. The following   statements are equivalent:
\begin{enumerate}
\item $\gm$ satisfies strong  Frobenius condition; 
    \item $\gm$ satisfies weak Frobenius condition and $\bargm$ satisfies strong  mod $p$ Frobenius condition. 
     \item{ $\gm$ satisfies weak Frobenius condition and $\bargm$ satisfies unaligned mod $p$ Frobenius condition.} 
\end{enumerate} 
 (Caution: \emph{a priori}, the relevant bases $\vec e_i$ are not the same).
\end{lemma}
 
%\begin{proof}  It is easy to see that Items \eqr weak frob} and \eqref{item strong mod p} in Definition \ref{def: strong frob cond}  implies Item \eqref{item strong frob} there. \end{proof}

\subsection{Weak Frobenius condition} \label{subsec: weak frob}

\begin{lem}  \label{lem: weak frob equiv} 
Use Notation \ref{nota general BK mod} (thus $\gm$ is not necessarily attached to a semi-stable representation).
 The following statements are equivalent:
\begin{enumerate}
\item \label{item weak frob}
$\gm$ satisfies weak Frobenius condition (Definition \ref{def: strong frob cond}\eqref{item weak frob old});

\item \label{item adap basis}  $\fil^\bullet \gmast$ has an adapted basis as in Notation \ref{nota:adapted};

\item The map $ \fil^i \gm_\dR \into \gm_\dR \cap \fil^i D_\dR $ is bijective for each $i$;

\item \label{item prop gls14 torfree} $\gr^i \gm_\dR=\gr_i \gm_\HT$ is torsion-free for each $i$.
\item The map  $ \fil_i \gm_\HT \into \gm_\HT \cap \fil_i \cm_\HT$ is bijective for each $i$;
\end{enumerate}   
\end{lem}
\begin{proof} 
With the axiomatic Lemma \ref{lem:adapted} applicable here, it   remains to prove  \eqref{item weak frob} $\Leftrightarrow$ \eqref{item adap basis}.
 For \eqref{item weak frob} $\Rightarrow$ \eqref{item adap basis}: if $\phi(\vec e) =(\vec e) X \Lambda Y$, then $(\vec e) X \Lambda $ is an adapted basis of $\gmast$. 
For \eqref{item adap basis} $\Rightarrow$ \eqref{item weak frob}: this is essentially the argument in second paragraph of the proof of  \cite[Thm. 4.20]{GLS14}, which we repeat here for convenience. Indeed choose any basis $\vec e$ of $\gm$, and write $\phi(\vec e)=\vec e A$; then there exists a matrix $B$ over $\gs$ such that $AB=E^{r_d}$.
  Let $\vec{\hat{e}}$ be an adapted basis of $\gmast$. As $(\vec e)A$ is also a basis of $\gmast$, there is an invertible matrix $Y$ such that
  \[(\vec e)A =\vec{\hat{e}}Y \]
  Consider $\fil^{r_d}\gmast=E^{r_d}\gm$, then there is another invertible matrix $X$ such that
  \[ (\vec e)E^{r_d}X= \vec{\hat{e}}E^{r_d} \Lambda^{-1}\]
  Then one easily deduces $A=X\Lambda Y$.
   \end{proof}
 
%The equivalences among (2), (3), (4), (5) follow from  Lemma \ref{lem:adapted}.

\begin{remark} \label{rem frob shape f gauge}
Suppose $\gm$  comes from an integral \emph{crystalline} representation  $T$, then the equivalent conditions in Lemma \ref{lem: weak frob equiv} are further equivalent to the following:
\begin{itemize}
\item The  $F$-gauge (cf. \cite{Bha22}) corresponding to $T$ is a vector bundle (on the syntomic stack $(\ok)^{\mathrm{syn}})$.
\end{itemize}
(We learn of the following stacky argument from Bhargav Bhatt).
Consider
\[ \spf(\wt{A}):=\spf( W(k)[[x]][u,t]/(ut-E(x)) ) \to \o_K^{\caln} \to \o_K^{\syn} \]
where the first map is the    faithfully flat cover as in \cite[Example 5.5.20]{Bha22} (we follow notation there and use $x$ as the variable for $\gs$), and the second map is the   \'etale cover as in \cite[Definition6.1.1]{Bha22}. It thus suffices to prove the pull-back   $\cale_{\wt A}$ is a vector bundle. The reduction $\cale_{\wt A}/(t, u)$ is precisely $\oplus_{i \in \bbz} \gr_i \gm_\HT$ and hence is free over $\ok$ by Condition \eqref{item prop gls14 torfree} of Lemma \ref{lem: weak frob equiv}. The module $\cale_{\wt A}$ thus has depth 3, and hence is projective by the Auslander--Buchsbaum formula.    

One can also argue using slightly more concrete  sheaf theoretic languages (as developed in e.g. \cite{GuoLigauge}). For example,  see \cite[Prop. 2.27, Lem. 2.28]{IKY2} for some similar discussions. 
\end{remark}

%%%\newpage
\subsection{Strong Frobenius condition: the case with weights $\leq p$} \label{subsec: GLS reproof}

In this short subsection, we reprove a   theorem in \cite{GLS14}, cf.~ Theorem \ref{thm: reprove gls14}. The proof actually uses results in   \S \ref{sec:mod-p-shape} (on mod $p$ shape of Frobenius); we include this subsection here because it is short, and the statement is an ``integral" one.
 We remark that this result is a most   technical theorem in \cite{GLS14}, and is a \emph{core} reason that Serre weight conjecture in (unramified) $\GL_2$-case can be fully proved in \emph{loc. cit.}.  Note the bound $r_d \leq p$ in Theorem \ref{thm: reprove gls14} is necessary by  \cite[Example 6.8]{GLS14}.

\begin{thm} \label{thm: reprove gls14} \cite[Thm. 4.1]{GLS14}. 
 Suppose $K$ is unramified, $T$ is crystalline, and $r_d \leq p$. Then $\gm$ satisfies the strong Frobenius condition in Definition \ref{def: strong frob cond}\eqref{item strong frob} (cf. also Remark \ref{rem: weak frob cond equiv}\eqref{item equiv strong}). 
\end{thm}
\begin{proof}
With Lem. \ref{lem: equiv strong frob} in mind, we first verify  weak Frobenius condition: by the filtration  Lemma \ref{lem: weak frob equiv}, it suffices to note $\gr_\bullet \gm_\HT$ is torsion-free in this case by Example \ref{ex: torfree p 2p}\eqref{item torfree p case}.
We then need to verify the  unaligned mod $p$ Frobenius condition: it will be proved in Theorem \ref{thm1 mod p frob}, indeed without restriction on $r_d$. 
\end{proof}
% With Lem. \ref{lem: equiv strong frob} at hand:  \begin{itemize}  \item[Step 1.]   We first verify the weak Frobenius condition in \S \ref{subsec: weak frob}. By a filtration lemma \ref{lem: weak frob equiv} (essentially)  from \cite{GLS14}, it suffices to verify $\gr^i \gm_\dR$ is torsion-free (the equivalent condition in Lemma  \ref{lem: weak frob equiv}\eqref{item prop gls14 torfree}): this is already proved in Example \ref{ex: torfree p 2p}\eqref{item torfree p case}.   \item[Step 2.] {In \S \ref{sec:mod-p-shape},  we will prove  {unaligned}  mod $p$ Frobenius condition without the restriction $r_d \leq p$.}   
 %We then verify the  \emph{strong}  mod $p$ Frobenius condition in \S \ref{subsec: mod p frob}.   Lemma \ref{lem: ai ri match} implies that in the $[0,p]$-case, the  \emph{strong}  mod $p$ Frobenius condition is equivalent to the \emph{unaligned}  mod $p$ Frobenius condition. To verify the later, we use a filtration lemma \ref{lem: mod p Frob two fil} from \cite{Bar20a}: it then reduces to prove a certain filtered map  $M_k \simeq M_k^\phi$ is a filtered isomorphism; this reduces to computation of dimensions of graded, which is carried out in Proposition \ref{prop: 525 holds}.  \end{itemize} 

\begin{remark} \label{rem diff GLS}
We highlight   the differences between our strategy with that of \cite{GLS14}.
\begin{enumerate}
\item The proof of \cite[Thm. 4.1]{GLS14} starts with a basis  for $\fil^i \gm_\dR$, and then \emph{lifts} it to a basis for $\fil^i \gmast$.
The construction repeatedly  uses the (\emph{rational}) operator $N$ on $\cald$,  relying particularly on the  Griffiths transversality it satisfies. 
Indeed, very roughly, the lifting process uses an ``approximation" technique, by \emph{truncating} the ``exponential" of the $N$-operator, which then involve very delicate \emph{integrality} analysis.

\item In contrast, the main innovation in our reproof lies on an \emph{extra} filtration: the conjugate filtration; the extra ``symmetry" it satisfies (the Sen operator with its   shrinking properties) leads to substantial simplification of the argument.  Indeed, this reproof shows that for applications, the  shrinking of the (integral) Sen operator not only can ``substitute" the use of Griffiths transversality of $N$, its stronger symmetry (cf. Remark \ref{rem: nnabla and N}\eqref{item nnabla inherit}) could lead to stronger consequences. 
 \end{enumerate}
\end{remark}

%\item Finally, we point that our argument depends on some (known)  \emph{filtration lemmas}: cf. Lem. \ref{lem: weak frob equiv}   and Lem. \ref{lem: mod p Frob two fil}. These filtration lemmas are very general results: they are always valid for     semi-stable $T$ and any $r_d$, and  Lem. \ref{lem: weak frob equiv} is also valid for ramified $K$.

\begin{remark}  \label{rem: future}
We comment on the \emph{ramified} case of Theorem \ref{thm: reprove gls14}.
\begin{enumerate}
\item  In the sequel \cite{GLS15} to \cite{GLS14}, the authors also fully prove Serre weight conjecture in the \emph{ramified} $\GL_2$-case, i.e., when the relevant $K$ is ramified. Similar to the scenario in \cite{GLS14}, the central technical theorem is to prove the \emph{pseudo-Barsotti--Tate} crystalline representations satisfy a certain ``strong Frobenius condition", cf. \cite[Thm. 2.4.1]{GLS15}.

\item  In a sequel \cite{GLram}, we can also use filtered Sen theory to reprove \cite[Thm. 2.4.1]{GLS15}. Here is a very important catch: when $K$ is ramified, we no longer have $\Theta=\theta_\kinfty$ (cf. Definition \ref{def: amplified Sen})! Indeed, the results from \S \ref{sec: vanishing} are not readily valid any more. A key idea is to consider crystalline representations  defined over a coefficient field $E/\qp$ where $E$ contains all Galois conjugates of $K$; this leads to other ``normalizations" of Sen operators making the ideas of the current paper useful.  
\end{enumerate}
 \end{remark}

%%\newpage
 \addtocontents{toc}{\ghblue{Sen theory: prismatic approach}}
% \section*{====pris Sen====}
  
\section{Prismatic interpretation of filtered Sen theory} \label{sec: pris interpretation}

In this section, we use prismatic arguments to ``reconstruct" filtered Sen theory in \S \ref{sec: fil sen}. The proof builds on the classification of (log-) prismatic $F$-crystals and Hodge--Tate crystals, as well as their connection with Sen theory over the Kummer tower.
We shall be brief here, and refer the readers to \cite{BS22, Kos21} for foundations of (log-) prismatic site.  Let $\star \in \{ \emptyset, \log \}$, and let $(\ok)_{\pris, \star}$ be the absolute (log-) prismatic site of $\spf \ok$. Let $\opris, \calI_\prism$ be the structure sheaf and the ideal sheaf, let $\baropris=\opris/\calI_\prism$ be the Hodge--Tate structure sheaf. Let $(\gs, (E), \star)$ be the  Breuil--Kisin ($\star$-) prism.

% \begin{remark} The relation between the prismatic argument here and our non-prismatic (locally analytic) argument is discussed in Proposition \ref{prop: pris sen lav sen}. Thus in particular, we could well use the construction in this section as \emph{the definition} of filtered Sen theory. We have chosen to keep the writing style of this paper, cf. Rem. \ref{rem: pris vs nonpris}.  \end{remark}

%In this section, we discuss filtrations on prismatic crystals associated to integral semi-stable representations.  We shall be only sketchy here, as the \emph{technical}  cores of this paper, i.e., the ensuing sections (except the final section), are written in a completely ``non-prismatic" fashion. However, the \emph{prismatic ideas} presented in this section are of utmost guiding importance: without these prismatic constructions, we would not have discovered the various \emph{integral} (and filtered) Sen operators in the ensuing sections. Finally, we want to remark that in order for these ideas to become ``reality", we will have to relate these ``prismatic Sen operators" to more ``classical" objects, i.e., Breuil--Kisin modules and their variants; whence the ensuing sections. Indeed, in the final section, we will continue the discussions in this section, and relate the prismatic and non-prismatic operators.

\subsection{Prismatic crystals}
\begin{theorem} 
Let $\star \in \{ \emptyset, \log \}$. 
There is an equivalence of categories
\[ \vect^{\phi, \mathrm{eff}}((\ok)_{\pris, \star}, \opris) \simeq \rep_\zp^{\star-\crys, \geq 0}(\gk) \]
Here, the LHS is the category of effective $F$-crystals on the absolute ($\star$)-prismatic site of $\ok$, and the RHS is the category of integral $\star$-crystalline representations whose Hodge--Tate weights are $\geq 0$.
\end{theorem}
 \begin{proof}     The crystalline case is first proved in \cite{BS23}. The semi-stable case is first proved in  \cite{DL23} using absolute log-prismatic site in \cite{Kos21}; a second proof in the semi-stable case is obtained in \cite{Yao}. 
 \end{proof}

We now review results on   Hodge--Tate prismatic crystals. 
\begin{defn} \label{defnnhtendo} 
Recall as in Definition \ref{def: amplified Sen}:
 \begin{equation*}
a=
\begin{cases}
   E'(\pi), &  \text{if } \star=\emptyset \\
  \pi E'(\pi), &  \text{if } \star=\log
\end{cases}
\end{equation*}  
(As noted in Convention \ref{conv minus plus}, this is opposite to \cite[Definition 1.9]{GMWHT}).
Let  $\End_{\calO_K}^{\star-\nht}$ (resp. $\End_{K}^{\star-\nht}$)  be the category consisting of pairs $(M,f_M)$ which we call a \emph{module equipped with $a$-small endomorphism}, where
    \begin{enumerate}
        \item[(1)] $M$ is a finite free $\calO_K$-module (resp. $K$-vector space), and
        \item[(2)] $f_M$ is an $\calO_K$-linear (resp. $K$-linear) endomorphism of $M$ such that 
        \begin{equation}
            \lim_{n\to+\infty}\prod_{i=0}^{n-1}(f_M-ai) = 0.
        \end{equation} 
            \end{enumerate}
\end{defn}

 \begin{theorem}  \label{thmintrostrat}
  Evaluation on the Breuil--Kisin (log-) prisms induce bi-exact equivalences of categories  
 \[ \Vect((\calO_K)_{\Prism, \star},\overline \calO_{\Prism}) \simeq  \End_{\calO_K}^{\star-\nht}                                                    \]
  \[ \Vect((\calO_K)_{\Prism, \star},\overline \calO_{\Prism}[1/p]) \simeq  \End_{K}^{\star-\nht}                                                    \]
 \end{theorem}
 \begin{proof} In the prismatic setting (i.e., $\star=\emptyset$), these are    proved    in \cite{BL1, BL2} using a stacky approach; cf. also the work of  \cite{AHLB1}. The full case is independently proved in   \cite{GMWHT}, using a site-theoretic approach.
 \end{proof}

\subsection{Prismatic Sen operators and filtrations}
\begin{notation}
Let $T \in \rep_\zp^{\star-\crys, \geq 0}(\gk)$, let $\bm$ be the corresponding (effective) $F$-crystal, and let $\gm$ be the evaluation of $\bm$ on ($\star$)-Breuil--Kisin prism $(\gs, (E), \star)$.
Then the data of $\bm$ translates into the  data of a stratification, i.e., a $\phi$-equivariant isomorphism
\[ \varepsilon: \gs^1_\star \otimes_{p_0, \gs} \gm   \simeq  \gs^1_\star \otimes_{p_1, \gs} \gm \]
 satisfying the usual cocycle condition; here $\gs^1_\star$ is the co-product of $(\gs, (E), \star)$ inside $(\ok)_{\pris, \star}$, and $p_i$ are the face maps.
\end{notation}

\begin{notation} \label{nota: dd data} 
Starting from $\varepsilon$, we construct the following isomorphisms (always satisfying cocycle conditions).
\begin{enumerate}
\item Recall we always identify $\gmast\subset \gm$ as a submodule. Since $\varepsilon$ is $\phi$-equivariant, it restricts to an isomorphism 
\[ \varepsilon: \gs^1_\star \otimes_{p_0, \gs} \gmast   \simeq  \gs^1_\star \otimes_{p_1, \gs}  \gmast  \]
This induces a sub-crystal $\bm^\ast$ of $\bm$. Note it is no longer an $F$-crystal, since $\phi$ on $\gmast$ is a $\phi(E)$-isogeny.

\item Take Nygaard filtrations on both sides of above.  
Note $p_i: \gs \to \gs^1_\star$ is flat  for $i=0, 1$, thus Lemma \ref{lem: flat base change} implies
\[ \Fil^n (\gs^1_\star \otimes_{p_i, \gs} \M^*)\simeq \gs^1_\star \otimes_{p_i, \gs} \Fil^n \M^* \]
Thus we  obtain
\[ \varepsilon:  \gs^1_\star \otimes_{p_0, \gs} \fil^n\gmast   \simeq  \gs^1_\star \otimes_{p_1, \gs}  \fil^n\gmast  \]

\item Taking graded of above, we obtain
\[  \gs^1_\star \otimes_{p_0, \gs} \gr^n\gmast   \simeq  \gs^1_\star \otimes_{p_1, \gs}  \gr^n\gmast  \]
 
\item    Forget about  Frobenius isogeny structure on $\bm$  and reduce modulo the prismatic ideal sheaf; one then obtains a Hodge--Tate crystal
  \[ \bm_\HT \in \Vect((\calO_K)_{\Prism, \star},\overline \calO_{\Prism}) \]
  The corresponding stratification is precisely the reduction of $\varepsilon$:
  \[ \varepsilon_\HT: \gs^1_\star \otimes_{p_0, \gs} \gm/E   \simeq  \gs^1_\star \otimes_{p_1, \gs} \gm/E \]
  
  \item Using $\gr^n \gmast \simeq \fil_n^\rmconj \gm_\HT$, stratifications from above two items fit into the following commutative diagram
\begin{equation} \label{eq: dd data fil conj}
\begin{tikzcd}
{\gs^1_\star \otimes_{p_0, \gs} \fil_n^\rmconj \gm_\HT } \arrow[rr, "{  \simeq}"] \arrow[d, hook] &  & {\gs^1_\star \otimes_{p_1, \gs}   \fil_n^\rmconj \gm_\HT } \arrow[d, hook] \\
{\gs^1_\star \otimes_{p_0, \gs}  \gm_\HT } \arrow[rr, "{  \simeq}"]                               &  & {\gs^1_\star \otimes_{p_1, \gs}  \gm_\HT }                                
\end{tikzcd}
\end{equation} 
\end{enumerate}
\end{notation}

  \begin{construction} \label{cons: conj fil subHT crystal}
  Theorem \ref{thmintrostrat} translates the stratification $\varepsilon_\HT$ on $\gm_\HT$ to a linear operator $$\Theta_\prism: \gm_\HT \to \gm_\HT$$
   which we call the ``prismatic Sen operator". 
  The commutative diagram \eqref{eq: dd data fil conj}  says that $\fil_n^\rmconj \gm_\HT$ induces a Hodge--Tate sub-crystal; equivalently,   $\Theta_\prism$ stabilizes $\fil_n^\rmconj \gm_\HT$, inducing a sub-object
\[ (\fil_n^\rmconj \gm_\HT, \Theta_\prism) \into  ( \gm_\HT, \Theta_\prism) \]
  \end{construction}
%This operator is known to be related to Sen theory over the Kummer tower developed in \cite{GMWHT}, cf. \S \ref{sec: sen kummer}.
   
 %  \begin{notation} \label{nota bmastht}     Note \[ \varepsilon^\ast: \gs^1_\star \otimes_{p_0, \gs} \gmast   \simeq  \gs^1_\star \otimes_{p_1, \gs}  \gmast  \]  cannot induce a $F$-crystal, as $\phi$ on $\gmast$ is no longer a $E$-isogeny, but a $\phi(E)$-isogeny. However, $\varepsilon^\ast$ still induces a prismatic crystal (without $F$-structure), which we denote as $\bm^\ast$. In particular, we can take reduction, and obtain a Hodge--Tate crystal $\bm^\ast_\HT$.     \end{notation}

 \begin{prop} \label{prop: pris sen lav sen}
 The prismatic Sen operator in Construction \ref{cons: conj fil subHT crystal}, (up to $\pm 1$-scaling) 
\[ \Theta_\prism: \gm_\HT \to \gm_\HT \]
 is exactly the (locally analytic)  amplified  Sen operator in  Definition \ref{def: amplified Sen}
 \[ \Theta: \gm_\HT \to \gm_\HT \]
 \end{prop}
   \begin{proof}
   It suffices to prove it after inverting $p$; then this is proved in \cite[Theorem 8.2]{GMWHT}.
   \end{proof}

   Consider the sub-crystal $\bm^\ast$ induced by stratification on $\gmast$,  and take reduction, we  obtain a Hodge--Tate crystal $\bm^\ast_\HT$. (Caution: the map $\gmast/E \to \gm/E$ is not injective, so this cannot be a sub-crystal of $\bm_\HT$).

\begin{lemma} \label{lem: bmast ht trivial} The crystal  $\bm^\ast_\HT$ is a trivial Hodge--Tate crystal (in the sense that it is isomorphic to $(\baropris)^{\oplus d}$).
\end{lemma}
\begin{proof} 
By definition of stratification, the map
\[ \varepsilon: \gs^1_\star \otimes_{p_0, \gs} \gm  \simeq  \gs^1_\star \otimes_{p_1, \gs}  \gm\] 
composed with the degeneracy map $\sigma_0: \gs^1_\star \to \gs$ becomes the identity map on $\gm$.
That is to say, for any $x\in \gm$,
\[ \varepsilon(x) \subset x+ \ker \sigma_0 \otimes_{p_1, \gs}  \gm\] 
To see $\bm^\ast_\HT$ is   trivial, it reduces to check that 
\[ \phi(\ker \sigma_0) \subset E\gs^1_\star;\]
this is proved in \cite[Lem. 2.11]{GMWHT}.   
\end{proof}
 
 \begin{remark} \label{rem: other triv pf}
We   record two other proofs of Lemma \ref{lem: bmast ht trivial}  which we find interesting.
\begin{enumerate}
\item (A  non-prismatic (locally analytic) proof). Note it suffices to show the (integral) Sen operator on $\gmast/E$ is zero map.
 By Lem. \ref{lem: tau range gmast},   for each $i \geq 1$,
\[ (\tau-1)^i(\gmast) \subset (\phi(\fkt))^i \gm^\ast_\inf \subset \fkt^i\cdot E \gm^\ast_\inf\]
Lem. \ref{lem: log tau converge} (with $b=1, Y=E\gm^\ast_\inf$) implies 
\[ \frac{\log \tau}{\fkt} (\gmast/E) =0 \]

\item (A stacky proof that we learnt from Bhargav Bhatt). Let $X=\spf \ok$ (indeed, the argument below works for any $p$-adic formal scheme $X$), and consider the prismatic case. We have the following commutative diagram:
\[
\begin{tikzcd}
X^\HT \arrow[d] \arrow[r]  & X \arrow[d] \\
X^\prism \arrow[r, "\phi"] & X^\prism   
\end{tikzcd}
\]
This is discussed in detail in \cite[Proposition 3.6.6]{BL1} when $\ok=\zp$ (for example, the bottom row is induced by $\phi$ on prisms); the general case is similar. 
Now consider $\bm$ as an object living on the  bottom right corner; its pull-back to the top left corner is precisely  $\bm^\ast_\HT$, which is ``trivial" in the sense it also comes from pull-back of a quasi-coherent sheaf on $X$.
\end{enumerate} 
\end{remark}
% Note the pair $(\gmast, \gm^\ast_\inf)$ is in some sense a ``Breuil--Kisin $\gk$-module without $\phi$".

% To avoid confusion, let us denote $\wt{\gm}=\gmast, \wt{\gm}_\inf=\gm_\inf^\ast$. Note the pair $(\wt{\gm}, \wt{\gm}_\inf)$ is in some sense a ``Breuil--Kisin $\gk$-module without $\phi$". Using Lem. \ref{lem: tau range gmast}, we see for each $i \geq 1$, \[ (\tau-1)^i(\wt{\gm}) \subset (\phi(\fkt))^i \wt{\gm}_\inf \subset \fkt^i\cdot E\wt{\gm}_\inf\] Lem. \ref{lem: log tau converge} (with $b=1, Y=E\minf$) implies  \[ \frac{\log \tau}{\fkt} (\wt{\gm}/E) =0 \] 

   \begin{construction} \label{cons: prismatic shift} We now show how to recover the shifted Sen operators using prismatic argument. 
Since $\bm^\ast_\HT$ is a trivial Hodge--Tate crystal  by Lemma \ref{lem: bmast ht trivial}, we have  
\[(\varepsilon -1)(\gs^1_\star \otimes_{p_0, \gs} \gmast) \subset E\gs^1_\star \otimes_{p_1, \gs}  \gmast.\]
Since $\varepsilon$ is $\gs^1_\star$-linear, we have
\begin{equation} \label{eq enm}
 (\varepsilon -1)(E^n\gs^1_\star \otimes_{p_0, \gs}  \gm) \subset E^n\gs^1_\star \otimes_{p_1, \gs}  \gm
\end{equation} 
Taking intersection leads to a map
\begin{equation}\label{eq inter enm mast}
 \varepsilon -1: \gs^1_\star \otimes_{p_0, \gs} \fil^n\gmast \to  \gs^1_\star \otimes_{p_1, \gs}  E \fil^{n-1}\gmast.
\end{equation} 
We can form a  commutative diagram
\begin{equation} \label{eq diag pris shift}
\begin{tikzcd}
{ \gs^1_\star \otimes_{p_0, \gs} \gr^n\gmast} \arrow[d, hook] \arrow[rr, "\overline{\varepsilon-1}"] &  & {\gs^1_\star \otimes_{p_1, \gs}  (E \fil^{n-1}\gmast/E \fil^{n}\gmast)} \arrow[d, hook] \\
{\gs^1_\star \otimes_{p_0, \gs} E^n\gm/E^{n+1}\gm} \arrow[rr, "\overline{\varepsilon-1}"]            &  & {\gs^1_\star \otimes_{p_1, \gs} E^n\gm/E^{n+1}\gm}                                     
\end{tikzcd}
\end{equation}
where the top (resp. bottom) row comes from taking graded of Eqn \eqref{eq inter enm mast} (resp. Eqn \eqref{eq enm}).
It is well-known that 
\[ \overline{\varepsilon}: \gs^1_\star \otimes_{p_0, \gs} E^n\gm/E^{n+1}\gm \to \gs^1_\star \otimes_{p_1, \gs} E^n\gm/E^{n+1}\gm\]
induces the \emph{Breuil--Kisin} twist $\bm_\HT\{n\}$ of the Hodge--Tate crystal $\bm_\HT$. Identifying $E^n\gm/E^{n+1}\gm$ with $\gm_\HT$, the prismatic Sen operator associated to $\bm_\HT\{n\}$ is
\[ \Theta_\prism-na:  \gm_\HT \to \gm_\HT \]
 Recall $\gs^1_\star/E \simeq \ok[X ]^{\wedge_p}_{\mathrm{pd}}$ is a $p$-complete PD polynomial ring with one variable (cf. \cite[Proposition 2.12]{GMWHT}); thus the bottom row of diagram \eqref{eq diag pris shift} is induced by a converging summation
\[  \sum_{i\geq 1} X^{[i]} f_i  \]
where $f_i:   \gm_\HT \to   \gm_\HT$ are $\ok$-linear maps for each $i$ and $X^{[i]}=X^i/i!$.
In fact, it is known that (cf. \cite[Theorem 3.6]{GMWHT}) 
\[ f_i=\prod_{0 \leq j < i} (\Theta_\prism-na -ja) \]
Since the   vertical arrows are injective, the top row is also induced by the summation; that is to say, the image of  $f_i$ lands inside $\fil^\rmconj_{n-1} \gm_\HT$. Thus the $f_1=\Theta_\prism-na$ map satisfies:
\begin{equation}\label{eq prismatic shifting operator}
\Theta_\prism-na:  \fil_n^\rmconj \gm_\HT \to \fil_{n-1}^\rmconj \gm_\HT
\end{equation}  
As a consequence of Proposition \ref{prop: pris sen lav sen}, this is exactly the operator  in Theorem \ref{thm: integral sen integral conj fil}.
  \end{construction}

   \begin{remark}
   Construction \ref{cons: prismatic shift}  has obvious analogues in the rational case and the mod $p$ case, thus we can also use prismatic argument to recover the shifted Sen operators (and hence filtered Sen theory) in Theorems \ref{thm: rational sen shift} and \ref{thm: mod p Sen op fil}; we leave details to the interested readers.
   \end{remark}

%%\newpage  
 \section{Truncated Sen operator} \label{sec:truncated-sen}

In \S \ref{sec: integral Sen}, we used normalizations of $\log \tau$ to define a Sen operator. In this section, we show that in the  crystalline  case, a \emph{truncated operator}, namely a normalization of ``$\tau-1$" (then modulo maximal ideal) would already suffice to recover the \emph{mod $p$ Sen operator}. This fact, Proposition \ref{prop MinWang}, is first discovered by Yu Min and Yupeng Wang  (in an earlier work not intended for publication); we thank them for allowing us to include this useful result here.
This ``truncated" operator will be further  ``\emph{stabilized}" in the next two sections \S \ref{sec:pris-max-ring} and \S  \ref{sec:p-GT-A}, cf. Remark \ref{rem:need-stable}.
 In this section, we allow $K$ to be ramified.

%\cite[Theorem 4.6]{MW21} (a preprint not for publication?)

\begin{construction} \label{cons:GMWHT-tau}
Let $\cale \in \Vect((\calO_K)_{\Prism, \star},\overline \calO_{\Prism})$ be a Hodge--Tate crystal. As reviewed in Theorem \ref{thmintrostrat}, it corresponds to an $a$-small endomorphism $f: M \to M$ (cf. Definition \ref{defnnhtendo}), where $M=\cale((\gs, (E), \star))$. Denote
\[\hat{M}: = \cale(((\ainf)^{G_L}, (E), \star)) =M\otimes_\ok \o_{\hat{L}},  \]
which admits $\gal(L/K)$-action by functoriality. 
Use \cite[Eqn (5.3) in Proposition 5.10]{GMWHT}, let $g=\tau$, then for $x \in M$, we have
\[ \tau(x)=  \sum_{n\geq 0} \left( \left(\prod_{i=0}^{n-1}(f- ia)(x)\right) \cdot  ( X(\tau))^{[n]} \right) \]
here 
\[X(\tau)=  \frac{\pi E'(\pi)}{a}\theta_\fon(\frac{[\zetaflat]-1}{E([\pi^{\flat}])})   \]
\end{construction}

% The following proposition follows easily from above valuation computation.

\begin{proposition}[\textnormal{Min--Wang}] \label{prop MinWang}
Use notations in Constructions \ref{cons:GMWHT-tau}.
Suppose $\star=\emptyset$ (that is, $\cale$ is defined on the prismatic site). Then the $\tau$-action   induces an operator 
\[ \frac{\tau-1}{X(\tau)}: M\to M \otimes_\ok \o_{\hat{L}}.\]
Consider reduction modulo  the maximal ideal of $\o_{\hat{L}}$, then its image falls inside $M/\pi M$, and the reduced operator is exactly mod $\pi$ reduction of $f: M \to M$.   That is:
\[ \frac{\tau-1}{X(\tau)} \pmod{\fkm_{\o_{\hat{L}}}}  \equiv f \pmod{\pi} \] 
\end{proposition}
\begin{proof}
In  the prismatic case,  $a=E'(\pi)$, and the $p$-adic valuation of $X(\tau)$ is 
\[ v_p(\pi \cdot (\zeta_p-1))=\frac{1}{e} + \frac{1}{p-1}. \]
Thus the $p$-adic valuation of $((X(\tau))^{n-1})/n!$ is positive and converges to infinity. We can easily conclude.
\end{proof}
%\begin{proof}
 
%by Legendre's formula, this  bounds $v_p(n!)$, and we are done.  recall  Legendre's formula \[ v_p(n!) =\frac{n-s_p(n)}{p-1}  \leq \frac{n-1}{p-1} \] here $s_p(n)$ is the sum of digits in  $p$-adic expansion of $n$.
%\end{proof}

\begin{remark}
\begin{enumerate}
\item Note in Proposition \ref{prop MinWang}, it is necessary to take reduction modulo the maximal ideal $\o_{\hat{L}}$ and hence reduction modulo $\pi$ (not modulo $p$) of $f$. Compare with the situation in Remark \ref{rem not mod pi}. This general mod $\pi$ operator will be further used in our investigation in the ramified case, cf. Remark \ref{rem: future}.

\item In the log-prismatic case, the valuation of $X(\tau)$ is $1/(p-1)$; thus the argument of Proposition \ref{prop MinWang} breaks.
\end{enumerate}
\end{remark}

\begin{cor} \label{cor-crys-truncate-sen}
Let $T \in \rep^\crys_\zp(\gk)$.
One can define an operator
\[  \delta_\tau:=\frac{\tau-1}{u\frac{p}{E(0)}\fkt}: \gm \to \minf;\]
After modulo $E$ then further modulo $\fkm_{\o_C}$, the image falls inside $\bargmht \subset \gminf/(E,\fkm_{\o_C})$, and it is the same as the mod $\pi$ amplified Sen operator. That is:
\[ \delta_\tau \equiv \Theta \pmod{E,\fkm_{\o_C}} \]
\end{cor}
\begin{proof}
We first consider another variant (cf. Remark  \ref{rem:renorm_fkt} for comment):
\begin{equation} \label{eq: delta-variant}
\delta':=\frac{\tau-1}{uw}: \gm \to \gminf, 
\end{equation}
where  $w=([\zetaflat]-1)/E \in \ainf$ is defined in Notation \ref{nota: ring for mod}.
This operator is well-defined by Lemma \ref{lem: tau range gmast}. Using Proposition \ref{prop MinWang}, and noting    $\theta_\fon(uw)$ is exactly $X(\tau)$, it is easy to see
\[ \delta' \equiv \Theta \pmod{E,\fkm_{\o_C}} \]
Back to the $\delta_\tau$-operator in the statement, it suffices to consider its  difference with $\delta'$: 
\[ \alpha= \frac{E(0)}{p} \frac{w}{\fkt} \in \ainf.\]
It is easy to see that $\theta_\fon(\alpha) =\theta_\fon(\phi(\lambda)) \in \o_K^\times$; indeed 
$\theta_\fon(\phi(\lambda)) \equiv 1 \pmod{\pi}$. From this we can conclude.
\end{proof}

\begin{remark} \label{rem:renorm_fkt}
We quickly comment that we shall prefer to use $\delta_\tau$ (normalized with $\fkt$) rather than $\delta'$ (normalized with $w$) for some normalization convenience when considering its intertwining with the Frobenius operator, cf. Remark \ref{rem:preferdeltatau} for the reason. Nonetheless,  the element $w$ has its own computational convenience because of its natural appearance in prismatic rings.

%It is convenient to use $w$ in Corollary  \ref{cor-crys-truncate-sen} as it is directly related with the variable $X(\tau)$ in the prismatic ring. One could also change it to $\fkt$ and re-normalize to get a similar statement. Then its reduction module $(E, \fkm_\oc)$ is the same as that of $\delta_\tau$. Indeed, it   suffices to consider the difference between them 
\end{remark}

\begin{remark} \label{rem:need-stable}
The ``truncated" (and purely integral) operator $\delta_\tau$ in Corollary  \ref{cor-crys-truncate-sen} is convenient in that it does not involve     $p$ in the denominators (as in $\log\tau$); however, it is not ``\emph{stable}" on $\gm$ (compare with the ``stable" operators of Breuil and Kisin in \S \ref{sec: BK mod}), hence is inconvenient to study its interaction with filtrations. 
One naturally hopes to extend the ring $\gs$ to ``\emph{stabilize}" $\delta_\tau$: it turns out to be a very tricky construction, as one needs to balance between the size of the ring and the properties we desire to have. This will be carried out in the next two sections \S \ref{sec:pris-max-ring} and \S  \ref{sec:p-GT-A}.
\end{remark}

%%\newpage

\section{Some prismatic rings} \label{sec:pris-max-ring}
 
 In this section, we study some prismatic rings and the Galois actions on them. This is in preparation for \S \ref{sec:p-GT-A}, where we ``\emph{stabilize}" the  truncated   operator $\delta_\tau$ in \S \ref{sec:truncated-sen}. This section is purely ring-theoretic, and has nothing to do with Galois representations. An interesting (though somewhat confusing) fact is  that we show when studying crystalline representations, it is still necessary to make use of \emph{log-prismatic} rings as they are the \emph{Galois stable} rings (cf. Lemma \ref{lem Galois stable}). 
 In this section, we always allow $K$ to be ramified. Let $\star \in \{ \emptyset, \log \}$.
 
\subsection{Log-prismatic max rings and Galois actions}

\begin{notation}  \label{nota-uvw}
The $\star$-Breuil--Kisin prism $(\gs, (E), \star)$ admits two embeddings to the $\star$-Fontaine prism $(\ainf, (E), \star)$, induced by the $W(k)$-linear maps
\[ \gs=W(k)[[u_1]] \into \ainf; u_1 \mapsto [\pi^\flat] \] 
\[ \gs=W(k)[[u_2]] \into \ainf; u_2 \mapsto [\piflat][\zetaflat] \]
By the universal property of $\star$-prismatic co-product, they induce a map
\[ \iota:  (\gs^1_\star, (E), \star) \to (\ainf, (E), \star) \] 
It is injective  by \cite[Cor 2.4.5]{DL23} (recall to prove this, one can reduce modulo $E$). Henceforth, we always regard $\gs^1_\star$ as a subring of $\ainf$. 
By construction, $ \gs_\log^1$ is  the $(p , E)$-completion of the $\delta$-ring $\gs[\frac{ u _2 -u_1}{u_1E(u_1)}]_\delta$. For convenience, define the following elements in $\ainf$:
\begin{itemize}
\item $u:=\iota(u_1)$
\item $v:=[\zeta^\flat]-1$ 
\item $w:=\iota(\frac{ u _2-u_1}{u_1E(u_1)}) =\frac{[\zeta^\flat]-1}{E}$ (already  in Notation \ref{nota: ring for mod}).
\end{itemize}
\end{notation}

\begin{notation} 
On $(\ok)_{\pris, \star}$, one can define the sheaf 
$  \prism_\max:=\opris [ {\calI_\prism}/{p}]^{\wedge_p}$. For an evaluation $A [{I}/{p}]^{\wedge_p}$, equip it with the $p$-adic topology.
Denote some of its evaluations as follows:
\begin{itemize}
\item $ \gsmax:= \prism_\max(\gs)$ 
\item $ \bfa_\max:=\prism_\max(\ainf)$ 
\item $ \gs^1_{\star,\max}:= \prism_\max(\gs^1_\star) $.
\end{itemize}
(The first two rings are already used in \S \ref{sec: sen max classical}). 
In \cite{DL23}, the ring $ \gs^1_{\star,\max}$ is   first defined \emph{explicitly} in \cite[Definition 2.2.1]{DL23}; then they are related with the prismatic definition   in \cite[Remark 2.2.11]{DL23}. Thus, (as per  \cite[Definition 2.2.1]{DL23} and the log-version above  \cite[Proposition 2.3.1]{DL23}), we have
\[ \gs^1_{\log,\max} = \gsmax[\gamma_i(w)]_{i \geq 0}^{\wedge_p}\]
\[ \gs^1_{\max} = \gsmax[\gamma_i(uw)]_{i \geq 0}^{\wedge_p}\]
More explicitly, an element in $ \gs^1_{\log,\max}$ can be \emph{uniquely} written as a summation
\[ \sum_{i \geq 0} a_i\gamma_i(w) \]
where $a_i \in \gsmax= \gs[\frac{E}{p}]^{\wedge_p}$ and converges to zero.
It falls inside the subring $\gs^1_{\max}$ if and only if $u^i \mid a_i$ for each $i$, and $a_i/u^i$ still converges to zero.
\end{notation}

\begin{notation} \label{notags1dr}
Recall in \cite{GMWdR}, one can define a prismatic ``de Rham" period sheaf via:
\[ \prism_\dR^+: =\varprojlim_{m \geq 0} (\o_\prism[1/p])/\calI_\prism^m \]
Given a prism $(A, I)$, equip the evaluation $A_\dR^+ =\varprojlim_{m \geq 0} A[1/p]/I^m$ with the inverse limit topology where each $A[1/p]/I^m$ is equipped with the $p$-adic topology with $A/I^m$ being the unit ball. 
Denote some of its evaluations as
\[ \gs_\dR^+: = \prism_\dR^+((\gs, (E), \star)) =K[[E]] \]
\[ \gs_{\star, \dR}^{1, +}:  = \prism_\dR^+((\gs^1_\star, (E), \star)) \]
By explicit computations in \cite[Proposition 3.8]{GMWdR}, an element in $\gs_{\log, \dR}^{1, +}$ has a \emph{unique} expression
\[ \sum_{i \geq 0} a_i\gamma_i(w) \]
where $a_i \in \gs_\dR^+=K[[E]]$ and converges to zero (with respect to the topology discussed above).
  It falls inside the subring $\gs_{\dR}^{1, +}$ if and only if $u^i \mid a_i$ for each $i$, and $a_i/u^i$ still converges to zero.
\end{notation}
  
We summarize the embeddings of these rings.

\begin{lemma} \label{lem: ring closed embed}
We have a commutative diagram of injective maps of rings 
\[
\begin{tikzcd}
\gs \arrow[d, hook,closed] \arrow[r, hook,closed]      & \gs^1 \arrow[d, hook,closed] \arrow[r, hook]      & \gs^1_\log \arrow[d, hook,closed] \arrow[r, hook]          & \ainf \arrow[d, hook,closed] \\
\gs_\max \arrow[d, hook] \arrow[r, hook,closed] & \gs^1_\max \arrow[r, hook] \arrow[d, hook] & {\gs^1_{\log,\max}} \arrow[d, hook] \arrow[r, hook] & \amax \arrow[d, hook] \\
\gs^+_\dR \arrow[r, hook,closed]                & {\gs^{1,+}_\dR} \arrow[r, hook]            & {\gs^{1,+}_{\log,\dR}} \arrow[r, hook]              & \bdrplus             
\end{tikzcd}
\]
Only those arrows crossed by a slash are closed  embeddings. In addition, all vertical composite maps from the top row to the bottom row (such as $\gs^1 \into \gs^{1,+}_\dR$) are closed  embeddings. 
\end{lemma}  

\begin{proof}
Consider the vertical maps first. Let us only consider the maps 
\[ \gs^1 \to \gs^1_\max \to \gs^{1,+}_\dR\]
since   other columns follow from similar (or easier) arguments.  
 The map $\gs^1  \into \gs^{1,+}_{\max}$ is a closed  embedding  by  \cite[Lemma 2.2.10]{DL23}. 
 The map $\gs^{1,+}_{\max} \to \gs^{1,+}_{ \dR}$ is injective by   explicit descriptions of these rings; it is not a closed embedding by considering the sequence $(E/p)^i$ which only converges in the target.
 Finally, the composite map  $\gs^1 \to \gs^{1,+}_{ \dR}$ is a closed  embedding because it is the inverse limit of $\gs^1/E^i \into \gs^1[1/p]/E^i$ for all $i$ which are closed  embeddings.

 Now, consider the horizontal arrows. The left most horizontal arrows ($\gs \to \gs^1$, $\gs_\max \to \gs^1_\max$, and $\gs^+
_\dR \to \gs^{1,+}_\dR$ are obviously closed embeddings. All other desired properties follow from the following diagram
\[ \begin{tikzcd}
\gs^1/E \arrow[d, "="] \arrow[r, hook]           & \gs^1_\log/E \arrow[r, hook] \arrow[d, "="]               & \ainf/E \arrow[d, "="]      \\
\gs^1_\max/(E/p) \arrow[r, hook] \arrow[d, hook] & {\gs^1_{\log,\max}/(E/p)} \arrow[r, hook] \arrow[d, hook] & \amax/(E/p) \arrow[d, hook] \\
{\gs^{1,+}_\dR/E} \arrow[r, hook]                & {\gs^{1,+}_{\log,\dR}/E} \arrow[r, hook]                  & \bdrplus/E                 
\end{tikzcd} \]
 where all horizontal arrows are not closed embeddings. For example,
 \[ \gs^1_\log/E=\ok[\gamma_i(w)]_{i \geq 0}^{\wedge_p} \into \ainf/E=\o_C\]
 is not closed because $w \pmod{E} \in \o_C$ has positive valuation.
\end{proof}

\begin{remark} \label{rem mod p ring}
Much of the proof of above Lemma \ref{lem: ring closed embed} follows by reduction modulo $E$ (or $E/p$), because these reductions have explicit descriptions. However, the mod $p$ properties of these rings are far from well-understood.  For example, we do not know (cf. discussion in \cite[Remark 3.2.4]{DL23}) if  $\gs^1_\star/p \to \ainf/p$ is injective.  
\end{remark}

\begin{lemma}\label{lem Galois stable}
 Regard all rings in Lemma \ref{lem: ring closed embed} as subrings of $\bdrplus$.
\begin{enumerate}
\item The rings $\gs^1_\log$, $\gs^1_{\log, \max}$ and $\gs^{1,+}_{\log, \dR}$ are $\gk$-stable; in addition, the action factors through $\hat{G}$. 

\item The rings $\gs^1$, $\gs^1_{\max}$ and $\gs^{1,+}_{\dR}$ are \emph{not} $\gk$-stable.  
\end{enumerate}
\end{lemma}
\begin{proof}
The log-prismatic case in Item (1) is actually relatively easy. 
The fact that $\gs^1_\log$ is $\gk$-stable is proved in \cite[Lemma 3.3.1]{DL23}  by a concrete  computation. Here is a more conceptual explanation. Using the universal property of the log-prismatic coproduct, it suffices to prove that \[ \tau(W(k)[[u_2]])= \tau(W(k)[[[\piflat\zeta^\flat]]]) \subset \gs^1_\log\]
 (since obviously   $ \tau(W(k)[[\piflat]]) \subset \gs^1_\log$); but this is easy since $[\zeta^\flat] \in \gs^1_\log$ (because $w\in \gs^1_\log$). The $\gk$-stability for $\gs^1_{\log, \max}$ and $\gs^1_{\log, \dR}$ follow as $\frac{\tau(E)}{E}$ is a unit in $\gs^1$ by \cite[Lemma 2.24]{BS22}.

The prismatic case in Item (2) is somewhat confusing, as one would naturally expect these rings should also be $\gk$-stable (by analogy with the log case). However, the non-stability for $\gs^1$  was already noticed in \cite{DL23} (without proof); e.g., \cite[Definition3.3.2]{DL23} was formulated only using the log-prismatic ring. 
To prove the non-stability, note $u_2=[\piflat][\zeta^\flat]$ is in all of these rings; we claim $\tau(u_2)=[\piflat][\zeta^\flat]^2$ is not even in the biggest ring, that is: 
\[ u(v+1)^2 \notin \gs^{1,+}_\dR \]
Indeed, inside the log ring $\gs^1_{\log,\dR}$, 
\[ u(v+1)^2= u(v^2+2v+1) = 
u\gamma_0(w)+2uE \gamma_1(w)+ 2uE^2\cdot \gamma_2(w) \]
It is not inside $\gs^{1,+}_\dR$ because $u^2$ does not divide $2uE^2$ (cf. discussion in Notation \ref{notags1dr}). 
\end{proof}

\begin{remark} \label{rem:logpris-necessary}
Lemma \ref{lem Galois stable} suggests that in order to study $\gk$-action (particular $\tau$-action) on modules, it is necessary to make use of $\gs^1_\log$ and its variants, even if when we are studying \emph{crystalline} representations.  This is somewhat strange; e.g., it does not seem easy to ``see/realize" these phenomenon on the prismatic site. (Of course,  we shall still need to   use $\gs^1$ and its variants.)
\end{remark}

We discuss some elements in these rings  which will be used later.

\begin{lemma} \label{lem-element-t}
Recall the elements $\lambda, t, \fkt$ are defined in Notation \ref{nota: ring for mod}. We have:
\begin{enumerate} 
\item   $\lambda \in \gs_\max$.
\item  
$ t\in \gs^1_{\log, \max}$, but  $t\notin \gs^{1,+}_{\dR}$.

\item  $\frac{\fkt}{w}$ is a unit in $\gs^1_{\log, \max} \cap \ainf$.
(We do not know if  $\frac{\fkt}{w}  \in \gs^1_\log$).

\item  $\fkt \in \gs^1_\log$, but $ \fkt \notin   \gs^{1,+}_{\dR}$. 
\end{enumerate}
\end{lemma}
\begin{proof}
Items (1)(2). 
The case for $\lambda$ is clear.
For $t$, its unique expression in $\gs^1_{\log, \max}$ is
\[ t=\log([\zetaflat]) =\sum_{i \geq 1}(-1)^{i+1} \frac{([\zetaflat]-1)^i}{i}
 =\sum_{i \geq 1} (-1)^{i+1} (i-1)!E^i \cdot \gamma_i(w)\]
it does not fall inside $\gs^{1,+}_{\dR}$ because $u^i$ does not divide $E^i$.

Items (3).   As discussed in Notation \ref{nota: ring for mod}, $\fkt/w$ is a unit in $\ainf$. 
Recall the infinite product
$$ \mathfrak{t} = \frac{t}{p\lambda}=\frac{p\phi^{-1}([\zetaflat]-1) \prod_{n \geq 0} \phi^n(\xi/p)}{p\prod_{n \geq 0} (\varphi^n(\frac{E(u)}{E(0)}))}.$$
Thus $\frac{\fkt}{w}$ is equal to the limit of \[ x_s= \frac{E(0)}{p} \frac{\prod_{n=1}^s \phi^n(\xi/p)}{\prod_{n=1}^s (\varphi^n(\frac{E(u)}{E(0)}))} \]
 Since 
\[ \frac{E}{p}, \frac{[\zetaflat]-1}{p}  \in \gs^1_{\log, \max}\]
it is easy to see the terms 
$ \phi^n(\xi/p)$ and $\varphi^n(\frac{E(u)}{E(0)})$ are all units in $\gs^1_{\log, \max}$ and they converge to $1$. Thus $x_s$ converges to a unit inside  $\gs^1_{\log, \max}$.

Item (4).  $\fkt \notin \gs^{1,+}_{\dR}$ because $t \notin \gs^{1,+}_{\dR}$.  Note $\fkt$ is the limit of 
\[ y_s =  \frac{p\phi^{-1}([\zetaflat]-1) \prod_{n = 0}^s \phi^n(\xi/p)}{p\prod_{n= 0}^s (\varphi^n(\frac{E(u)}{E(0)}))}\]
which is already known to converge inside $ \gs^1_{\log, \max}$. Since $\gs^1_\log\into \gs^1_{\log, \max}$ is a closed embedding (Lemma \ref{lem: ring closed embed}), it remains  to prove that  
$y_s \in \gs^1_\log$ for each $s$.
The case $s=0$ is obvious. We prove the general case by induction.  Suppose $y_s \in \gs^1_\log$. Then
\[ \phi(y_s) \in \gs^1_\log \]
It is easy to check $\phi(y_s)$  is in the kernel of $\theta$-map on $\ainf$; thus
\[ \phi(y_s) \in \gs^1_\log \cap E\ainf=E \gs^1_\log \]
where the last equality follows from  the injection $\gs^1_\log/E \into \ainf/E$. 
 Thus
\[ y_{s+1} =\frac{E(0)}{p}\frac{\phi(y_s)}{E} \in  \gs^1_\log \]  
\end{proof}

We study the Galois action on $\gs^1_{\log, \max}$ (recall the prismatic ring  $\gs^1_{\max}$ is not Galois stable by Lemma \ref{lem Galois stable}). We shall repeatedly use the formal relation
\begin{equation} \label{eqtauab}
(\tau -1)(ab) =(\tau-1)(a)\tau(b) +a(\tau-1)(b)
\end{equation}

\begin{lemma} \label{lem tau stmax}
\begin{enumerate}
\item $(\tau-1)(E/p) \in (E/p) \cdot  uw \cdot \gs^1_{\log, \max}$
\item $(\tau-1)(w) \subset  w\cdot uw \cdot \gs^1_{\log, \max}$.
\item   $(\tau-1)(\gs^1_{\log, \max}) \subset uw \cdot \gs^1_{\log, \max}$.
\end{enumerate}
\end{lemma}
\begin{proof}  
Item (1). Since $(\tau-1) (u)  =uv= uw E$, we have
\[ (\tau-1)(E) =  u w E y \text{ with some $y \in W(k) [u , v] \in \gs^1_\log$ } \]

Item (2). Note  $\tau (E) /E = 1+ u w y$. Thus
\[ (\tau -1) w = (\tau -1) (\frac{ [\zeta^\flat]-1}{E})= ([\zeta^\flat]-1) (\tau -1) (\frac 1 E)= v \frac{E -\tau (E)}{ E \tau (E)}= u w^2y  (\tau(E)/E)^{-1}= u w^2 y (1 + u wy)^{-1}  \] 

For Item (3), consider   an element in $\gs^1_{\log, \max}$ 
\[x= \sum_{i \geq 0} a_i \gamma_i (w) \text{ with $a_i \in \gs[E/p ]^{\wedge_p}$ converging to zero }\]
\eqref{eqtauab} implies  for each $i$,
\[ \frac{\tau-1}{uw} (E/p)^i \in  (E/p)^i \cdot \gs^1_{\log, \max}\]
This implies
\[ \frac{\tau-1}{uw}(a_i) \in \gs^1_{\log, \max}, \text{ and converges to zero }\]
\eqref{eqtauab} also implies 
\[ \frac{\tau-1}{uw}(\gamma_i(w)) \subset \gamma_i(w) \cdot \gs^1_{\log, \max}\]
Thus \[ \frac{\tau-1}{uw}(x) \in \gs^1_{\log, \max}\]
\end{proof}

\subsection{A flat mod $p$ subring}
\label{subsec-flat-subring}
Lemma \ref{lem tau stmax} suggests that $\gs^1_{\log, \max}$ might be a useful ring to stabilize the truncated  operator $\delta_\tau$ in \S \ref{sec:truncated-sen}. Unfortunately, similar to the situation in \S \ref{sec: sen max classical}, $\gs^1_{\log, \max}$  is not flat over $\gs$. For our purpose, we need to introduce certain flat subring (at least in the mod $p$ case as $\delta_\tau$ recovers the mod $p$ Sen operator). It turns out,   a naive idea---intersecting with the $\gs$-flat ring $\ainf$--- works out, after some modification.

\begin{notation} \label{cons-flat-subring}
  Define 
$$A : = \gs^1_{\log, \max}[1/p] \cap \ainf,$$
 with intersection taken inside $ \amax[1/p]$. 
 \begin{enumerate}
 \item Warning:  we do not know if  $A$ is $p$-adic complete, cf. Remark \ref{rem implicit A} for more comments. 
  \item  Define 
  $$A_1=A/pA.$$
Since $A \cap p\ainf =p(\gs^1_{\log, \max}[1/p] \cap \ainf)=pA$, the induced map  $ A_1  \to \ainf/p= \ocflat$ is injective. Henceforth, we regard $A_1$ as a subring of $\ocflat$. 
\item  As $A_1$ has no $u$-torsion, the map  $\gs/p\gs \to A_1$ is \emph{flat}. 

\item  Let $ \fkm_{A_1} : = A_1 \cap \fkm_\ocflat$ (with intersection taken in $\ocflat$). (We will prove $A_1/ \fkm_{A_1} \simeq k$ in Lemma \ref{lem reductionA1}).
 \end{enumerate}  
\end{notation}

\begin{remark} \label{rem implicit A} 
Unfortunately, the ring $A$ is rather implicit. Here are some further comments. 
\begin{enumerate}
\item Related to discussions in Remark \ref{rem mod p ring}, we do not know if the inclusion
\[\gs^1_{\log, \max} \cap \ainf \subset \gs^1_{\log, \max}[1/p] \cap \ainf\]
is equal. Indeed, we do not know if  $A$ is $p$-adic complete (either using induced topologies or its own $p$-adic topology). One can take $p$-adic completion of $A$; but this remains an implicit ring. 
Fortunately, for our application, we shall only need the mod $p$ ring $A_1$ (which is flat over $k[[u]]$).

\item One can also consider $\gs^1_\log$ which is flat over $\gs$,  but we do not know if the analogues of Lemma \ref{lem tau1A} and Proposition \ref{lem-tau-1} hold for this ring: its structure is too complicated for many concrete computations.
\end{enumerate}
\end{remark}

In the remainder of this subsection, we record some properties of $ A$ which will be used  later.

\begin{lemma}\label{lem:difffktw}
Consider the element $\frac{p}{E(0)}\cdot\frac{\fkt}{w} \in A^\times$ (by Lemma \ref{lem-element-t}). Modulo $p$, we have
\[ \frac{p}{E(0)}\cdot \frac{\fkt}{w} \pmod{p} \in 1+uA_1 \]
\end{lemma}
\begin{proof}
It is easy to see the kernel of $\theta_\fon$-map restricted to  $\gs^1_{\log, \max}[1/p]$ is the principal ideal generated by $E$; thus \[ \Ker (\theta_\fon |_A)=EA.\]
 As computed in proof of  Corollary \ref{cor-crys-truncate-sen}, 
\[\theta_\fon(\frac{p}{E(0)}\cdot\frac{\fkt}{w}) \in 1+\pi\ok\]
Thus \[ \frac{p}{E(0)}\cdot \frac{\fkt}{w}  \in 1+u\gs+EA. \]
We can modulo $p$ to conclude.
\end{proof}

\begin{lemma}\label{lem-A-max} 
We have the following identities, where all intersections are taken inside $\amax[1/p]$. 
\begin{enumerate}
\item $ u \gs^1_{\star, \max} \cap \ainf =u(\gs^1_{\star, \max} \cap \ainf)$.
\item $ u \gs^1_{\star, \max}[1/p] \cap \ainf =u(\gs^1_{\star, \max}[1/p]  \cap \ainf)$.
\item $w \gs^1_{\star, \max} \cap \ainf =w(\gs^1_{\star, \max} \cap \ainf)$.
\item $w \gs^1_{\star, \max}[1/p] \cap \ainf =w(\gs^1_{\star, \max}[1/p] \cap \ainf)$. 
\item $\fkt \gs^1_{\star, \max} \cap \ainf =\fkt(\gs^1_{\star, \max} \cap \ainf)$.
\item $\fkt \gs^1_{\star, \max}[1/p] \cap \ainf =\fkt(\gs^1_{\star, \max}[1/p] \cap \ainf)$. 
\end{enumerate}
\end{lemma}
\begin{proof}
Items (1) and (2). Suppose $x=uy$ with $x\in \ainf$ and $y$ is an element of  $\gs^1_{\star, \max}$ or $\gs^1_{\star, \max}[1/p]$. Then $\phi(x)=u^p\phi(y)$. Note $\phi(y) \in \bcrisplus$.  Using  \cite[Lem. 3.2.2.]{Liu13}, we have $\phi(y) \in \ainf$  and hence $y \in \ainf$.

Items (3) and (4). Suppose $x=wy$ with $x\in \ainf$ and $y$ is an element of  $\gs^1_{\star, \max}$ or $\gs^1_{\star, \max}[1/p]$.  
Note $w$ is not an element in the prismatic ring  $\gs^1_{\max}$ (i.e. the $\star=\emptyset$ case), however this will not cause problem in the following argument. Note  for each $m \geq 0$,
$$\phi^m(Ex)=\phi^m(vy) \in \fil^1 \bdrplus.$$
Thus $Ex \in I^{[1]}\ainf$ (cf. Notation \ref{nota: ring for mod}), which is principally generated by $v$; thus $y \in \ainf$. 

Items (5) and (6). When $\star=\log$, then these follow from Items (3) and (4) as $w/\fkt$ is a unit in $\gs^1_{\log, \max}$ by Lemma \ref{lem-element-t}. But indeed, the general case also follow from similar argument as Items (3) and (4). Let $x=\fkt y$ with $x \in \ainf$  and $y$ is an element of  $\gs^1_{\star, \max}$ or $\gs^1_{\star, \max}[1/p]$. Then one can similarly check $Ex \in I^{[1]}\ainf=E\fkt \ainf$, thus $y \in \fkt \ainf$.
\end{proof}

\begin{lemma} \label{lem tau1A}
$(\tau-1)A\subset uwA =u\fkt A$.
\end{lemma}
\begin{proof}
By Lemma \ref{lem tau stmax}, $(\tau-1)A$ is contained in
$uw \gs^1_{\log, \max}[1/p] \cap \ainf$, which is equal to $uwA$ by Lemma \ref{lem-A-max}. Finally $uwA =u\fkt A$ because $w/\fkt \in A^\times$  by Lemma \ref{lem-element-t}.
\end{proof}

\begin{construction} \label{cons nu}
Recall the projection map (first appeared in \cite[5.3.3]{Fon94}):
\[ \nu: \ainf \onto  W(\overline{k}) \]
and use $I_+=W(\fkm_\ocflat) $ to denote its kernel.
We have  $[\zeta^\flat]-1 \in I_+$ (this is well-known; cf.   \cite[Lemma 4.1]{GabberLi} for a recent exposition). As $\nu(E)=p$, we also have $w=  \frac{[\zeta^\flat]-1}{E} \in I_+$.
This map extends to
\[ \nu: \amax[1/p]  \onto  W(\overline{k})[1/p] \]
%with kernel given by $p$-adic closure of  $I_+ [\frac{u ^e}{p}]$. 
    
\end{construction}

\begin{lemma} \label{lem reductionA1}
Recall $\fkm_{A_1} \subset A_1$ is defined in  Notation \ref{cons-flat-subring}.  We have
\[  A_1/ \fkm_{A_1} \simeq k \]
\end{lemma}
\begin{proof} 
Since $\nu(w)=0$, it is easy to see $\nu(\gs^1_{\log, \max}[1/p])=W(k)[1/p]$; then it is easy to see $\nu(A)=W(k)$.
Consider the commutative diagram
\[
\begin{tikzcd}
\gs/u \arrow[d, two heads] \arrow[r, "\simeq"] & A/(I_+\cap A)=W(k) \arrow[r, hook] \arrow[d, dashed] & \ainf/I_+ \arrow[d, two heads] \\
\gs_1/u \arrow[r, hook]                        & A_1/\fkm_{A_1} \arrow[r, hook]                       & \ocflat/\fkm_\ocflat          
\end{tikzcd}
\]
Since all horizontal arrows are injective,  the dotted arrow must exist and is surjective.
 This forces $A_1/\fkm_{A_1}\simeq k$. 
\end{proof}

%%\newpage  
 \section{Stabilized truncated   operator}\label{sec:p-GT-A}

 In \S \ref{subsec:Stabilized truncated operator}, we use the ring $A$ in \S \ref{subsec-flat-subring} to \emph{stabilize} the truncated   operator $\delta_\tau$ in \S \ref{sec:truncated-sen}; this works for any $K$. 
 In \S \ref{subsec:rel with fil}, we   study its interaction with Frobenius operator and filtrations (when $K$ is unramified), and derive results concerning generalized eigenvalues of the mod $p$ Sen operator (more ``refined" results will be further discussed in \S \ref{sec:pDS vs pGT}).

%Note $\wh{\Theta}$ does not kill $A_1/u$, hence need to mod $\fkm_{A_1}$ to get Sen op. 
 
\subsection{Stabilized truncated operator} \label{subsec:Stabilized truncated operator}
  
  Recall  the ring in Notation  \ref{cons-flat-subring}:
$$A   = \gs^1_{\log, \max}[1/p] \cap \ainf.$$
For a Breuil--Kisin module $\gm$, denote
\[\gm_A: =\gm\otimes_\gs A\]

\begin{prop}\label{lem-tau-1}
(Allow $K$ to be ramified). Let $\gm$ be the Breuil--Kisin module attached to an integral crystalline representation. Then
\[ (\tau-1)(\gm_A) \subset \gm\otimes_\gs u\fkt A. \] 
As a consequence, we can define an operator
\[ \delta_\tau:= \frac{\tau-1}{u\frac{p}{E(0)}\fkt}: \gm_A \to \gm_A. \]
Modulo $p$ and then modulo $\fkm_{A_1}$, then it is   the same as the mod $\pi$ reduction of the amplified Sen operator:
\[ \Theta: \gmht \to \gmht.\]
 \end{prop} 
\begin{proof}  
The fact that 
\[ \delta_\tau \pmod{p, \fkm_{A_1}} \equiv  \Theta \pmod{\pi}\]
follows from Lemma \ref{lem reductionA1} and Corollary  \ref{cor-crys-truncate-sen}. It thus suffices to construct $\delta_\tau$. Since $w/\fkt \in A^\times$ by Lemma \ref{lem-element-t}, we shall use $w$ in the following for computational convenience.
 By Lemma \ref{lem tau1A}, we already have
\[   (\tau-1)A\subset uwA.\]
It remains to prove
\[ (\tau-1) (\gm) \subset  \gm\otimes_\gs uwA. \] 
 As we already know 
\[ (\tau-1) (\gm) \subset  \gm\otimes_\gs uw\ainf; \]
it suffices to prove
\[ (\tau-1) (\gm) \subset  \gm\otimes_\gs uw\cdot \gs^1_{\log, \max}[1/p]  \] 
 As $\gm$ is associated to a prismatic $F$-crystal, we have
\[ (\tau-1)(\gm) \subset \gm \otimes_\gs \gs^1 \subset \gm \otimes_\gs \gs^1_\max\]
Fix any basis $e_1, \cdots, e_d$ of $\gm$, then for any $m \in \gm$, we have
\[ (\tau-1)(m) = \sum_{i} x_i e_i \in\gm \otimes_\gs \gs^1_\max, \quad \text{where } x_i=\sum_{j \geq 0} a_{ij} \gamma_j (uw) \in \gs^1_\max.\]
As mentioned above, we have  $x_i  \in uw \ainf$. Use the map 
 $ \nu : \amax [\frac 1 p]  \to W(\bar k ) [\frac 1 p]$ recalled in 
Construction \ref{cons nu}.
We must have  $a_{i0} = \nu( x_i) = 0.$ 
 Lemma \ref{lem aozero}   in the following implies that $x_i$ is  divisible by $uw$.

%To prove the claim that $a_{i0}=0$, it suffices to prove this regarding $x_i$ as an element in $\gs^{1,+}_\dR$; then this follows from the explicit stratification formula of $\prism_\dR^+$-crystals in \cite[Proposition 7.9]{GMWdR}, which says that there is certain (regular connection) operator $\nabla_M$ on $M=\gm\otimes_\gs \gs^+_\dR$ such that for $x \in M$,  \[ (\tau-1)(x)= \sum_{n \geq 1} \left( (\prod_{i=0}^{n-1}(\nabla_M-i))(x)\cdot \gamma_n(uw)\cdot \gamma_n(\frac{E}{u-\pi}) \right)\] That is: there is no ``constant term" on the RHS. Note the extra factor $\gamma_n(\frac{E}{u-\pi})$ appears above because \emph{loc. cit.} was using different variables; but this is a unit in $\gs^+_\dR=K[[E]]$ hence is harmless. 
\end{proof}

\begin{lemma} \label{lem aozero}
Consider an element in $\gs^1_{\max}$  
\[x= \sum_{i \geq 0} a_i \gamma_i (uw) \text{ with $a_i \in \gs[ E/p ]^{\wedge_p}$ converging to zero. }\] Suppose $a_0=0$, then
\[ x\in uw \cdot \gs^1_{\log, \max}[1/p] \]
\end{lemma}
\begin{proof}
Note
\[\gamma_i(uw) =uw \cdot \frac{u^{i-1}}{i} \cdot \gamma_{i-1}(w) \]
then note $\frac{u^{i-1}}{i}$ is in $\gs_\max[1/p]$ and converges to zero.
 (Caution: it is necessary to assume $x$ is in the \emph{prismatic} ring $\gs^1_\max$ here, although the division happens inside the \emph{log-prismatic} ring $\gs^1_{\log, \max}[1/p]$.)
\end{proof}

\subsection{Relation with filtrations: $p$-GT and $p$-DS} \label{subsec:rel with fil}
 
 We now study the relation between the  truncated operator and filtrations. We shall make use of filtrations related to $\phi(\bargm)$.

\begin{construction} \label{ConstructionfilphiA}
 We recall some notations following  Definition \ref{defn fil phiM} with respect to the triples $(\gs_1, u, \phi)$ and $(A_1, u, \phi)$. For notational simplicity, let $\bargm_A:=\bargm\otimes_{\gs_1}A_1$.
\begin{enumerate}
\item Recall   $\fil^\bullet \phi(\bargm)=\phi(\bargm) \cap u^\bullet \bargm$, and let $N_{n,i}$ be the image of the composite map
\[u^i \fil^{n-i} \phi(\bargm) /u^i \fil^{n+1-i} \phi(\bargm) \into \fil^n \bargmast/\fil^{n+1} \bargmast  \xrightarrow{u^{-n}, \simeq} \fil_n \bargm_\HT. \] 
One can similarly define $\fil^\bullet \phi(\bargm_A)$, and define $N_{A, n, i}$ as the image of
\[u^i \fil^{n-i} \phi(\bargm_A) /u^i \fil^{n+1-i} \phi(\bargm_A) \xhookrightarrow{u^{-n}} \fil_n \bargm_{A,\HT}. \]

\item Recall $\gs_1 \to A_1$ is flat (cf. Notation \ref{cons-flat-subring});  by Lemma \ref{lem: flat base change}, we have
\[ \fil_n \bargm_\HT \otimes_{k} A_1/u A_1 \simeq \fil_n \bargm_{A,\HT}. \]
Thus there is a projection map by reduction modulo $\fkm_{A_1}$:
\[\fil_n \bargm_{A,\HT} \onto  \fil_n \bargm_\HT. \]
Denote the image of  $N_{A, n, i}$ under the projection as  $\wt{N}_{n,i}$.

\item As we warned in Remark \ref{rem noflat}, the filtrations on $\phi(\bargm)$ does not satisfy flat base change; thus the relation between $N_{n,i}$ and $N_{A, n, i}$ is not clear. However, since  the induced map $\gs_1/u \to A_1/u$ is injective, we have  $u^n \bargm \cap u^{n+1} \bargm_A=u^{n+1}\bargm$; this implies $\fil^\bullet \varphi (\bargm)\to\fil^\bullet (\bargm_A)$ is strict.  Thus we have an injective map 
\[ \Fil^n \varphi (\bargm) / \Fil^{n+1} \varphi (\bargm) \into \Fil^n \varphi (\bargm_A) / \Fil^{n+1} \varphi (\bargm_A) \]
This   implies that (inside $\bargmht$):
\[ N_{n,0}\subset \wt{N}_{n,0}  \]
\end{enumerate}
\end{construction}

 The following can be regarded as some mod $p$ analogue of Lemma \ref{lem: Kis nyg fil}, cf. Remark \ref{rem:comparepGT} for more comments.
 Recall $\fil^\bullet \bargm_A$ is defined as the $\phi^{-1}(\fil^\bullet \phi(\bargm_A))$ (following Definition \ref{defn fil phiM}).

 \begin{prop}[$p$-Griffiths transversality on $A_1$-level] \label{prop-pGT-A-level} 
Suppose $K$ is unramified.  Following Proposition \ref{lem-tau-1}, consider the mod $p$ map
\[ \delta_\tau= \frac{\tau-1}{u\frac{p}{E(0)}\fkt}: \bargm_A \to \bargm_A.\]
Then we have 
\begin{enumerate}
\item  $\delta_\tau(u^n \bargm_A) \subset  u^n \bargm_A$;
\item \label{item2pgta} $\delta_\tau(\Fil^n \varphi (\bargm_A)) \subset u^p   \fil^{n-p}\phi(\bargm_A)$;
\item  \label{item3pgta}   \emph{($p$-Griffiths transversality)}: $\delta_\tau(\Fil^n  \bargm_A) \subset  \fil^{n-p} \bargm_A$. 
\end{enumerate}
\end{prop}
\begin{proof}
For Item (1), it suffices to check $\delta_\tau(u^n) \subset u^nA_1$, which is easy.

For Item (2), consider  the formal relation
\[\delta_\tau \phi =  \frac{\phi(u\frac{p}{E(0)}\fkt)}{u\frac{p}{E(0)}\fkt} \phi \delta_\tau =\frac{p}{\phi(E(0))}u^p \phi \delta_\tau \]
Here  in above computation we use  $\phi(\fkt)=\frac{pE(u)}{E(0)}\fkt$   and use $K$  is unramified.
    Thus, using Item (1), we have
   \[\delta_\tau(\Fil^n \varphi (\bargm_A)) \subset  u^p  \phi(\bargm_A) \cap u^n \bargm_A = u^p  \fil^{n-p}\phi(\bargm_A)\]
Note in above argument, we used the fact $\frac{p}{\phi(E(0))}u^p \varphi (\bargm_A))= u^p \varphi (\bargm_A))$ because $\frac{p}{\phi(E(0))}$ is a unit in $\phi(A_1)$ (not just in $A_1$!); cf. Remark \ref{rem:preferdeltatau} for more comments.
   
For Item (3), it is easy to check:
\[\phi(\delta_\tau(\Fil^n  \bargm_A)) \subset u^{-p} \delta_\tau (u^n\bargm_A) \subset u^{n-p}\bargm_A.\]
\end{proof}

\begin{remark}\label{rem:comparepGT}
We compare Proposition \ref{prop-pGT-A-level}  with Lemma \ref{lem: Kis nyg fil} (expanding discussions following Proposition  \ref{prop:introtauA} in the introduction).
\begin{enumerate}
\item \label{itemcompare2}
The $A$-version of Lemma \ref{lem: Kis nyg fil}\eqref{kis1} (hence Lemma \ref{lem: Kis nyg fil}\eqref{kis2}) does not hold, that is, we   do not have $\delta_\tau(\bargm^\ast_A) \subset u\bargm^\ast_A$. Indeed, we do not have  $\delta_\tau(A_1) \subset uA_1$. To see this, suppose $E=u-p$ and $p>2$, consider $w \in A_1$. Using notations in the computation of Lem \ref{lem tau stmax}, note $(\tau-1)(E)=uwE$ and hence $y=1$. Thus the valuation of $\delta_\tau(w)$ (modulo $p$) in $\ocflat$ is equal to valuation of $w$, which is strictly smaller than valuation of $u$.

%It is easy to see $\delta_\tau(\bargm^\ast_A) \subset \bargm^\ast_A$. However, it is not clear if  $\delta_\tau(\bargm^\ast_A) \subset u\bargm^\ast_A$: this is because we do not know if  $\delta_\tau(A_1) \subset uA_1$.   As a result, we do not know if the $A$-version of Lemma \ref{lem: Kis nyg fil}\eqref{kis2} should hold, that is, we   do not know if \[ \delta_\tau(\fil^n \bargm_A^\ast) \subset u\fil^{n-1} \bargm_A^\ast\] But these are not needed.

\item \label{itemcompare3}
  Proposition \ref{prop-pGT-A-level}\eqref{item2pgta} says that $\delta_\tau$ satisfies a ``twisted $p$-Griffiths transversality" on $\fil^\bullet \phi(\bargm_A)$ (note the extra twisting by $u^p$, which is a uniformizer of $\phi(\gs_1)$). This will induce a \emph{$p$-degree shrinking}, cf. Corollary  \ref{cor tilde gen eigen} and Theorem \ref{prop: p griffiths}.

\item   Proposition \ref{prop-pGT-A-level}\eqref{item3pgta} says that $\delta_\tau$ satisfies a   $p$-Griffiths transversality on $\fil^\bullet \bargm_A$. This shall---without any twisting in the process---induce another $p$-Griffiths transversality, cf. Theorem \ref{thm:BL-p-GT}.
\end{enumerate}
\end{remark}

\begin{remark} \label{rem:preferdeltatau}
Continuing Remark \ref{rem:renorm_fkt}, recall we defined another operator $\delta':=\frac{\tau-1}{uw}$ in the proof of Corollary \ref{cor-crys-truncate-sen}. There is the formal relation on  $\bargm_A$:
 \[\delta' \phi = \frac{\phi(uw)}{uw} \phi \delta'
= (\zeta^\flat -1)^{p-1} \phi \delta' \]
It is easy to see  $(\zeta^\flat -1)^{p-1} = u^p \beta$ with $\beta \in  (A_1)^\times$ (because $w/\fkt \in A^\times$).
However, we do not know if $\beta \in  (\phi(A_1))^\times$; this means that if we use $\delta'$ instead of $\delta_\tau$, we shall have
\[   \delta'(\Fil^n \varphi (\bargm_A)) \subset u^p \beta  \fil^{n-p}\phi(\bargm_A)  \]
Actually this extra $\beta$ does not really hinder our applications, except we need to carry it along the way. We decide to use $\delta_\tau$ since we prefer the ``cleaner" formula it satisfies.
\end{remark}

\begin{construction} \label{ConstructiontwistA}
  We quickly recall the mod $p$ Sen operator and its twistings.  As $K$ is unramified and $\bargm$ comes from reduction of an integral  crystalline representation. 
Following Notation \ref{nota: crys final}, in this case, we have  $\theta_\kinfty= \Theta$; thus we simply use $\theta$ to denote the mod $p$ operator:
\[ \theta :    \bargm_\HT \to   \bargm_\HT. \] 
The diagram \eqref{diag shift sen bargmht} then simplifies as
\begin{equation} \label{new unram diag shift sen bargmht}
\begin{tikzcd}
\fil^n \bargmast/\fil^{n+1} \bargmast \arrow[d, "{\times u^{-n}, \simeq}"] \arrow[rr, "\theta"] &  & u \fil^{n-1} \bargmast/ u \fil^{n} \bargmast  \arrow[d, "{\times u^{-n}, \simeq}"] \\
\fil_n  \bargm_\HT \arrow[rr, "\theta-n"]                                            &  & \fil_{n-1}  \bargm_\HT                                                  \end{tikzcd}
\end{equation}
\end{construction}

 As a result of Remark \ref{rem:comparepGT}\eqref{itemcompare2}, we do not have the $A$-version of Diagram \eqref{new unram diag shift sen bargmht}. However, as suggested by Remark \ref{rem:comparepGT}\eqref{itemcompare3}, there is another version inducing a $p$-degree shrinking. 

%From Proposition \ref{prop-pGT-A-level}, one sees that $\delta_\tau$ induces the operators (that we still denote by $\delta_\tau$): \[ \delta_\tau:  \bargm_{A,\HT} \to \bargm_{A,\HT} \] \[ \delta_\tau:  \Fil^n \varphi (\bargm_A) / \Fil^{n+1} \varphi (\bargm_A) \to  u^p\Fil^{n-p} \varphi (\bargm_A) / u^p\Fil^{n+1-p} \varphi (\bargm_A) \]

\begin{prop} \label{prop:pDS-A}
We have a commutative diagram, where the vertical arrows are well-defined by  Proposition \ref{prop-pGT-A-level},
\begin{equation} \label{eqshiftAver}
\begin{tikzcd}
\Fil^n \varphi (\bargm_A) / \Fil^{n+1} \varphi (\bargm_A) \arrow[d, "\delta_\tau"] \arrow[rr, "u^{-n}"] &  & {\bargm_{A,\HT} } \arrow[d, "\delta_\tau -n"] \\
u^p  \Fil^{n-p} \varphi (\bargm_A) / u^p  \Fil^{n+1-p} \varphi (\bargm_A)  \arrow[rr, "u^{-n}"]             &  & {\bargm_{A,\HT} }                            
\end{tikzcd}
\end{equation}
In addition, we have
\[ (\delta_\tau -n)(N_{n,0,A}) \subset N_{n,p,A}= N_{n-p,0,A}\]
\end{prop}
\begin{proof}
Once commutativity is checked, the relation $(\delta_\tau -n)(N_{n,0,A}) \subset N_{n,p,A}$ follows by considering images of horizontal maps in the diagram. 
To  check commutativity, we need to show for $x \in \fil^n \phi(\bargm_A)$, we have
\[ (\delta_\tau -n)(u^{-n}x) \equiv u^{-n}\delta_\tau(x)\pmod{u\bargm_A}. \]
Using the formal relation
\[ \delta_\tau(ab)=\delta_\tau(a)\tau(b)+a\delta_\tau(b),\]
it reduces to check  
\[  \delta_\tau(u^{-n})\tau(x) \equiv -nu^{-n}x \pmod{u\bargm_A} \]
By Lemma \ref{lem:difffktw}, the (mod $p$) operators $\delta_\tau$ and $\delta'=(\tau-1)/(uw)$ differ by an element in $1+uA_1$, thus we could switch to use $\delta'$ (which is convenient here).
Writing $x=u^ny$ with $y \in \bargm_A$, we are reduced to check
\[ \delta'(u^{-n})\tau(u^ny) \equiv -n y\pmod{u\bargm_A} \]
We have 
\[ \delta'(u^{-n}) =u^{-n}(\frac{\zetaflat^{-n}-1}{\zetaflat-1})  \in u^{-n}(-n+uA_1) \]
where the last inclusion follows from the binomial expansion of $\zetaflat^{-n} =(\zetaflat-1+1)^{-n}$ and the fact $\zetaflat-1 \in uA_1$.
We can conclude using the fact  that
\[\tau(u^ny) \in u^n(\zeta^\flat)^n(y+uw\bargm_A).\]
 \end{proof}

\begin{cor}[$p$-degree shrinking on $\wt{N}_{n,0}$]
\label{cor tilde gen eigen}
 Suppose $K$ is unramified and $T$ is crystalline. Then
 \[(\theta -n) (\wt{N}_{n,0}) \subset \wt{N}_{n,p} =\wt{N}_{n-p, 0} \]  
 As a consequence, $\theta$ is stable on $\wt{N}_{n,0}$ with all eigenvalues equal to $n$ (mod $p$).
\end{cor}
\begin{proof} 
 It follows from the following commutative diagram
\[
\begin{tikzcd}
N_{n,0,A} \arrow[d, "\delta_\tau-n"] \arrow[rr, two heads] &  & {\wt{N}_{n,0}} \arrow[d, "\theta-n"] \\
N_{n,p, A} \arrow[rr, two heads]                             &  & {\wt{N}_{n,p}}       
\end{tikzcd}
\]
Here: the   horizontal surjections are defined in Construction \ref{ConstructionfilphiA}, and the left vertical arrow is constructed in Proposition \ref{prop:pDS-A}. Commutativity is a consequence of Proposition \ref{lem-tau-1}.
\end{proof}

  The following  easy corollary follows from the fact that $N_{n,0} \subset \wt{N}_{n,0}$.
  
  \begin{cor} 
   \label{cor gen eigen}
Inside $\bargmht$, $N_{n,0}$ is contained in the generalized eigenspace of eigenvalue $n$ with respect to $\theta$-action.
  \end{cor} 
%  We have   \[ (\theta-n)^{\lceil \frac{n}{p} \rceil+1}(N_{n,0})=0\]

  Note that at this point, we do not even know if $\theta$ is stable on $N_{n,0}$. Although in Theorem \ref{prop: p griffiths}, we will prove the (stronger) fact that indeed $N_{n,0}=\wt{N}_{n,0}$! But this requires some extra argument (using the Frobenius operator, cf. the next section \S \ref{sec:mod-p-shape}): interestingly, it uses the (weaker) Corollary  \ref{cor gen eigen} as an \emph{input}.

 %%\newpage 
 \addtocontents{toc}{\ghblue{Applications: mod $p$ case}}

  \section{Shape of mod $p$ Frobenius}\label{sec:mod-p-shape}
  
Let $\bargm$ be an effective mod $p$ Breuil--Kisin module (that is not necessarily  from reduction of a semi-stable representation) with Frobenius height $\geq 0$. 
In \S \ref{subsec-equiv-cond}, we first provide several equivalent conditions for $\bargm$ to satisfy the unaligned mod $p$ Frobenius condition in Definition \ref{def: strong frob cond}\eqref{item unaligned mod p}; this subsection is purely  ``$\phi$-theoretic" as it does not involve Galois representations. Then in Theorem \ref{thm1 mod p frob}, we verify these equivalent conditions when $\bargm$ comes from reduction of a crystalline representation with $K$ unramified, making use of the Sen operator $\theta$.

  \subsection{Equivalent conditions} \label{subsec-equiv-cond}

 Let $\bargm$ be an effective mod $p$ Breuil--Kisin module (that is not necessarily  from reduction of a semi-stable representation), and consider it as an effective isogeny with respect to the triple $(k[[u]], u, \phi)$ (instead of $(k[[u]], u^e, \phi)$). So in particular, as we are in mod $p$ case, the discussion is still valid even if $K$ is ramified. (However, to avoid confusion, the readers could safely assume $K$ is unramified throughout this section).

%   In this subsection, we discuss various filtration structures related with $\bargm$.
 
 % The readers could safely assume $\bargm$ comes from reduction of a crystalline representation: but this is not necessary in this subsection. This general set up  has the benefit to \emph{highlight} the power of the extra ``symmetry" from Sen operators, once we assume $\bargm$ comes from reduction of a crystalline representation in  \S \ref{subsec: pGT}.

% In particular, recall in Definition \ref{def: strong frob cond}\eqref{item unaligned mod p}, we defined the  \emph{unaligned mod $p$ Frobenius condition} for $\bargm$, but with the running assumption in    \S \ref{sec: frob matrix} that $\gm$ (hence $\bargm$) comes from semi-stable representations; however, this is not relevant here:  indeed, this condition is not even related with Hodge--Tate weights.

%Let $\bargm$ be a mod $p$ Breuil--Kisin module (which might not be related with $\gk$-represenatations); note in mod $p$ case, finite $E$-height is equivalent to finite $u$-height, thus the discussions in this section remain valid even if $K$ is ramified.

\begin{notation} \label{notamodpbk}
Similar to Notation \ref{nota: mod p bk},  suppose for some basis $\vec{e}=(e_1, \cdots, e_d)$ of $\bargm$, we have
\[\phi(\vec{e})=  \vec{e} X \cdot \Lambda \cdot Y\] where $\Lambda=\diag(u^{a_1}, \cdots, u^{a_d})$ with $0 \leq a_1 \leq \cdots \leq a_d$ and $X, Y$ are invertible matrices. 
It is easy to see that $\fil^\bullet \bargmast$ always has an adapted basis in the sense that
 \[ \fil^n \bargmast = \oplus_{i=1}^d \overline{\gs} \cdot  u^{\max\{ n-a_i, 0\}} \cdot f_i^\ast \] 
where $ \vec{f^\ast} = \vec{e} X \Lambda$  forms a basis of $\bargmast$.
\end{notation}

\begin{defn}[A ``sub-Nygaard filtration"] \label{subnygadapted}
 Say the filtered $k[[u^p]]$-module $\fil^\bullet \phi(\bargm)$ has an adapted basis if
\[ \fil^n \phi(\bargm) =\oplus_{i=1}^d k[[u^p]]\cdot u^{p\cdot \max\{0, \lceil \frac{n-a_i}{p} \rceil \}} g_i \]
where $g_i$ is basis of $\phi(\bargm)$. 
\end{defn}

\begin{defn}[A ``sub-Hodge filtration"] \label{Definitionsub hodge fil}
Denote $$\phi(M_k): =\phi(\bargm)/u^p\phi(\bargm).$$
(Note   $\phi(\bargm) \into \bargmast$ induces an isomorphism of $k$-vector spaces $ \phi(M_k) \simeq \bargm_\dR$; we keep the notation $\phi(M_k)$ since we will define a possibly different ``sub-Hodge filtration" on it in the following). 
As in Definition \ref{defn fil phiM}, one can use the decreasing filtration  $\fil^i \phi(\bargm)$ to induce a decreasing filtration on $ \phi(M_k)$; call it the \emph{sub-Hodge filtration}. 
One computes that 
\[ \fil^n_\rmh \phi(M_k) = \fil^n \phi(\bargm)/ ( \fil^n \phi(\bargm) \cap \phi(u\bargm))= \fil^n \phi(\bargm)/ u^p \cdot \fil^{n-p} \phi(\bargm).  \]
It is then easy to check (as done  in \cite[Lem. 5.2.4]{Bar20a}) that   $\phi(\bargm) \into \bargmast$ induces an \emph{injective} map
\[ \fil^n_\rmh \phi(M_k) \into \fil^n \bargm_\dR,\]
justifying the terminology ``sub-Hodge filtration".
\end{defn}

\begin{construction}[The $i$-th piece of a ``sub-conjugate filtration"] \label{rem oplus not inj}
We have a decomposition (of $k[[u^p]]$-modules)
\begin{equation} \label{eqdecomposebargmast}
\oplus_{i=0}^{p-1} u^i\phi(\bargm)  =\bargmast.
\end{equation} 
Recall we can use the filtration on $\phi(\bargm)$ to define subspaces $N_{n,i} \subset \fil_n \bargm_\HT$, cf. Construction \ref{ConstructionfilphiA}.
In analogy with the decomposition \eqref{eqdecomposebargmast}, one can (forcibly) define   the direct sum map
\[  \oplus_{i=0}^{p-1} N_{n,i} \to \fil_n \bargm_\HT \]
which is however not necessarily injective. 
\end{construction}

The following is an extension of a lemma of Bartlett \cite[Lem. 5.2.5]{Bar20a}, by incorporating conjugate filtrations to the picture.
 
\begin{lemma}\label{lem: mod p Frob two fil}    Recall we consider $\bargm$ as an effective isogeny over $(k[[u]], u, \phi)$. 
The following statements are equivalent.
\begin{enumerate}
\item \label{itemf1} $\bargm$ satisfies the unaligned mod $p$ Frobenius condition in Definition \ref{def: strong frob cond}\eqref{item unaligned mod p}.

\item \label{itemf2}  $\fil^\bullet \phi(\bargm) $ has an adapted basis as in Definition \ref{subnygadapted}.

\item \label{itemf3} The injective filtered map (of $k[[u^p]]$-modules) $\oplus_{i=0}^{p-1} \fil^\bullet u^i\phi(\bargm) \into \fil^\bullet \bargm^\ast$ is bijective.

\item \label{itemf4} The  injective filtered map 
$\fil^\bullet_\rmh \phi(M_k) \into \fil^\bullet \bargm_\dR$ in Definition \ref{Definitionsub hodge fil}  is bijective.

\item  \label{itemf5} The direct sum map
$ \oplus_{i=0}^{p-1} N_{n,i} \to \fil_n \bargm_\HT $  in Construction \ref{rem oplus not inj}  is bijective for each $n$.

\item \label{itemf6} The direct sum map
$ \oplus_{i=0}^{p-1} N_{n,i} \to \fil_n \bargm_\HT $ is injective for each $n$.

%\item one more equiv condition. Note there is injective map \[ \fil^n \bargm \xrightarrow{\mathrm{bij}} \fil^n \phi(\bargm)  \hookrightarrow \fil^n \bargmast \] thus there is a sub-sub-Hodge fil. It should also be equal to Hodge fil. indeed, since already know strong Frob condition. Can write down $ \fil^n \bargm$ easily.
\end{enumerate}
\end{lemma}
\begin{proof} The equivalence $\eqref{itemf4} \Leftrightarrow \eqref{itemf1}$ is  \cite[Lem. 5.2.5]{Bar20a}.

For $\eqref{itemf1} \Rightarrow \eqref{itemf2}$: use Notation \ref{notamodpbk}, let $\vec{g}=\vec{e}X\Lambda$, then it is the desired $k[[u^p]]$-basis of $\phi(\gm)$ per Definition \ref{subnygadapted}.

For $\eqref{itemf2} \Rightarrow \eqref{itemf3}$:  use the fact
\[ \fil^\bullet u^i \phi(\bargm) =u^{i} \fil^{\bullet -i} \phi(\bargm)\] 

For $\eqref{itemf3} \Rightarrow \eqref{itemf4}$: use the following facts:
\[ \fil^n_\rmh \phi(M_k) = \fil^n \phi(\bargm)/ u^p \cdot \fil^{n-p} \phi(\bargm).  \]
\[\fil^n \bargm_\dR =  \fil^n \bargmast/u\fil^{n-1} \bargmast\]

For $\eqref{itemf3} \Rightarrow \eqref{itemf5}$: take graded.

For $\eqref{itemf5}  \Rightarrow \eqref{itemf6}$:  this is trivial.

Finally, we prove $\eqref{itemf6}  \Rightarrow \eqref{itemf4}$.
We feed the following data into Lemma \ref{lem zip fil} (with $h$  there set as $r_d$): 
\begin{itemize}
\item $P^n =\fil^n \bargm_\dR$
\item $Q_n=\fil_n \bargm_\HT$
\item $\wt{P}^n =\fil^n_\rmh \phi(M_k)$  
\item   $\wt{Q}_n=\oplus_{i=0}^{p-1} N_{n,i}$  
\end{itemize}
It suffices to verify the assumptions in  Lemma \ref{lem zip fil}. The conditions for $P^n, Q_n$   follow from the general fact Lemma  \ref{lem: matching graded dR and HT}. 
For $\wt{P}^n, \wt{Q}_n$, we need to check 
 \[ \dim_k \left( (\oplus_{i=0}^{p-1} N_{n,i})/(\oplus_{i=0}^{p-1} N_{n-1,i}) \right) =\dim_k  \left( \fil^n_\rmh \phi(M_k)/\fil^{n+1}_\rmh \phi(M_k) \right) \]
 Here on LHS, we are taking quotient for subspaces inside $\bargmht$ (which is doable by assumption in Condition \eqref{itemf6}); thus the LHS is precisely $\dim_k N_{n,0}/N_{n-p,0}$ and we can conclude by Lemma \ref{lem match grade phiM}. (A hidden fact is that $\wt{Q}_{r_d}=\bargm_\HT$: but this is automatic  for dimension reasons, since we must have $\dim \wt{Q}_{r_d} =\dim \wt{P}^0=\dim \bargm_\HT$). 
\end{proof}

 \begin{lemma} \label{lem zip fil}
Let $P$ and $Q$ be two finite dimensional  vector spaces (over a field) with the same dimension, such that the following conditions are satisfied.
\begin{enumerate}
\item Suppose there are  filtrations concentrated in the range $[0,h]$ (for some $h \geq 0$):
\begin{align*}
 \cdots  P = P &=      P^0 \supset   P^1 \supset \cdots     P^h \supset  0 =  0 \cdots \\
\cdots  0= 0 & \subset       Q_0 \subset    Q_1 \subset \cdots    Q_h= Q  = Q \cdots
\end{align*}
such that the gradeds of these two filtrations have same dimension: that is, for each $i \in \mathbb{Z}$, 
\[ \dim P^i/P^{i+1} =\dim Q_i/Q_{i-1} \]
\item Suppose furthermore there are two \emph{sub-filtrations}  
 \begin{align*}
\cdots  P = P  & =     \wt{P}^0 \supset   \wt{P}^1 \supset \cdots     \wt{P}^h \supset  0 =  0 \cdots \\
 \cdots  0= 0  & \subset       \wt{Q}_0 \subset    \wt{Q}_1 \subset \cdots    \wt{Q}_h= Q  =Q \cdots 
\end{align*}
where for each $i$,
 \[\wt{P}^i \subset P^i, \quad \wt{Q}_i  \subset Q_i \] 
 and such that  the gradeds of these two filtrations also have the same dimension: that is, for each $i$, 
\[ \dim \wt{P}^i/\wt{P}^{i+1} =\dim \wt{Q}_i/\wt{Q}_{i-1} \]
(which \emph{a priori} is not necessarily equal to $\dim P^i/P^{i+1}$).
\end{enumerate}
 Then the two sub-filtrations coincide with the original filtrations. That is, for each $i$,  
\[ \wt{P}^i = P^i, \quad \wt{Q}_i  = Q_i. \]
\end{lemma}
\begin{proof}
The proof is elementary. 
Consider graded at $i=0$, we have
\[ \dim \wt{Q}_0 = \dim \wt{P}^0/\wt{P}^1 = \dim  P/\wt{P}^1 \geq \dim  P/P^1 =\dim Q_0 \]
thus we must have equality, and thus $\wt{Q}_0 =Q_0$ and     $\wt{P}^1=P^1$. An obvious induction argument would finish the proof.

Here is an alternative (slightly) more conceptual proof. Define   ``Hodge numbers" of these filtrations
\[ t(P^\bullet) =\sum_{i \in \bbz} i\dim P^i/P^{i+1}, \quad t(Q_\bullet) =\sum_{i \in \bbz} i\dim Q_i/Q_{i-1}\]
and similarly for $\wt{P}^\bullet, \wt{Q}_\bullet$.
Then \cite[Lemma 3.3.1]{Bar20a} (resp. its dual form for increasing filtrations)  implies
\[ t(P^\bullet) \geq t(\wt{P}^\bullet), \quad \text{resp. } t(Q_\bullet) \leq t(\wt{Q}_\bullet)\]
But we also have $t(P^\bullet)=t(Q_\bullet)$ and $t(\wt{P}^\bullet)=t(\wt{Q}_\bullet)$; thus indeed all these ``Hodge numbers"  coincide. Then \cite[Lemma 3.3.1]{Bar20a} (resp. its dual form)  implies $P^\bullet=\wt{P}^\bullet$ resp. $Q_\bullet=\wt{Q}_\bullet$.
\end{proof}

\subsection{Verifying mod $p$ Frobenius condition} \label{subsec:strong-mod-p-Frob}

The following theorem is first due  to  ongoing work of Bhatt--Gee--Kisin \cite{BGK}, using a stacky approach, cf.~ Remark \ref{rem y matrix}. 
We expect strong relation between their approach and ours; but as far as we learn from Bhatt--Gee--Kisin, there does not seem to be obvious translations at this point. In particular, our proof crucially uses the $p$-degree shrinking property  Corollary  \ref{cor tilde gen eigen} and Corollary  \ref{cor gen eigen}.

\begin{theorem}\label{thm1 mod p frob}
Suppose $K$ is unramified and $T$ is crystalline. Then $\bargm$ satisfies the unaligned mod $p$ Frobenius condition  in Definition \ref{def: strong frob cond}\eqref{item unaligned mod p}.
\end{theorem}
\begin{proof}
Corollary \ref{cor gen eigen} implies the direct sum map
\[ \oplus_{i=0}^{p-1}  {N}_{n,i} \to \fil_n \bargm_\HT \]
factors through the generalized eigenspace decomposition with respect to $\theta$-action on $\fil_n \bargm_\HT$, and thus is injective. This verifies Item \eqref{itemf6} in Lemma \ref{lem: mod p Frob two fil}.
\end{proof}

%%\newpage
\section{$p$-degree shrinking: mod $p$ Sen operator} \label{sec:pDS vs pGT}
In this section, we always suppose $K$ is unramified and $T$ is crystalline.
In this  case, we show that the mod $p$ Sen operator satisfies a \emph{$p$-degree shrinking} on a certain increasing filtration, as well as a \emph{$p$-Griffiths transversality} on a certain  decreasing filtration. 
These are relatively easy consequences of results in previous two sections. We isolate these results in this section for convenience of discussions and comparisons. 
The $p$-Griffiths transversality is already known by work of Bhatt--Lurie \cite{Bha22} via a stacky approach. However, the $p$-degree shrinking result, as far as we understand, is a completely new structure; we expect it will have  further applications in the future.

\subsection{$p$-degree shrinking}
\begin{theorem}[$p$-degree shrinking] \label{prop: p griffiths}
Suppose $K$ is unramified and $T$ is crystalline. Recall $\theta$ is defined in Construction \ref{ConstructiontwistA}, and $N_{n,i}$ is defined in Construction \ref{ConstructionfilphiA}.
We have
\begin{enumerate}
\item  The operator $\theta-(n-i)$ satisfies a $p$-degree shrinking  on $N_{n,i}=N_{n-i,0}$ in the sense that
\[ (\theta-(n-i))(N_{n,i}) \subset N_{n-p,i}. \]
%\ghblue{note   here} \[ N_{n-p,i} =N_{n-p-1,i-1}=\cdots=N_{n-p-i,0}\] 

\item   $N_{n,i}$ is exactly the generalized eigenspace of $\fil_n \bargm_\HT$ with $\theta$-eigenvalue $n-i$, and we have an equality
\[ \oplus_{i=0}^{p-1} N_{n,i} =\fil_n \bargm_\HT,\]
which is exactly the generalized eigenspace decomposition with respect to $\theta$-action on $\fil_n \bargm_\HT$.
\end{enumerate} 
\end{theorem}   
\begin{proof}
For eigenvalue reasons (as already used in proof of Theorem \ref{thm1 mod p frob}), we have inclusions 
\[ \oplus_{i=0}^{p-1} N_{n,i} \subset \oplus_{i=0}^{p-1} \wt{N}_{n,i} \subset \fil_n \bargm_\HT.\]
But via Lemma \ref{lem: mod p Frob two fil}, we already know
\[ \oplus_{i=0}^{p-1} N_{n,i} =\fil_n \bargm_\HT\]
Thus we must have
\[ N_{n,i} = \wt{N}_{n,i} \]
Then we can conclude using Proposition \ref{prop-pGT-A-level}.
\end{proof}

\begin{remark}  \label{rem: pGT}    The $1$-degree shrinking of $\theta-n$ on $\fil_n \bargm_\HT$ (Theorem \ref{thm: mod p Sen op fil}) and the $p$-degree shrinking of $\theta-(n-i)$ on $N_{n,i}$ (Theorem \ref{prop: p griffiths}) are related by the following. Repeatedly using $1$-degree shrinking, we have a composite
\[ \prod_{j=p-1}^0 (\theta-(n-j)) : \fil_n \bargm_\HT \to  \fil_{n-1} \bargm_\HT  \to \cdots \to \fil_{n-p} \bargm_\HT \]
Note this composite operator is nothing but $\theta^p-\theta$ (which is independent of $n$). 
Now, fix one  $0\leq i \leq p-1$. For each $j \neq i \pmod{p}$, the action of $(\theta-(n-j))$ on $N_{n,i}$ is \emph{invertible} (as the eigenvalues are invertible); thus on $N_{n,i}$, the actions of $\theta^p-\theta$ and $(\theta-(n-i))$ differ by an isomorphism. In particular, the $p$-degree shrinking  can be re-written as:
\[ (\theta^p-\theta)(N_{n,i}) \subset N_{n-p,i} \] 
\end{remark}

\subsection{$p$-Griffiths transversality} 
 
\begin{construction} \label{consfinmnhthn}
Suppose $K$ is unramified and $T$ is crystalline.
\begin{enumerate}
\item   Apply Definition \ref{defn fil phiM} to the module $\bargm$ with respect to the triple $(\gs_1, u, \phi)$, we can construct a decreasing filtration $\fil^\bullet_\rmh \bargmht$ such that
\[ \fil^n_\rmh \bargm_\HT =\fil^n \bargm/u\fil^{n-p} \bargm\] 
As discussed in Definition \ref{defn fil phiM}, we have a bijection
\[ \fil^\bullet_\rmh \bargmht \xrightarrow{\phi, \simeq} \fil^\bullet_\rmh \phi(\bargm/u\bargm)\]
which in turn is isomorphic to $\fil^\bullet \bargm_\dR$ as proved by  Theorem \ref{thm1 mod p frob}.

\item  Apply Definition \ref{defn fil phiM} to the module $\bargm_A$ with respect to the triple $(A_1, u, \phi)$, we obtain $\fil^\bullet_\rmh \bargm_{A, \HT}$. As discussed in Construction \ref{ConstructionfilphiA}, $\fil^\bullet \bargm \into \fil^\bullet \bargm_A$ is strict; thus we have an injective map
\[ \fil^\bullet_\rmh \bargmht \into \fil^\bullet_\rmh \bargm_{A, \HT}\]

\item  As noted in Construction \ref{ConstructionfilphiA}, we have flat base change
\[   \bargm_\HT \otimes_{k} A_1/u A_1 \simeq   \bargm_{A,\HT} \]
Thus there is a projection map by reduction modulo $\fkm_{A_1}$:
\[ \bargm_{A,\HT} \onto   \bargm_\HT\]
Denote the image of $\fil^n_\rmh \bargm_{A, \HT}$ under the projection as $H^n$; then we have
\[ \fil^n_\rmh \bargmht \subset H^n\]
\end{enumerate}
\end{construction}

\begin{lemma} \label{lem-proj-hodgeht}
Suppose $K$ is unramified and $T$ is crystalline. 
The inclusion in Construction \ref{consfinmnhthn}
\[\fil^n_\rmh \bargmht \subset H^n\]
is an equality. 
\end{lemma}
\begin{proof} 
As both $\fil^\bullet_\rmh \bargmht$ and $H^\bullet$ form (effective) decreasing filtrations on the vector space $\bargmht$, it suffices to show their gradeds have the same dimension. 
By Lemma \ref{lem match grade phiM}
\[ \gr^n_\mathrm{H} \bargm_\HT \simeq N_{n,0}/N_{n-p,0}\]
Similarly, $H^n/H^{n+1}$ is isomorphic to
\[  \gr^n_\mathrm{H} \bargm_{A, \HT} \otimes_A k \simeq (N_{n,0, A}/N_{n-p,0, A}) \otimes_A k \simeq  (N_{n,0, A}\otimes_A k)/ (N_{n-p,0, A}\otimes_A k) \simeq N_{n,0}/N_{n-p,0}\]
Here the first isomorphism follows from Lemma \ref{lem match grade phiM} applied to $\bargm_A$, and the last isomorphism follows from Theorem \ref{prop: p griffiths} (saying that $\wt{N}_{n,0}=N_{n,0}$).
\end{proof}

\begin{theorem}[\textnormal{$p$-Griffiths transversality}]   \label{thm:BL-p-GT}
Suppose $K$ is unramified and $T$ is crystalline. 
The operator $\theta$   satisfies a $p$-Griffiths transversality on $\fil_\rmh^\bullet \bargmht$. That is,  for each $n$, 
\[ \theta( \fil^n_{\mathrm{H}} \bargm_\HT) \subset  \fil^{n-p}_{\mathrm{H}} \bargm_\HT\]
\end{theorem}
\begin{proof}   
This follows from the commutative diagram
  \[
\begin{tikzcd}
\fil^n \bargm_A \arrow[d, "\delta_\tau "] \arrow[r, two heads] & {\fil^n_{\mathrm{H}}\bargm_{A,\HT}} \arrow[d, "\delta_\tau"] \arrow[r, two heads] & \fil^n_{\mathrm{H}}\bargm_{\HT} \arrow[d, "\theta"] \\
\fil^{n-p} \bargm_A \arrow[r, two heads]                       & {\fil^{n-p}_{\mathrm{H}}\bargm_{A,\HT}} \arrow[r, two heads]                      & \fil^{n-p}_{\mathrm{H}}\bargm_{\HT}                
\end{tikzcd}
\] 
Here, the first vertical arrow is constructed in Proposition \ref{prop-pGT-A-level}; the second vertical arrow is induced by the first. The left horizontal arrows follow by definition; the right horizontal surjection, which is the mod $\fkm_{A_1}$ reduction map
\[ \fil^\bullet_H \bargm_{A, \HT} \to \fil^\bullet_H \bargm_{\HT}\]
is well-defined by Lemma \ref{lem-proj-hodgeht}. (Caution: unlike Corollary  \ref{cor tilde gen eigen}, there is no ``shifting by $-n$" here, as there is no multiplication by $u^{-n}$ in the whole process).
 \end{proof}
 
 \begin{remark} \label{rem:specialcaseBL}
 The $p$-Griffiths transversality that we proved in Theorem \ref{thm:BL-p-GT} can be regarded as a  special case  of a result of Bhatt--Lurie \cite{Bha22} (written for $W(k)=\zp$ there).
 \begin{enumerate}
 \item \label{item1redsyn}  As summarized in the proof of \cite[Cor 6.5.14]{Bha22}, a perfect complex $E$ (not necessarily related with a crystalline representation) on $W(k)_{\mathrm{red}}^{\syn}$ gives rise to a quadruple $(V, F^\ast, F_\ast, \Theta)$ where $V$ is a perfect $k$-complex, with $F^\ast$ resp. $F_\ast$ a decreasing resp. increasing filtration, and $\Theta: V \to V$ is an operator satisfying 
\[ \Theta: F^\ast V \to F^{\ast -p} V \]
\[ \Theta -\ast: F_\ast V \to F_{\ast-1} V\]
(and other ``nilpotency" conditions).

\item In our set up, the crystalline representation gives rise to a reflexive $F$-gauge \cite[\S 6.6]{Bha22} on the entire $W(k)^\syn$, hence restricts to a coherent sheaf on  $W(k)_{\mathrm{red}}^{\syn}$. The corresponding quadruple object is precisely
\[ (\bargmht, \fil^\bullet_\rmh \bargmht, \fil_\bullet^\rmconj \bargmht, \theta) \]  
 \end{enumerate}
 \end{remark}

\begin{remark} \label{rem:difference}
We clarify the differences between \emph{$p$-degree shrinking} and \emph{$p$-Griffiths transversality}.
\begin{enumerate}
\item Although these two properties can be ``unified" in the $A_1$-level (Proposition \ref{prop-pGT-A-level}), it is not clear if they can be ``directly" related. In particular, even assuming the result (of \cite{Bha22}) discussed in Remark \ref{rem:specialcaseBL}\eqref{item1redsyn}, we do not see how that will imply  Theorem \ref{prop: p griffiths}.

%We do mention again: the  $p$-degree shrinking  result in Theorem \ref{prop: p griffiths}, (as far as we understand), does not seem to be contained in \cite{Bha22}. 
 
\item Unlike the $p$-degree shrinking in Theorem \ref{prop: p griffiths} which gives useful information on eigenvalues of $\theta$, the $p$-Griffiths transversality in Theorem \ref{thm:BL-p-GT} does not seem to imply further information on the eigenvalues. Indeed,  when the Hodge--Tate weight of $T$ is contained in  $[0, p-1]$, then Theorem \ref{thm:BL-p-GT} provides no information on filtrations (other than saying $\theta$ is stable on $\bargmht$).
\end{enumerate}
\end{remark}

%%\newpage     
  
%\section*{==END end=}
   %%%\newpage 
\bibliographystyle{alpha}
% \bibliography{archive/Gaobib}

% \begin{thebibliography}{1}
 %  \end{thebibliography}
\end{document}